\documentclass[reqno]{amsart}
\usepackage{amsfonts,amsmath,amssymb,color}

\numberwithin{equation}{section}

\usepackage[kerning=true]{microtype}

\usepackage[top = 4cm, left = 3.1cm, right = 3.1cm]{geometry}

\usepackage{amsthm,leqno}
\usepackage{hyperref}
\usepackage{bigints}
\usepackage{dsfont}

\usepackage{amsfonts,amsmath,amssymb,enumerate}
\usepackage{graphicx}
\usepackage{tikz}
\usetikzlibrary{decorations.markings}
\usetikzlibrary{plotmarks}
\usetikzlibrary{patterns}

\newtheorem{theo}{Theorem}[section]
\newtheorem{prop}[theo]{Proposition}

\newtheorem{lemme}[theo]{Lemma}
\newtheorem{corol}[theo]{Corollary}
\newtheorem{claim}[theo]{Claim}

\theoremstyle{remark}

\newcommand{\be}{\begin{equation*}}
\newcommand{\ee}{\end{equation*}}
\newcommand{\ben}{\begin{equation}}
\newcommand{\een}{\end{equation}}
\newcommand{\begincal}{\begin{eqnarray*}}
\newcommand{\fincal}{\end{eqnarray*}}
\newcommand{\bal}{\begin{aligned}}
\newcommand{\eal}{\end{aligned}}

\newcommand{\pui}{\frac{n-2}{2}}

\newcommand{\tvk}{\tilde{w}_k}

\newcommand{\Lg}{\mathcal{L}_g}
\newcommand{\Lx}{\mathcal{L}_\xi}

\newcommand{\Dg}{\overrightarrow{\triangle}_g }

\newcommand{\Lax}{\Lambda_{\xi}}

\newcommand{\lki}{\lambda_{k}^i}
\newcommand{\lkja}{\lambda_{k,j}^1}
\newcommand{\lkjb}{\lambda_{k,j}^2}

\newcommand {\dk}{\delta_k}
\newcommand {\tdk}{\tilde{\delta}_k}

\newcommand{\tk}{\theta_k}

\newcommand{\tpk}{\tilde{\phi}_k}
\newcommand{\tpz}{\tilde{\phi}_0}

\newcommand{\mk}{\mu_k}
\newcommand{\xk}{\xi_k}
\newcommand{\exk}{\textrm{exp}_{\xi_k}^{g_{\xi_k}}}
\newcommand{\vek}{\varepsilon_k}

\newcommand{\RR}{\mathbb{R}}

\newcommand{\ve}{\varepsilon}
\newcommand{\vp}{\varphi}

\numberwithin{equation}{section}

\begin{document}

\title[]{A pointwise finite-dimensional reduction method for a fully coupled system of Einstein-Lichnerowicz type.}

\author{Bruno Premoselli}
 
\address{Bruno Premoselli, Universit\'e Libre de Bruxelles, Service de G\'eom\'etrie Diff\'erentielle,
CP 218,
Boulevard du Triomphe,
B-1050 Bruxelles,
Belgique.}
\email{bruno.premoselli@ulb.ac.be}

\begin{abstract} 
We construct non-compactness examples for the fully coupled Einstein-Lichnerowicz constraint system in the focusing case. The construction is obtained by combining pointwise a priori asymptotic analysis techniques, finite-dimensional reductions and a fixed-point argument.

\medskip

\noindent 
More precisely, we perform a fixed-point procedure on the remainders of the expected blow-up decomposition. The argument consists of an involved finite-dimensional reduction coupled with a ping-pong method. To overcome the non-variational structure of the system, we work with remainders which belong to strong function spaces and not merely to energy spaces. Performing both the ping-pong argument and the finite-dimensional reduction therefore heavily relies on the \emph{a priori} pointwise asymptotic techniques of the $C^0$ theory, as developed in Druet-Hebey-Robert \cite{DruetHebeyRobert}.
\end{abstract}

\maketitle  
\section{Introduction}

\subsection{Statement of the results}

\noindent Let $(M,g)$ be a closed Riemannian manifold of dimension $n \ge 6$, where closed means here compact without boundary. We investigate non-compactness issues in strong spaces for the set of positive solutions of the Einstein-Lichnerowicz system of equations in $M$:
\ben \label{intro1}
\left \{ \bal
 \triangle_g u + h u & = f u^{2^*-1} + \frac{\left| \Lg T + \sigma \right|_g^2 + \pi^2}{u^{2^*+1}} \\
 \Dg T & = u^{2^*} X + Y.
 \eal \right.
 \een 
The unknowns of \eqref{intro1} are $u$, a smooth positive function in $M$, and $T$, a smooth field of $1$-forms in $M$. In \eqref{intro1} we denoted by
  $\Lg T $ the conformal Killing derivative of $T$, whose expression in coordinates is:
\ben \label{kilkil}
 \Lg T_{ij} = \nabla_i T_j + \nabla_j T_i - \frac{2}{n} \big(\textrm{div}_g T \big) g_{ij},
 \een 
and we denoted by  $\Dg$ the Lam\'e operator acting on $1$-forms as $\Dg T = - \textrm{div}_g  \big( \Lg T \big)$. In \eqref{kilkil}, $\nabla$ stands for the Levi-Civita connection of the metric. The second equation of \eqref{intro1} is in particular a $1$-form equation. Also, in \eqref{intro1}, $\triangle_g = - \textrm{div}_g (\nabla \cdot) $ is the Laplace-Beltrami operator, $h,f,\pi$ are smooth functions in $M$, $\sigma$ is a smooth field of $2$-forms with $\textrm{tr}_g \sigma = 0$ and $\textrm{div}_g \sigma = 0$ and $X$ and $Y$ are smooth fields of $1$-forms in $M$. The notation $2^* = \frac{2n}{n-2}$ denotes the critical exponent for the embedding of the Sobolev space $H^1(M)$ into Lebesgue spaces. We also assume that:
\ben \label{focusingcase}
f > 0 \textrm{ in } M,
\een
and that $\triangle_g + h $ is coercive (which is necessary in view of \eqref{focusingcase}).

\medskip

\noindent The Einstein-Lichnerowicz system \eqref{intro1} arises in the initial-value problem in Mathematical General Relativity, as a conformal formulation of the constraint equations; see Bartnik-Isenberg \cite{BarIse} for the physical derivation of \eqref{intro1}. In the relativistic physical case, the coefficients of \eqref{intro1} express in terms of the background physics data of the problem: for instance, in a scalar-field setting, they write as:
\begin{equation} \label{coefficients}
\begin{aligned}
& h  = c_n \left( S_g - |\nabla \psi|_g^2 \right) ~~,~~f  =  c_n \left( 2 V(\psi) - \frac{n-1}{n} \tau^2 \right)  ~~,~~X = - \frac{n-1}{n} \nabla \tau ~~,~~ Y = - \pi \nabla \psi ,  \\
\end{aligned} \end{equation}
where $c_n = \frac{n-2}{4(n-1)}$. In \eqref{coefficients}, $V : \RR \to \RR$ is a potential, $\psi, \pi: M \to \RR$ are scalar-field components, $\tau: M \to \RR$ is the mean extrinsic curvature and $S_g$ is the scalar curvature of $g$. In view of \eqref{coefficients}, the focusing assumption \eqref{focusingcase}  naturally arises as one considers non-trivial non-gravitational physics data. 

\medskip

\noindent In the focusing case \eqref{focusingcase}, existence results for \eqref{intro1} were first obtained in Hebey-Pacard-Pollack \cite{HePaPo}, under the assumption $X \equiv 0$ that decouples \eqref{intro1}; multiplicity results were then obtained in 
Ma-Wei \cite{MaWei}, Premoselli \cite{Premoselli2} and, for specific physical cases, in Holst-Meier \cite{HolstMeier}, Chru\'sciel-Gicquaud \cite{ChruscielGicquaud} and Bizo\'n-Pletka-Simon \cite{BizonPletkaSimon}. For the fully coupled system -- when $X \not \equiv 0$ -- existence results are in Premoselli \cite{Premoselli1} and  Gicquaud-Nguyen \cite{GicquaudNguyen}. 

\medskip

\noindent A more satisfactory picture of system \eqref{intro1} is obtained through the analysis of its stability issues. Following the general definition given in Druet \cite{DruetENSAIOS} (see also Hebey \cite{HebeyZLAM}), we say that system \eqref{intro1} is \emph{stable} if, for any sequence $(h_k, f_k, \pi_k, \sigma_k, X_k, Y_k)_k$ of coefficients converging towards $(h,f,\pi, \sigma, X,Y)$ as $k \to + \infty$ in some strong topology (to be precised),
and for any sequence $(u_k, T_k)_k$ of solutions of:
\ben \label{eqperturbe}
\left \{ \bal 
 \triangle_g u_k + h_k u_k & = f_k {u_k}^{2^*-1} + \frac{\left| \Lg T_k + \sigma_k \right|_g^2 + \pi_k^2}{u_k^{2^*+1}} \\
 \Dg T_k & = u_k^{2^*} X_k + Y_k,
\eal \right.
\een
with $u_k > 0$, there holds, up to a subsequence and up to elements in the kernel of $\Lg$, that $(u_k, T_k)_k$ converges to some positive solution $(u_0,T_0)$ of \eqref{intro1} in $C^{1,\eta}(M)$ for all $0 < \eta < 1$. One defines the \emph{compactness} of \eqref{intro1} analogously, by taking a constant sequence of coefficients $(h_k, f_k, \pi_k, \sigma_k, X_k, Y_k)_k = (h,f,\pi, \sigma, X,Y)$.  In the focusing case \eqref{focusingcase} and for the decoupled system (when $X \equiv 0$), stability results were first obtained in Druet-Hebey \cite{DruetHebeyEL}, and then in Hebey-Veronelli \cite{HebeyVeronelli} and Premoselli \cite{Premoselli2}. For the fully coupled case $X \not \equiv 0$, the stability of \eqref{intro1} has been investigated in Druet-Premoselli \cite{DruetPremoselli} and Premoselli \cite{Premoselli4} on locally conformally flat manifolds. It is shown in \cite{Premoselli4} that for any $n \ge 6$, equation \eqref{intro1} is stable in the $C^2$ topology as soon as $\pi \not \equiv 0$ and:
\ben \label{condstabilite}
\textrm{either } \quad \bal X \textrm{ and } \nabla f \textrm{ have no common zero in } M \quad \textrm{ or } \quad h - c_n S_g + \frac{(n-2)(n-4)}{8(n-1)} \frac{\triangle_g f}{f} < 0,
\eal
\een
where $S_g$ denotes the scalar curvature of $g$. Stability actually holds for a slightly weaker topology, see \cite{DruetPremoselli} and \cite{Premoselli4}. The stability of the decoupled version of \eqref{intro1} was used in \cite{Premoselli2} to describe multiplicity issues in small dimensions. In the physical case, where the coefficients are given by \eqref{coefficients}, the stability and the instability of \eqref{intro1} relate to the relevance of the conformal method from which it arises. We refer to the discussion in \cite{Premoselli4}. In the opposite direction, the only known instability result for \eqref{intro1} has been obtained by Premoselli-Wei \cite{PremoselliWei} in the decoupled case $X \equiv 0$, and for the physical choice of coefficients given by \eqref{coefficients}. The motivation in \cite{PremoselliWei} was to obtain explicit physical counter-examples to the stability of the conformal method. 

\medskip

\noindent In this work, we address the instability behavior of system \eqref{intro1} from a different perspective: we leave the physical case treated in \cite{PremoselliWei} aside but consider instead the fully coupled case $X \not \equiv 0$ of \eqref{intro1}. As a direct counterpart of the results in \cite{DruetPremoselli} and \cite{Premoselli4}, we obtain non-compactness results for system \eqref{intro1} in any dimension $n \ge 6$ as soon as \eqref{condstabilite} is not satisfied. Our main result states as follows: 
\begin{theo} \label{thprincipal}
Let $(M,g)$ be a closed Riemannian manifold of dimension $n \ge 6$ of positive Yamabe type and possessing no non-trivial conformal Killing fields. There exist regular coefficients $(h,f,\pi,\sigma,X,Y)$, with $\triangle_g + h$ coercive, $f > 0$, $\pi \not \equiv 0$ and $X \not \equiv 0$ such that the associated system of equations \eqref{intro1}:
\[
\left \{ \bal
 \triangle_g u + h u & = f u^{2^*-1} + \frac{\left| \Lg T + \sigma \right|_g^2 + \pi^2}{u^{2^*+1}} \\
 \Dg T & = u^{2^*} X + Y
 \eal \right.
\]
possesses a blowing-up sequence of solutions $(u_k, T_k)_k$, that is satisfying $\Vert u_k \Vert_{L^\infty(M)} \to + \infty$ and $ \Vert  \Lg T_k \Vert_{L^\infty(M)}  \to + \infty$ as $k \to +\infty$. Also, the $u_k$ are positive, possess a single blow-up point and blow-up with a non-zero limit profile.
\end{theo}
\noindent Theorem \ref{thprincipal} is a \emph{non-compactness} result, and therefore also an instability result. The explicit expression of the coefficients $(h,f,\pi,\sigma,X,Y)$, their regularity as well as the generality of the construction are discussed in Section \ref{notations} below. For $n \ge 7$, these coefficients do not satisfy \eqref{condstabilite}, thus establishing its sharpness in the coupled case. A conformal Killing field is a field of $1$-forms $X$ in $M$ satisfying $\Lg X = 0$. The assumption that $(M,g)$ possesses no non-trivial conformal Killing fields is generic as shown in Beig-Chru\'sciel-Schoen \cite{BeChSc} and implies, since $M$ is closed, that $\Dg$ has no kernel. 

\medskip
\noindent A striking consequence of Theorem \ref{thprincipal} is the existence of an infinite number of solutions of \eqref{intro1}:
\begin{corol} \label{corolmulti}
Let $(M,g)$ be a closed Riemannian manifold of dimension $n \ge 6$ of positive Yamabe type and possessing no non-trivial conformal Killing fields. There exist regular coefficients $(h,f,\pi,\sigma,X,Y)$, with $\triangle_g + h$ coercive, $f > 0$, $\pi \not \equiv 0$ and $X \not \equiv 0$ 
 such that the system \eqref{intro1} possesses an infinite number of solutions.
\end{corol}

\medskip

\noindent In the fully coupled case case $X \not \equiv 0$ that we investigate here, system \eqref{intro1} is of a different nature than in the physical decoupled case of \cite{PremoselliWei}. In particular, \eqref{intro1}  possesses no variational structure, exhibits supercritical nonlinearities and is not well-posed for $(u,W)$ in the energy space $H^1(M)$. To prove Theorem \ref{thprincipal}  we thus develop a pointwise finite-dimensional reduction method, which combines a priori pointwise asymptotic analysis techniques with a variational Lyapunov-Schmidt-type approach in order to perform an involved ping-pong method in strong spaces. We explain in detail the method and the strategy of proof in the next sub-section.

\subsection{Strategy of the proof of Theorem \ref{thprincipal}.}

\medskip

\noindent   In the fully coupled case $X \not \equiv 0$, system \eqref{intro1} exhibits a strong nonlinear coupling via the $(\left| \Lg T + \sigma \right|_g^2 + \pi^2) u^{-2^*-1}$ term, as well as a super-critical nonlinearity in the right-hand side of the $1$-form equation. Hence \eqref{intro1} does not possess a well-posed variational formulation in $H^1(M)$. If one only assumes that $u \in H^1(M)$, the right-hand side of the $1$-form equation is merely bounded in $L^1(M)$ which yields no integral control on $(\left| \Lg T + \sigma \right|_g^2 + \pi^2) u^{-2^*-1}$ whatsoever. This is a serious obstacle to combining a ping-pong approach with a finite-dimensional reduction method for the scalar equation of \eqref{intro1}. 

\medskip

\noindent To prove Theorem \ref{thprincipal} we therefore work in strong topologies. We construct a blowing-up sequence of solutions $(u_k,T_k)_k$ of \eqref{intro1} whose scalar component writes as: 
\ben \label{blowtete}
 u_k = u_{k,t,p} = W_{k,t,p} + u + \vp_{k,t,p},
 \een
but where this decomposition holds in $C^0(M)$ and not in $H^1(M)$. Here, as usual, $W_{k,t,p}$ denotes a bubbling profile depending on some parameters $(t,p) \in \RR^{n+1}$ whose expression is given in Sections \ref{notations} and \ref{DL} below. The positive function $u$ is the scalar component of a solution $(u,T)$ of \eqref{intro1} satisfying a non-degeneracy assumption (detailed in \eqref{uetastableX} below). And the remainder term, $\vp_{k,t,p}$, is chosen to be small with respect to $W_{k,t,p}$ and $u$ in the following pointwise sense: 
\ben \label{apriorireste}
|\vp_{k,t,p}| \le \vek \big( W_{k,t,p} + u \big) \quad \textrm{ pointwise in } M,
\een
for some suitably chosen sequence $(\vek)_k$ of positive numbers converging to zero. 

\medskip
\noindent Decomposition \eqref{blowtete} is of course reminiscent of Struwe's $H^1$ a priori result \cite{Struwe}. 
But the motivation for the choice of \eqref{blowtete} comes from the work of  Druet-Hebey-Robert \cite{DruetHebeyRobert}, where Struwe's decomposition was shown to actually hold true in $C^0(M)$ for $H^1$-bounded solutions of critical stationary Schr\"odinger equations. 
 The blow-up analysis of system \eqref{intro1} performed in \cite{DruetPremoselli} and \cite{Premoselli4} confirms the choice of \eqref{blowtete} \emph{pointwise}, at least at a local scale.

 \medskip
 
 \noindent The construction of a sequence $(u_k, T_k)_k$ satisfying \eqref{blowtete} goes through an involved ping-pong procedure which relies upon a finite-dimensional reduction made possible by asymptotic analysis techniques. Section \ref{notations} introduces the definition of the coefficients $(h,f,\pi,\sigma,X,Y)$ and of the bubbling profiles considered and Section \ref{technicalresults} gathers several technical results used throughout the article. The structure of the proof and the organization of the remaining sections of the article are as follows:

\medskip

\noindent \large{\textbf{Section \ref{theorieH1}: Semi-decoupling and $H^1$ reduction.}} \normalsize
Let $\vp$ be a remainder satisfying \eqref{apriorireste}. Since $\Dg$ has no Kernel by assumption, there exists a unique $1$-form $T_{k,t,p}$ in $M$ satisfying $\Dg T_{k,t,p} = \left( W_{k,t,p} + u + \vp \right)^{2^*} X + Y$.  The $C^0$ bound \eqref{apriorireste} on $\vp$ yields explicit pointwise bounds on $\Lg T_{k,t,p}$, which happens to blow-up too fast for standard $H^1$ finite-dimensional reduction procedures to apply to the scalar equation of \eqref{intro1} with $\Lg T_{k,t,p}$ seen as a coefficient. We therefore artificially discard the $|\Lg T_k + \sigma|_g^2$ term into a source term and consider instead the following equation:
\ben \label{schemastep11} 
\bal
 \triangle_g u + h u  = f u^{2^*-1}  + \frac{|\Lg T + \sigma|_g^2 + \pi^2}{u^{2^*+1}}  \underbrace{+ \left( \frac{ |\Lg T_k + \sigma|_g^2 - |\Lg T + \sigma|_g^2 }{\left( W_{k,t,p} + u + \vp \right)^{2^*+1}} \right)}_{\textrm{ Source term }} ,
  \eal \een
 where $T$ satisfies $\Dg T = u^{2^*}X + Y$. We perform a standard finite-dimensional reduction procedure and construct a solution of \eqref{schemastep11} -- up to Kernel elements -- having the form 
 \ben \label{H1soluk}
 u_{k,t,p} = W_{k,t,p} + u + \psi,
 \een for some new remainder $\psi \in H^1(M)$. 

\medskip
\noindent \large \textbf{Sections \ref{partie4}, \ref{theorieC0ordre2} and \ref{pointfixage}: Pointwise asymptotic analysis of the remainder $\psi$ and fixed-point argument.} \normalsize
In view of the ping-pong argument, the main part of the proof is an involved application of Banach-Picard's fixed-point theorem to the remainders' mapping $\vp \mapsto \psi$\footnote{we use the notations of the previous paragraph} in the strong space of functions satisfying \eqref{apriorireste}. The finite-dimensional reduction provides an $H^1$ bound on $\psi$ but says nothing about a  $C^0$ control. In this part we show that $\psi$ still satisfies the pointwise estimate \eqref{apriorireste}. A first step consists in showing that 
\ben \label{petitopsi}
 \psi = o \big(W_{k,t,p} + u \big) \textrm{ in } C^0(M) 
 \een
as $k \to + \infty$. This step is achieved in Section \ref{partie4} through an asymptotic blow-up analysis of the solution $u_{k,t,p}$ in \eqref{H1soluk} and relies on the machinery of the $C^0$ theory, developed in Druet-Hebey-Robert \cite{DruetHebeyRobert} (see also Hebey \cite{HebeyZLAM}), adapted to take into account the source term in \eqref{schemastep11}. 

\smallskip

\noindent A second step consists in quantifying the $o(1)$ in \eqref{petitopsi}. 
This is the purpose of Section \ref{theorieC0ordre2}. Proposition \ref{C02pres} provides a local improvement of \eqref{petitopsi} in the region where the bubbling profile is dominant in the $C^0$ decomposition of $u_{k,t,p}$. It is obtained by a second-order blow-up analysis on and relies on the fact that, by construction, $\psi$ is orthogonal to the Kernel elements of the finite-dimensional reduction. The global improved version of \eqref{petitopsi} is Proposition \ref{propestglob}, which yields the following control on $\psi$: for some positive constant $C$ independent of $k,t,p$,
\[ |\psi| \le C ( \Vert X \Vert_{L^\infty} \vek + \mk ) \Big( W_{k,t,p} + u \Big). \] 
The latter estimate implies in particular that, provided $\Vert X \Vert_{L^\infty(M)}$ is sufficiently small, $\psi$ satisfies again \eqref{apriorireste} for a suitable choice of $(\vek)_k$.
\smallskip

\noindent Section \ref{pointfixage} contains the last step in the proof of the fixed-point argument. The application of Banach-Picard's fixed-point theorem is performed in Proposition \ref{propsystreduit} using again asymptotic analysis techniques in the spirit of those appearing in the proof of Proposition \ref{C02pres}. In Proposition \ref{proploinmieux} we also obtain an improvement of the global estimate on the remainder, in the region where the weak limit becomes dominant in the $C^0$ decomposition of $u_{k,t,p}$. 

\smallskip
\noindent Sections \ref{partie4}, \ref{theorieC0ordre2} and \ref{pointfixage} are the core of the analysis of the paper. The degree of precision of the local and global pointwise estimates on $\psi$ is also the key tool in the final step of the proof of Theorem \ref{thprincipal}.

\medskip

\noindent \large \textbf{Sections \ref{DL} and \ref{argumentconclusif}: Annihilation of the Kernel components.} \normalsize

\noindent Section \ref{pointfixage} provides us with a solution $(u_{k,t,p}, W_{k,t,p})_k$, $ u_{k,t,p} = W_{k,t,p} + u + \vp_k(t,p),$ of \eqref{intro1} up to kernel elements of the scalar equation:
\begin{equation} \label{contconfker}  
\left \{
\begin{aligned}
& \triangle_g u_{k,t,p} + h u_{k,t,p}  = f u_{k,t,p}^{2^*-1} + \frac{ \pi^2 +  |\sigma + \mathcal{L}_g W_{k,t,p} |_g^2 }{ u_{k,t,p}^{2^*+1}} + \sum_{j=0}^n \lambda_{k,j}(t, p) Z_{j,k,t, p}~, \\  
&  \overrightarrow{\triangle}_{g} W_{k,t,p}  = u_{k,t,p}^{2^*}X +Y, ~ 
\end{aligned}
\right.
\end{equation}
where the $Z_{j,k,t,p}$ are defined in \eqref{defZk} (see also Section \ref{DL}). The proof is concluded by showing that one can annihilate all the $\lambda_{k,j}(t,p)$, $0 \le j \le n$, for a suitable value $(t_k,p_k)_k$ of the parameters. 

\smallskip

\noindent In Section \ref{DL} we compute an asymptotic expansion of the $\lambda_{k,j}(t,p)$ as $k \to + \infty$. 
The lack of a variational structure forces us to again proceed differently than in the variational setting. Since the remainder $\vp_k(t,p)$ of the solution $u_{k,t,p}$ was not obtained by variational means we only possess a rough estimate on $\Vert \vp_{k}(t,p) \Vert_{H^1(M)}$. It is given by \eqref{eqsystreduit} below, and is intimately related to the choice of $\vek$ -- and hence to the (limited) precision of \eqref{ponctuelphik} below. In particular, no matter the precision on the choice of the ansatz $W_{k,t,p} + u$, we cannot hope to get a better estimate on $\Vert \vp_k(t,p) \Vert_{H^1(M)}$. Therefore, the usual characterization of the $\lambda_{k,j}(t,p)$ in terms of a reduced energy does not hold here. We overcome this by estimating the $\lambda_{k,j}(t,p)$ directly, using the second-order pointwise estimates on $\vp_k(t,p)$ and $\nabla \vp_k(t,p)$ obtained in Sections \ref{theorieC0ordre2} and \ref{pointfixage}. The conclusion of the proof of Theorem \ref{thprincipal} is then given in Section \ref{argumentconclusif}, using a degree-theoretic argument. 

\bigskip

\noindent Let us conclude this introduction with a few remarks:
\begin{itemize}
\item We prove Theorem \ref{thprincipal} assuming that the $L^\infty$ norm of the coupling field $X$  is small, depending on $n,g, h,f, \pi, \sigma$ (see Section \ref{notations}).  
 This assumption is harmless, since smallness conditions on $X$ are necessary for solutions of \eqref{intro1} to exist: see \cite{HePaPo, Premoselli1, Premoselli2}. 
\item Our choice of the coefficients $(h,f,\pi,\sigma,X,Y)$ is driven by the a priori stability analysis of \cite{Premoselli4}. This is explicit in the choice of the order of smallness of $X$ at bumps, and in the localisation of the concentration point constructed. As shown in \cite{Premoselli4}, the latter point is a common zero of $X$ and $\nabla f$ and a zero of $h - \frac{n-2}{4(n-1)} S_g$ (except in dimension $6$, see below). 
\item We chose to construct here \emph{non-compactness} examples, which forces us to work with bubbling profiles supported in shrinking balls (see Section \ref{notations}). Up to obvious modifications of the coefficients $(h,f,\pi,\sigma,X,Y)$, our proof works for bubbling profiles supported on balls of fixed positive radius thus yielding instability examples for \eqref{intro1}.
\item The scalar component of our blowing-up solutions is constructed to have a positive weak limit $u$. This is necessary, as shown by Lemma \ref{minor} below.
\item The non-degeneracy assumption on $u$ is crucially used twice in our proof: to achieve the finite-dimensional reduction procedure on equation \eqref{schemastep11} and to prove Proposition \ref{propestglob}. It does not restrain the choice of the coefficients $(h,f,\pi,\sigma,Y)$: as shown in Section \ref{notations}, for small $\Vert X \Vert_\infty$, \eqref{intro1} always possesses such a non-degenerate solution. 
 \item As already noticed in \cite{PremoselliWei}, the $6$-dimensional case requires a more careful analysis. Roughly speaking, we are forced to perform an expansion of higher order than in dimensions $n \ge 7$, see Section \ref{DL} below. In particular, a rescaling of $X$ at a concentration point appears in the limiting expansion, see \eqref{J55} and \eqref{defkappa} below.  
\end{itemize}

\medskip

\noindent Possible references for the finite-dimensional reduction method alluded to above, without pretending to be exhaustive, are Ambrosetti-Malchiodi \cite{AmbrosettiMalchiodi}, Berti-Malchiodi \cite{BertiMalchiodi}, Brendle \cite{Brendle}, Brendle-Marques \cite{BreMa}, Del Pino-Felmer-Musso \cite{DelPinoFelmerMusso} Del Pino-Musso-Pacard \cite{DelPinoMussoPacard}, Del Pino-Musso-Pacard-Pistoia \cite{DelPinoMussoPacardPistoia1}, 
Rey \cite{Rey}, Rey-Wei \cite{ReyWei}, Robert-V\'etois \cite{RobertVetois2, RobertVetois4}, Wei \cite{Wei0, Wei, Wei2, WeiGM}, and the references therein.

\noindent The a priori analysis techniques used in our proof have been developed in the context of the $C^0$ theory in Druet-Hebey-Robert \cite{DruetHebeyRobert}, see also Druet \cite{DruetJDG}, Druet-Hebey \cite{DruetHebey2}, Druet-Hebey-Vetois \cite{DruetHebeyVetois2} and Hebey-Robert \cite{HebeyRobert2}. Related techniques have independently been developed in the investigation of compactness phenomena for the Yamabe problem (see Li-Zhu \cite{LiZhu}, Druet \cite{DruetYlowdim}, Marques \cite{Marques}, Li-Zhang \cite{LiZhang}, Khuri-Marques-Schoen \cite{KhuMaSc}) and in the investigation of stability issues for nonlinear elliptic equations (Hebey-Wei \cite{HebeyWei}, Druet-Hebey-V\'etois \cite{DruetHebeyVetois}).

\medskip

\noindent \textbf{Acknowledgements:} The author wants to thank O. Druet for fruitful discussions from which this paper originated.

\section{Setting of the problem and notations} \label{notations}

\noindent In what follows we let $(M,g)$ be a closed $n$-dimensional Riemannian manifold, $n \ge 6$, of positive Yamabe type. We will always assume that $(M,g)$ possesses no non-trivial conformal killing $1$-forms or, equivalently, that the operator $\Dg$ is invertible in $M$. Let $\xi_0 \in M$ and assume that $|W(\xi_0)|_g > 0$ if $(M,g)$ is not locally conformally flat, where $W(\xi_0)$ denotes the Weyl tensor of $g$. The standard conformal normal coordinates result of Lee-Parker \cite{LeeParker} asserts that there exists $\Lambda \in C^\infty(M\times M)$ such that for any point $\xi \in M$ there holds, for some arbitrarily large integer $N$:
\ben \label{confnorm}
\left| \left( \textrm{exp}_\xi^{g_\xi} \right)^* g_{\xi} \right|(y) = 1 + O(|y|^N),
\een
$C^1$-uniformly in $\xi \in M$ and in $y \in T_\xi M$ in a small geodesic ball for the metric $g_\xi$. In \eqref{confnorm} we have let 
\ben \label{metconforme}
g_\xi = \Lambda_\xi^{\frac{4}{n-2}}g,
\een 
where the conformal factor  $\Lambda_\xi = \Lambda(\xi, \cdot)$ can in addition be chosen to satisfy:
\ben \label{propLambda}
\Lax(\xi) = 1, \quad \nabla \Lax (\xi) = 0 .
\een
In \eqref{confnorm}, the notation $\textrm{exp}_\xi^{g_\xi}$ denotes the exponential map for the metric $g_\xi$ at point $\xi$ with the identification of $T_\xi M$ to $\RR^n$ via a smooth orthonormal basis $(e_1, \cdots, e_n)$ of $T_\xi M$ defined in a neighbourhood of $\xi_0$. The Lee-Parker \cite{LeeParker} result also assert that there holds, for any $\xi \in M$:
\ben \label{propLambdaSg}
S_{g_{\xi}}(\xi) = 0, \quad \nabla S_{g_\xi}(\xi) = 0, \quad \triangle_{g_\xi} S_{g_\xi} (\xi) = \frac{1}{6} |W_g(\xi)|_g^2 ,
\een
where $S_{g_\xi}$ denotes the scalar curvature of the conformal metric $g_\xi$.

\medskip

\noindent Let $(\tau_k)_k$ be a sequence of positive real numbers  such that $\sum_k \tau_k < +\infty$. We define a sequence $(\mk)_k$ as follows:
\ben \label{defmk}
\mk = \left \{
\bal
& \tau_k & \textrm{ if } n = 6, \\
& \tau_k^{\frac{2}{n-6}} &\textrm{ if } (M,g) \textrm{ is l.c.f. or if } 7 \le n \le 10, \\
& \tau_k^{\frac{1}{2}}  &\textrm{ if } n \ge 11 \textrm{ and } (M,g) \textrm{ is not l.c.f.}.\\
\eal
\right.
\een
Let $(\xk)_k$ be a sequence of points of $M$ converging towards $\xi_0$ and satisfying $d_g(\xk, \xi_{k+1}) << \frac{1}{k^2}$ as $k \to +\infty$. Let $(\beta_k)_k$ be a sequence of positive numbers converging to zero as $k \to + \infty$ and satisfying:
\ben \label{propbetak}
\bal
& \beta_k  >> \mk \textrm{ if } n \ge 7 \\
 \mu_k^{\frac12} >> & \beta_k >> \mk  \textrm{ if } n = 6 . \\
\eal
\een
Let $f $ be a smooth positive function, let $\sigma$ be a smooth traceless and divergence-free $(2,0)$-tensor in $M$ and let $\pi$ be a smooth function in $M$ with $\pi \not \equiv 0$. Let $Y$ be a smooth field of $1$-forms and denote by $\tilde{Y}$ the only solution of $\Dg \tilde{Y} = Y$ in $M$. We let also $H$ be a smooth nonnegative function in $\RR^n$, compactly supported in $B_0(1)$ with $H(0) = 1$, and for which $0$ is a \emph{non-degenerate critical point}. 

\bigskip

\noindent \textbf{The $n \ge 7$ case.}  We define:
\ben \label{defhintro}
\bal
h & = c_n S_g + \sum_{k \ge 0}  \tau_k H \left( \frac{1}{\beta_k} \big( \exk \big)^{-1}(x) \right),
\eal
\een
where we have let $c_n = \frac{n-2}{4(n-1)}$ for all $n$. With this definition, $h$ is equal to $c_n S_g + \tau_k H \left( \frac{1}{\beta_k} \big( \exk \big)^{-1}(\cdot) \right)$ on every $B_{\xi_k}(\beta_k)$ -- where the ball is taken for the metric $g_{\xi_k}$ -- and equals $c_n S_g$ outside of their reunion.  Note, with \eqref{defmk} and \eqref{propbetak} and since we are considering the $n \ge 7$ case here, that for any $r \in \mathbb{N}^*$, one can always choose $\beta_k$ as in \eqref{propbetak} so that $h \in C^r(M)$. Let $u_0$ be a smooth, positive solution of the following Einstein-Lichnerowicz equation:
\ben \label{defu0}
\triangle_g u_0 + h u_0 = f u_0^{2^*-1} + \frac{|\Lg \tilde{Y} + \sigma|_g^2 + \pi^2}{{u_0}^{2^*+1}},
\een
The coefficients $f, \pi, \sigma$ and $Y$ can always be chosen so that \eqref{defu0} possesses a smooth positive solution, see 
 Hebey-Pacard-Pollack \cite{HePaPo}. Up to a slight modification of $f, \pi, \sigma, Y$ we can always assume $u_0$ to be \emph{strictly stable}, that is satisfying, for any $ \psi \in H^1(M)$:
\ben  \label{uetastable}
\int_M |\nabla \psi |_g^2 + \left [c_n S_g - (2^*-1)f u_0^{2^*-2} + (2^*+1) \frac{|\Lg \tilde{Y} + \sigma|_g^2 + \pi^2}{u_0^{2^*+2}} \right] \psi^2 dv_g \ge \frac{1}{C_0} \| \psi \|_{H^1}^2, 
\een
for some positive constant $C_0$. This is the case if $u_0$ is chosen to be the smallest solution of \eqref{defu0}, see Premoselli \cite{Premoselli2}, Section $7$. 
Such a minimal solution is then the only strictly stable one of \eqref{defu0} (see Dupaigne \cite{Dupaigne}, proposition $1.3.1$). 

\medskip

\noindent \textbf{The $n = 6$ case.} Up to assuming that the $(\tau_k)_k$ are small enough, and since $(M,g)$ is of positive Yamabe type, standard sub- and super-solution arguments show that there exists a unique positive solution $u_0$ of:
\[  \triangle_g u_0 + \left( \frac15 S_g - \sum_{k \ge 0} \tau_k H \left( \frac{1}{\beta_k} \big( \exk \big)^{-1}(x) \right)\right) u_0 = - f u_0^{2} + \frac{|\Lg \tilde{Y} + \sigma|_g^2 + \pi^2}{{u_0}^{4}}.\]
With \eqref{propbetak}, for any $0 <\alpha<1$ such a $u_0$ can be chosen to belong to $C^{3,\alpha}(M)$ and it is then a positive and strictly stable solution (in the sense of \eqref{uetastable}) of the following Einstein-Lichnerowicz equation:
\[ \triangle_g u_0 + h u_0 = f u_0^{2} + \frac{|\Lg \tilde{Y} + \sigma|_g^2 + \pi^2}{{u_0}^{4}}, \]
where we have let in this case:
\ben \label{defhintro6}
h = \frac15 S_g + 2 f u_0- \sum_{k \ge 0} \tau_k H \left( \frac{1}{\beta_k} \big( \exk \big)^{-1}(x) \right).
\een
Again, \eqref{defmk} and \eqref{propbetak} ensure that this $h$ can be chosen to belong to $C^{1,\alpha}(M)$ for fixed $0 < \alpha < 1$, but not to $C^2(M)$, contrary to the $n \ge 7$ case.

\medskip

\noindent Now, in every dimension $n \ge 6$, the implicit function theorem shows that there exists a constant $\eta_0 = \eta_0(n,g,h,f,\pi,\sigma,Y)$ such that, for any field $X$ of $1$-forms in $M$ satisfying 
\ben \label{introeta}
  \| X \|_{L_\infty(M)} = \eta \le \eta_0,
 \een 
the Einstein-Lichnerowicz system of equations:
\ben \label{syseta}
\left \{
\bal
\triangle_g u + hu & = f u^{2^*-1} + \frac{|\Lg T + \sigma|^2 + \pi^2}{u^{2^*+1}} \\
\Dg T & = u^{2^*} X + Y,
\eal
\right.
\een
possesses a solution $(u(X), T(X))$ such that $u(X) \to u_0$ in $C^2(M)$ as $\eta$, defined in \eqref{introeta}, goes to $0$. In \eqref{syseta}, $h$ is given by \eqref{defhintro} or \eqref{defhintro6}, depending on the dimension. Up to choosing $\eta$ small enough, it is easily seen that $u(X)$ is again a strictly stable solution of the scalar equation of \eqref{syseta}, that is:
\ben \label{uetastableX}
\int_M |\nabla \psi |_g^2 + \left [h - (2^*-1)f u(X)^{2^*-2} + (2^*+1) \frac{|\Lg T(X) + \sigma|_g^2 + \pi^2}{u(X)^{2^*+2}} \right] \psi^2 dv_g \ge \frac{1}{C} \| \psi \|_{H^1}^2, 
\een
for any $\psi \in H^1(M)$ and for some positive constant $C$ independent of $\eta$. In the following, for the sake of clarity and since $X$ will be fixed, the dependence on $X$ of the solutions $(u(X), T(X))$ will be omitted, and they will just be denoted by $(u,T)$. 

\medskip
\noindent The space $H^1(M)$ will denote the standard Sobolev space of functions in $L^2$ possessing $L^2$ distributional derivatives, and we endow $H^1(M)$ with the following scalar product: for any $u,v \in H^1(M)$,
\ben \label{psH1}
\langle u, v \rangle_h = \int_M \left(  \langle \nabla u,  \nabla v \rangle_g + h uv \right) dv_g,
\een
where $h$ is given by \eqref{defhintro} or \eqref{defhintro6}. For any $J \in H^1(M)'$ we will denote by $(\triangle_g +h)^{-1}(J)$ the unique element  of $H^1(M)$ such that for any $v \in H^1(M)$ there holds:
\[ J(v) = \langle (\triangle_g +h)^{-1}(J), v \rangle_h. \]
Since no ambiguity will occur, the notation $H^1(M)$ will also be used to denote the space of Sobolev $1$-forms over $M$.

\medskip

\noindent The blowing-up solutions of Theorem \ref{thprincipal} are obtained by glueing degenerating peaks over the solution $(u,T)$. Here we define such peaks. Let $(r_k)_k$ be a sequence of positive real numbers which converges towards zero as $k \to \infty$ and satisfy:
\ben \label{relationrkmk}
\bal
\beta_k << r_k  << d_g(\xk, \xi_{k+1}) \quad \textrm{ and } \quad r_k^N >> \mk, \\
\eal
\een
where $\beta_k$ is given by \eqref{propbetak}, and for some large enough integer $N$, as $k \to + \infty$. The lowest value of $N$ required can be made explicit (only depending on $n$), as is easily seen in the proof of Theorem \ref{thprincipal}.  
 Examples of sequences $(\tau_k)_k$, $(\beta_k)_k$ and $(r_k)_k$ that satisfy \eqref{propbetak}, \eqref{relationrkmk} and $\sum_k \tau_k < + \infty$ are easily found, for instance as different powers of $(1/k)_k$. For $t >0$ we define the sequence $(\delta_k(t))_k$ by: 
 \ben \label{defdkyk}
\delta_k(t)  = \mk t,
\een
where $\mk$ is as in \eqref{defmk}. We let $r_0 >0$ be such that $r_0 < i_{g_\xi}(M)$ for all $\xi \in M$,  where $i_{g_\xi}$ denotes the injectivity radius of the metric $g_\xi$ given by \eqref{metconforme}. By definition of $r_k$ there holds $2r_k < r_0$ for any $k$.  We let $\chi \in C^\infty(\RR)$ be a nonnegative, smooth compactly supported function such that $\chi \equiv 1$ in $[-1,1]$ and $\chi \equiv 0$ outside of $[-2,2]$. The blow-up profiles we investigate in this work are given by the following expression: for $t >0$ and $\xi \in M$, and for any $x \in M$:
\ben \label{bulle}
W_{k,t,\xi}(x) = \Lambda_{\xi}(x) \chi \left( \frac{d_{g_{\xi}}(\xi,x)}{r_k}\right) \delta_k^{\frac{n-2}{2}} \left( \delta_k^2 + \frac{f(\xi)}{n(n-2)}  d_{g_{\xi}}(\xi,x)^2 \right)^{1 - \frac{n}{2}},
\een
where $\Lambda_{\xi}$ is as in \eqref{propLambda} and $\delta_k$ is given by \eqref{defdkyk}. These profiles are localized in $B_{\xi}(2r_k)$, which denotes here the geodesic ball of radius $2 r_k$ with respect to the metric $g_{\xi}$. 
For a given $\xi \in M$ we let $V_{0,\xi}, \cdots, V_{n,\xi}: \RR^n \to \RR$ be given by:
\ben \label{defVki}
\bal
V_{0,\xi}(x) & =  \left( \frac{f(\xi)}{n(n-2)}|x|^2 - 1 \right) \left( 1 + \frac{f(\xi)}{n(n-2)}|x|^2\right)^{-\frac{n}{2}} \\
V_{i,\xi}(y) & = f(\xi) x_i \left( 1 + \frac{f(\xi)}{n(n-2)}|x|^2 \right)^{-\frac{n}{2}} .\\
\eal
\een
For any $0 \le i \le n$, the $V_{i,\xi}$ span the set of solutions in $H^1(\RR^n)$ of the linearized equation (Bianchi-Egnell \cite{BianchiEgnell}):
\ben \label{eqlin}
\triangle_{\textrm{eucl}} V_{i,\xi} = (2^*-1) f(\xi) {U_\xi}^{2^*-2} V_{i,\xi},
\een
where we have let: 
\ben \label{eqlindefU}
U_{\xi}(y) = \left( 1 + \frac{f(\xi)}{n(n-2)} |y|^2\right)^{1 - \frac{n}{2}}, \textrm{ for any } y \in \RR^n.
\een 
We also define, for any $x \in M$, any $1 \le i \le n $ and any $\xi \in M$:
\ben \label{defZk}
\bal
Z_{0,k,t,\xi}(x) & = \Lambda_{\xi}(x)  \chi \left( \frac{d_{g_{\xi}}(\xi,x)}{r_k}\right) \delta_k^{\frac{n-2}{2}}  \left(  \delta_k^2 + \frac{f(\xi)}{n(n-2)}  d_{g_{\xi}}(\xi,x)^2 \right)^{- \frac{n}{2}} \\
& \times \left( \frac{f(\xi)}{n(n-2)}d_{g_{\xi}}(\xi,x)^2 - \delta_k^2 \right) \\
Z_{i,k,t,\xi}(x) &=  \Lambda_{\xi}(x)  \chi \left( \frac{d_{g_{\xi}}(\xi,x)}{r_k}\right)  \delta_k^{\frac{n}{2}} \left(  \delta_k^2 + \frac{f(\xi)}{n(n-2)}  d_{g_{\xi}}(\xi,x)^2 \right)^{- \frac{n}{2}} \\
& \times f(\xi)\left \langle \left( \textrm{exp}_{\xi}^{g_{\xi}}\right)^{-1}(x), e_i(\xi) \right\rangle_{g_{\xi}(\xi)}.
\eal
\een
In \eqref{defZk}, the $(e_i)_i$ denote the field of orthonormal basis centered at $\xi_0$ introduced in the beginning of this Section. Finally, we let 
\ben \label{noyau}
K_{k,t,\xi} = \textrm{Span} \left \{ Z_{i,k,t,\xi}, i=0 \dots n \right \}.
\een
Since $(V_{0,\xi}, \cdots, V_{n,\xi})$ forms an orthonormal family for the scalar product $(u,v) = \int_{\RR^n} \langle \nabla u, \nabla v \rangle dx$ in $\RR^n$, $K_{k,t,\xi}$ is $(n+1)$-dimensional for $k$ large enough and the $Z_{i,k,t,\xi}$ are ``almost'' orthogonal. We denote by $K_{k,t,\xi}^{\perp}$ its orthogonal in $H^1(M)$ for the scalar product given by \eqref{psH1}.

\medskip

\noindent We did not specify the choice of $f$ and $X$ yet. Let $f_0 > 0$ be a positive constant. Here $f$ will be a perturbation of $f_0$ by small bumps:
\ben \label{deffintro}
f = f_0 + \sum_{k \ge 0} s_k \chi\left( \frac{1}{r_k} \big( \exk \big)^{-1}(x) \right), 
\een
where $(s_k)_k$ is a sequence of real numbers converging to zero and satisfying $|s_k| = O(\mk^N)$ for a sufficiently large $N \in \mathbb{N}^*$.  Let $X_0$ denote any smooth field of $1$-forms in $M$ which vanishes in a neighbourhood of $\xi_0$. Let $Z$ be a fixed smooth $1$-form in $\RR^n$, compactly supported in $B_0(1)$, and with $|Z_0(0)| > 0$. Define then, for any $x \in M$:
\ben \label{defXintro}
X(x) = X_0(x) + \sum_{k \ge 0} \mk^{\frac{n-1}{2}} Z \left( \frac{1}{r_k} \big( \exk \big)^{-1}(x) \right) ,
\een
where $\mk$ and $r_k$ are as in \eqref{defmk} and \eqref{relationrkmk}. Up to reducing $\Vert X_0 \Vert_\infty$ and the $\tau_k$ such an $X$ always satisfies \eqref{introeta}. In the following, we will also let: 
\ben \label{defalpha}
\alpha := |Z(0)| = \frac{|X(\xi_k)|_{g_{\xi_k}}}{\mk^{\frac{n-1}{2}}} \textrm{ for all } k.
\een
This parameter $\alpha$ will have to be chosen small, see the proof of Proposition \ref{propsystreduit}. Again, with \eqref{defmk} and \eqref{relationrkmk}, $f$ and $X$ can always be chosen to belong to $C^r(M)$ for $r \in \mathbb{N}^*$.

\medskip

\noindent Finally, let 
\ben \label{defEk}
 \mathcal{E} = \left \{ (\vek)_{k\in \mathbb{N}}, \vek >0, \lim_{k \to \infty} \vek = 0 \right \}
 \een
be the set of sequences of positive real numbers converging to $0$. For $(\vek)_k \in \mathcal{E}$ and for a given value of $(t,\xi) \in (0, +\infty) \times M$ we define the following sequence of subsets of $C^2(M)$:
\ben \label{defFea}
F_k = F \big( \vek,t,\xi \big) = \left \{ v \in C^0(M) \textrm{ such that }  \left \| \frac{v}{u + W_{k,t,\xi}} \right \|_{C^0(M)} \le \vek \right \},
\een
where $u=u(X)$ is defined after \eqref{syseta} and $W_{k,t,\xi}$ is as in \eqref{bulle}.

\section{An $H^1$ finite-dimensional reduction method for the semi-decoupled system} \label{theorieH1}

\noindent As discussed in the Introduction, the supercritical nonlinear coupling of system \eqref{intro1} makes the usual variational energy methods ineffective.  In this Section we perform the first step of the fixed-point procedure in the proof of Theorem \ref{thprincipal}.

\medskip

\noindent As a first step in view of the $H^1$-theory, we get rid of the negative non-linearity in \eqref{intro1}. We let, for any $\ve >0$ and $r \in \RR$:
\[
\rho_\ve(r) = \left \{
\bal
& \ve \textrm{ if } r < \ve \\
& r \textrm{ if } r \ge \ve. \\ 
\eal \right.
\] 
We let $T$ denote any field of $1$-forms and we introduce the following truncation of the scalar equation of \eqref{intro1}:
\ben \label{ELeps}
\triangle_g u + hu = f u^{2^*-1} + \frac{|\Lg T + \sigma|^2 + \pi^2}{\rho_{\ve}(u)^{2^*+1}}.
\een
The following Lemma, proven in Premoselli-Wei \cite{PremoselliWei}, holds:
\begin{lemme} \label{minor}
There exists $\ve_0 > 0$ depending only on $g,h,f,\pi$  such that for any $1$-form $T$ and any $0 \le \ve \le \ve_0$, any $C^2$ positive solution $u$ of \eqref{ELeps} satisfies:
\[ \min_M u \ge \ve_0. \]  
In particular, for $0 \le \ve \le \ve_0$, any $C^2$ positive solution of \eqref{ELeps}  also solves:
\[ \triangle_g u + hu = f u^{2^*-1} + \frac{|\Lg T + \sigma|^2 + \pi^2}{u^{2^*+1}}. \]
\end{lemme}
\noindent In the following, if $\ve_0$ is given by Lemma \ref{minor}, we will let:
\ben \label{defeta}
\rho = \rho_\ve \textrm{ for some fixed } 0 < \ve \le \frac{1}{4} \ve_0.
\een 

\noindent We now decouple system \eqref{intro1}. Let $(\vek)_k \in \mathcal{E}$, $(t,\xi) \in (0, +\infty) \times M$  and $v_k \in F_k = F \big( \vek,t,\xi \big)$, where $\mathcal{E}$ and $F_k$ are defined in \eqref{defEk} and \eqref{defFea}. To every $v_k$ we associate a unique $1$-form $T_k = T_{k,t,\xi}$ defined as the unique solution in $M$ of
\ben \label{defZk1forme}
\Dg T_k = \Big(u + W_{k,t,p} + v_k \Big)^{2^*} X + Y.
\een
Such a $1$-form $T_k$ is unique since $(M,g)$ possesses no conformal Killing fields, and Proposition \ref{controleLTkdessus} provides us with sharp pointwise asymptotics on $T_k$. We introduce the following equation in $M$, which is the scalar equation of \eqref{intro1} with an additional source term, of unknown $w$:
\ben \label{damping}
\triangle_g w + hw = fw^{2^*-1} + \frac{|\Lg T + \sigma|^2 + \pi^2}{\rho(w)^{2^*+1}} +  \frac{ |\Lg T_{k,t,\xi} + \sigma|_g^2 - |\Lg T + \sigma|_g^2 }{\left( u + W_{k,t_k,\xi_k} + v_k \right)^{2^*+1}},
\een
where $T_{k,t,\xi}$ is given by \eqref{defZk1forme} and $(u,T)$ are as in \eqref{syseta}. Choosing $v_k$ and investigating equation \eqref{damping} amounts to semi-decoupling the system \eqref{intro1}.
In this section we solve equation \eqref{damping} \emph{via} a standard $H^1$ finite-dimensional reduction.

\medskip
\noindent Classically, the first step consists in showing that the linearized operator for the scalar equation of \eqref{intro1} can be inverted in the orthogonal of $K_{k,t,\xi}$:
\begin{prop} \label{propinvoplin}
Let $D >0$. There exists a positive constant $C$ such that for any $(t,\xi) \in [1/D, D]\times M$ and for any $k$ there holds:
\ben \label{invoplin}
\frac{1}{C}  \| \psi \|_{H^1(M)}  \le \Vert L_{k,t,\xi} (\psi) \Vert_{H^1(M)} \le C \| \psi \|_{H^1(M)}  \textrm{ for any } \psi \in K_{k,t,\xi},
\een
where $K_{k,t,\xi}$ is as in \eqref{noyau} and where we have let, for any $\psi \in K_{k,t,\xi}$ :
\ben \label{defoplin}
\bal
L_{k,t,\xi}(\psi) = 
\Pi_{K_{k,t,\xi}^{\perp}} \Bigg(\psi - \left( \triangle_g + h \right)^{-1} \Big[ (2^*-1) f \big( u + W_{k,t,\xi} \big)^{2^*-2} \psi  \\
- (2^*+1) \frac{|\Lg T + \sigma|^2 + \pi^2}{\big(u + W_{k,t,\xi} \big)^{2^*+2}} \psi \Big]  \Bigg),\\
\eal 
\een
where $\Pi_{K_{k,t,\xi}^{\perp}}$ denotes the orthogonal projection on $K_{k,t,\xi}^\perp$ with respect to the scalar product given in \eqref{psH1}. In \eqref{defoplin}, $(u,T) = (u(X), T(X))$ are defined in the discussion following \eqref{syseta}.
\end{prop}
\noindent The proof of Proposition \ref{propinvoplin} follows from standard arguments and is clearly detailed in Robert-Vetois \cite{RobertVetois}. The main result of this section is the resolution of the full non-linear equation \eqref{damping}, given by the following Proposition:

\begin{prop} \label{propptfixe1}
Let $D >0$ and $(\vek)_k \in \mathcal{E}$ and assume that  $\vek >> \mk^{\frac{3}{2}}$ as $k \to + \infty$, where $\mk$ is defined in \eqref{defmk}. Let $(t_k,\xi_k)_k$ be a sequence in $[1/D, D]\times M$ and, for any $k$, let $v_k \in F_k = F \big( \vek,t_k,\xi_k \big)$. For $k$ large enough, there exists a function $\phi_k = \phi_k \big(t_k,\xi_k,v_k \big) \in K_{k,t_k,\xi_k}^\perp$ that solves the following equation:
\ben \label{restephik}
\bal
&\Pi_{K_{k,t_k,\xi_k}^\perp} \Bigg \{ u + W_{k,t_k,\xi_k} + \phi_k \\
&- \left( \triangle_ g+ h \right)^{-1} \left( f \left( u + W_{k,t_k,\xi_k} + \phi_k \right)^{2^*-1}  + \frac{|\Lg T + \sigma|_g^2 + \pi^2}{\rho \left( u + W_{k,t_k,\xi_k} + \phi_k\right)^{2^*+1}} \right) \\
& - \left( \triangle_g + h \right)^{-1} \left( \frac{ |\Lg T_k + \sigma|_g^2 - |\Lg T + \sigma|_g^2 }{\left( u + W_{k,t_k,\xi_k} + v_k \right)^{2^*+1}} \right) \Bigg \} = 0.
\eal
\een
This $\phi_k$ is the unique solution of \eqref{restephik} in $K_{k,t_k,\xi_k}^\perp \cap B_{H^1(M)}(0, C \eta \ve_k)$, where $C$ denotes some
positive constant that does not depend on $k$, on the choice of $(t_k,\xi_k)_k$ or on $\eta$ as in \eqref{introeta}. Also, in \eqref{restephik}, $K_{k,t_k,\xi_k}$ is as in \eqref{noyau}, $\rho$ is as in \eqref{defeta},
 and $T_k$ is as in \eqref{defZk1forme}. 
\end{prop}
\noindent As before, in \eqref{restephik} $(u,T) = (u(X), T(X))$ denote the specific solution of \eqref{syseta} obtained by the implicit function theorem when \eqref{introeta} holds. As an obvious consequence of Proposition \ref{propptfixe1}, the function $\phi_k$ constructed therein satisfies:
\ben \label{propphik1}
\phi_k \in K_{k,t_k,\xi_k}^\perp \textrm{ for all } (t_k,\xi_k)_k \in [1/D, D] \times M,
\een
and
\ben \label{propphik2}
\Vert \phi_k \Vert_{H^1(M)} \le C \eta \vek,
\een
for some constant $C$ which is independent of $(t_k,\xi_k)_k$, $k$ and $\eta$.

\medskip
\noindent Note that the necessity to introduce a source term in \eqref{damping} is due to the blow-up behavior of $\Lg T_k$ as given by Proposition \ref{controleLTkdessus} below. Here, in the setting of Proposition \ref{propptfixe1}, the finite-dimensional reduction will work since the nonlinearity $u \mapsto (|\Lg T + \sigma|^2 + \pi^2)\rho(u)^{-2^*+1}$ is of subcritical type in the sense of Robert-V\'etois \cite{RobertVetois}. But, as can be easily checked, there is no hope to even get an analogue of Proposition \ref{propinvoplin} if instead we considered $ (|\Lg T_k + \sigma|^2 + \pi^2)\rho(u)^{-2^*+1}$, with $T_k$ given by \eqref{defZk1forme}.

\begin{proof}
Let $D >0$ and $(\vek)_k \in \mathcal{E}$. Assume that 
\ben \label{condvek}
\vek >> \mk^{\frac{3}{2}}
\een
where $\mk$ is as in \eqref{defmk}. Let $(t_k,\xi_k)_k \in [1/D,D] \times M$ and $(v_k)_k$, $v_k \in F_k =  F_{\ve_k,t_k,\xi_k}$ be fixed. In the proof of Proposition \ref{propptfixe1}, for the sake of clarity and since no ambiguity will occur, we will omit the dependence in $t_k$ and $\xi_k$ in the quantities appearing. Hence we shall denote $K_{k,t_k,\xi_k}$ by $K_k$, $W_{k,t_k,\xi_k}$ by $W_k$ and so on. Let $\phi \in K_{k}^\perp$ be fixed.  It is easily seen that $\phi$ solves:
\be
\bal
\Pi_{K_{k}^\perp} &\Bigg \{ u + W_{k} + \phi \\
&- \left( \triangle_ g+ h \right)^{-1} \left( f \left( u + W_{k} + \phi \right)^{2^*-1}  + \frac{|\Lg T + \sigma|_g^2 + \pi^2}{\rho \left( u + W_{k} + \phi\right)^{2^*+1}} \right) \\
& + \left( \triangle_g + h \right)^{-1} \left( \frac{ |\Lg T_k + \sigma|_g^2 - |\Lg T + \sigma|_g^2 }{\left( u + W_{k} + v_k \right)^{2^*+1}} \right) \Bigg \} = 0
\eal
\ee
if and only if $\phi$ solves the fixed-point equation:
\be 
\phi = \Theta_k(\phi) \qquad \textrm{ in } K_k^{\perp},
\ee
where we have let
\ben \label{propptfixe2}
\Theta_k(\phi) = L_{k}^{-1} \circ \Pi_{K_k^\perp} \circ \left( \triangle_g + h \right)^{-1} \circ N_k(\phi) - L_k^{-1} \circ \Pi_{K_k^\perp} (R_k),
\een
where $L_k= L_{k,t_k,\xi_k}$ is as in \eqref{defoplin} and where in \eqref{propptfixe2} we have let
\ben \label{propptfixe3}
\bal
N_k(\phi) & = f \left[ \left(u + W_k + \phi \right)^{2^*-1} - \left( u + W_k \right)^{2^*-1} - (2^*-1) \left( u + W_k \right)^{2^*-2} \phi \right] \\
& + \left( |\Lg T + \sigma|_g^2 + \pi^2 \right) \Bigg[ \rho\left( u + W_k + \phi \right)^{-2^*-1} - \left( u + W_k \right)^{-2^*-1} + (2^*+1) \left( u+W_k \right)^{-2^*-2} \phi \Bigg]
\eal
\een
and
\ben \label{propptfixe4}
\bal
R_k = u + W_k - \left( \triangle_g + h\right)^{-1} & \left( f \left( u + W_k \right)^{2^*-1} + \frac{|\Lg T + \sigma|_g^2 + \pi^2 }{\left( u + W_k \right)^{2^*+1}}\right) \\
& + \left( \triangle_g + h \right)^{-1} \left( \frac{ |\Lg T_k + \sigma|_g^2 - |\Lg T + \sigma|_g^2}{\left(u + W_k + v_k \right)^{2^*+1} } \right).
\eal
\een
Note that $\Theta_k$ is well-defined because of Proposition \ref{propinvoplin}. We now apply Banach-Picard's fixed-point theorem to $\Theta_k$ defined in \eqref{propptfixe2}. Let $\phi_1, \phi_2 \in K_{k}^\perp$. Standard computations show that there holds, for some positive constant $C$ independent of $k$ and $\eta$ as in \eqref{introeta}:
\ben \label{propptfixe5}
\bal
\left | \left | f \left[ \left(u + W_k + \phi_1 \right)^{2^*-1} - \left( u + W_k + \phi_2 \right)^{2^*-1} - (2^*-1) \left( u + W_k \right)^{2^*-2} ( \phi_1 - \phi_2) \right] \right| \right|_{(H^1(M))'} \\
\le C \left( \Vert \phi_1\Vert_{H^1(M)}^{\frac{4}{n-2}} + \Vert \phi_2 \Vert_{H^1(M)}^{\frac{4}{n-2}} \right)  \Vert \phi_1 - \phi_2 \Vert_{H^1(M)},
\eal
\een
and 
\ben \label{propptfixe6}
\bal
& \Bigg| \Bigg|  \left( |\Lg T + \sigma|_g^2 + \pi^2 \right) \Bigg[ \rho\left( u + W_k + \phi_1 \right)^{-2^*-1} - \rho \left( u + W_k + \phi_2 \right)^{-2^*-1} \\
& \qquad \qquad \qquad \qquad \qquad \qquad \qquad \qquad + (2^*+1) \left( u+W_k \right)^{-2^*-2} (\phi_1 - \phi_2) \Bigg] \Bigg| \Bigg|_{(H^1(M))'} \\
& \le C \sup_{v \in [\phi_1, \phi_2]} \left| \left| \frac{\rho' \left( u + W_k + v \right)}{\rho \left( u+W_k+v \right)^{2^*+2}} - \frac{1}{\left(u + W_k \right)^{2^*+2}} \right| \right|_{L^{\frac{n}{2}}(M)} \Vert \phi_1 - \phi_2 \Vert_{H^1(M)}.
\eal
\een
Using the definition of $\rho$ in \eqref{defeta} one obtains that for any $v \in [\phi_1, \phi_2] \subset K_k^\perp$:
\ben \label{propptfixe7}
 \left| \left| \frac{\rho' \left( u + W_k + v \right)}{\rho \left( u+W_k+v \right)^{2^*+2}} - \frac{1}{\left(u + W_k \right)^{2^*+2}} \right| \right|_{L^{\frac{n}{2}}(M)} \le C \Vert v \Vert_{H^1(M)}^{\frac{4}{n-2}},
 \een
 where $C$ is a positive constant that does not depend on $k, \eta$ or $v \in [\phi_1, \phi_2]$. Gathering \eqref{propptfixe5}, \eqref{propptfixe6} and \eqref{propptfixe7} one therefore has that:
 \ben \label{propptfixe8}
\Vert \Theta_k(\phi_1) - \Theta_k(\phi_2) \Vert_{H^1(M)} \le C_1 \left(  \Vert \phi_1\Vert_{H^1(M)}^{\frac{4}{n-2}} + \Vert \phi_2 \Vert_{H^1(M)}^{\frac{4}{n-2}} \right) \Vert \phi_1 - \phi_2 \Vert_{H^1(M)},  
 \een
for some positive constant $C_1$ that neither depends on $k$ nor on $\eta$. We now estimate $\Theta_k(0)$. By Proposition \ref{propinvoplin} there holds:
\ben \label{propptfixe9}
\Vert \Theta_k(0) \Vert_{H^1(M)} \le C \Vert (\triangle_g + h)R_k \Vert_{L^{\frac{2n}{n+2}}(M)},
\een
for some positive $C$ independent of $k$ and $\eta$, where $R_k$ is defined in \eqref{propptfixe4}. Straightforward computations using \eqref{defhintro}, \eqref{defhintro6} and \eqref{deffintro} show that:
\ben \label{propptfixe91}
\left| \left| \left( \triangle_g + h\right)(u + W_k) -  f \left( u + W_k \right)^{2^*-1} + \frac{|\Lg T + \sigma|_g^2 + \pi^2 }{\left( u + W_k \right)^{2^*+1}} \right| \right|_{L^{\frac{2n}{n+2}}(M)} = o(\dk^{\frac32})
\een 
as $k \to + \infty$, and the $o(\dk^{\frac32})$ term in \eqref{propptfixe91}  is uniform in the choice of $(t_k,\xi_k)_k \in[1/D, D] \times M$. We now write that
\[ |\Lg T_k + \sigma|_g^2 - |\Lg T + \sigma|_g^2 = \left \langle \Lg T_k - \Lg T, \Lg T_k + \Lg T + 2 \sigma \right \rangle_{g}\] 
where the scalar product $\langle \cdot, \cdot \rangle_{g}$ is the standard one induced by the metric $g$ on $(2,0)$-tensors. Using the pointwise estimates on $|\Lg T_k - \Lg T|_g$ given by  \eqref{estLTkLT}, \eqref{defXintro} and \eqref{relationrkmk} 
one finds that:
\ben \label{propptfixe10}
 \left| \left| \frac{ |\Lg T_k + \sigma|_g^2 - |\Lg T + \sigma|_g^2}{\left(u + W_k + v_k \right)^{2^*+1} } \right| \right|_{L^{\frac{2n}{n+2}}(M)} \le C' \left( \eta \vek + \dk^{\frac{n}{4}} \right),
\een
for some positive $C'$ independent of $k$ and $\eta$. Remember that throughout this paper we always assume that $n \ge 6$, so \eqref{propptfixe9}, \eqref{propptfixe91} and \eqref{propptfixe10} show that there exists a positive constant $C_2$ independent of $k$ and $\eta$ such that:
\ben \label{propptfixe11}
\Vert \Theta_k(0) \Vert_{H^1(M)} \le C_2 \Big( \eta \vek + \dk^\frac{3}{2} \Big).
\een
We let now 
\ben \label{propptfixe111}
c_k = 2 C_2 \eta \vek,
\een
where $C_2$ is given by \eqref{propptfixe11} and $\eta$ is as in \eqref{introeta}. With \eqref{propptfixe8}, \eqref{propptfixe11} and \eqref{condvek} we obtain that, for any $\phi_1, \phi_2 \in K_k^\perp \cap B_{H^1(M)}(0,c_k)$, there holds:
\ben \label{propptfixe12}
 \Vert \Theta_k(\phi_1) - \Theta_k(\phi_2) \Vert_{H^1(M)} \le C_1 \Big(2  c_k^\frac{4}{n-2} \Big) \Vert \phi_1 - \phi_2 \Vert_{H^1(M)}
 \een
and 
\ben \label{propptfixe13}
 \Vert \Theta_k(\phi_1) \Vert_{H^1(M)} \le \left( \frac{1}{2} + o(1) \right) c_k.
 \een
Up to choosing $k$ large enough, \eqref{propptfixe12} and \eqref{propptfixe13} show that $\Theta_k$ is $\frac{1}{2}$-Lipschitz from $B_{H^1(M)}(0,2 c_k)$ into itself. Banach-Picard's fixed-point theorem applies and, with \eqref{propptfixe11}, provides us with a function $\phi_k \in K_k^\perp$ satisfying \eqref{restephik}, \eqref{propphik1}, \eqref{propphik2}, and which is the only solution of \eqref{restephik} in $K_k \cap B_{H^1(M)}(0, c_k)$, where $c_k$ is as in \eqref{propptfixe111}.
\end{proof}

\noindent Let us point out again that the proof of Proposition \ref{propptfixe1} crucially relies on the \emph{pointwise} estimates available on $v_k$, in particular for estimate \eqref{propptfixe10}.

\section{Pointwise $C^0$ estimates on $\phi_k$} \label{partie4}

\noindent Let $D >0$ and $(\vek)_k \in \mathcal{E}$, defined in \eqref{defEk}. Assume that  $\vek >> \mk^{\frac{3}{2}}$ as $k \to + \infty$, where $\mk$ is defined in \eqref{defmk}. Let $(t_k,\xi_k)_k$ be any sequence in $ [1/D, D]\times M$. For any $v_k \in F_k = F \big( (\vek)_k,t_k,\xi_k \big)$, Proposition \ref{propptfixe1} shows the existence of a function $\phi_k = \phi_k \big(t_k,\xi_k,v_k \big) \in K_{k,t_k,\xi_k}^\perp$ that solves \eqref{restephik} and satisfies \eqref{propphik1} and \eqref{propphik2}.

\noindent As discussed in the Introduction, the proof of Theorem \ref{thprincipal} goes through the application of a Banach-Picard fixed-point theorem to the mapping $v_k \mapsto \phi_k$. This requires to show in particular that such a mapping leaves $F_k$ defined in \eqref{defFea} invariant or, in other words, to obtain an explicit pointwise control on $\phi_k$. This task is achieved in this section and in the two following ones.

\medskip
\noindent From here until the end of the paper, if $(f_k)_k$ denotes some sequence of real numbers or some sequence of functions, the notation $O(f_k)$ will denote a quantity whose absolute value can be bounded by the product of $|f_k|$ and of a constant \emph{independent of $k$ and $\eta$ as in \eqref{introeta}}. The notation $o(f_k)$ is defined accordingly. Similarly, we will write ``$f_k \lesssim g_k$'' when there exists a positive constant $C$ \emph{independent of $k$ and $\eta$ as in \eqref{introeta}} such that $f_k \le C g_k$ for any $k$.

\medskip
\noindent In this section we obtain a first, rough, asymptotic pointwise control of $\phi_k$:
\begin{prop} \label{propC0grossier}
Let $D >0$ and $(\vek)_k \in \mathcal{E}$ and assume that  $\vek >> \mk^{\frac{3}{2}}$ as $k \to + \infty$, where $\mk$ is defined in \eqref{defmk}. Let $(t_k,\xi_k)_k$ be a sequence of points in $[1/D, D]\times M$, and let
$v_k \in F_k = F \big( \vek,t_k,\xi_k \big)$. There exists a sequence $(\nu_k)_k$ of positive numbers that goes to zero as $k \to + \infty$ such that:
\ben \label{estC0grossiere}
|\phi_k(x)| \le \nu_k \Big( u(x) + W_{k,t_k,\xi_k}(x) \Big) \qquad \textrm{ for any } x \in M.
\een
In \eqref{estC0grossiere} we have let $\phi_k = \phi_k \big(t_k,\xi_k,v_k \big) \in K_{k,t_k,\xi_k}$ be the solution of \eqref{restephik} given by Proposition \ref{propptfixe1}. 
\end{prop}

\noindent As a consequence of Proposition \ref{propC0grossier}, Lemma \ref{minor} and Proposition \ref{propptfixe1} one has that for any $(t_k,\xi_k)_k \in [1/D, D] \times M$ and for any $v_k \in F(\vek,t_k,\xi_k)$, letting
\ben \label{defuk}
u_{k,t_k,\xi_k,v_k} = u + W_{k,t_k,\xi_k} + \phi_k(t_k,\xi_k,v_k),
\een
there exist real numbers $(\lki(t_k,\xi_k,v_k))_{0 \le i \le n}$ such that $u_{k,t_k,\xi_k,v_k}$ satisfies:
\ben \label{vraieeqphik}
\bal
 \left( \triangle_g + h \right)&u_{k,t_k,\xi_k,v_k}  - f u_{k,t_k,\xi_k,v_k}^{2^*-1}  - \frac{|\Lg T + \sigma|_g^2 + \pi^2}{ u_{k,t_k,\xi_k,v_k}^{2^*+1}} \\
& -  \frac{ |\Lg T_k + \sigma|_g^2 - |\Lg T + \sigma|_g^2 }{\left( u + W_{k,t_k,\xi_k} + v_k \right)^{2^*+1}}  = \sum_{i=0}^n \lki(t_k,\xi_k,v_k) \left( \triangle_g + h \right) Z_{i,k}.
\eal
\een
Here again, $T_{k,t_k,\xi_k}$ is as in \eqref{defZk1forme}.

\begin{proof}
Let $D >0$ and $(\vek)_k \in \mathcal{E}$ and assume that  $\vek >> \mk^{\frac{3}{2}}$ as $k \to + \infty$, where $\mk$ is defined in \eqref{defmk}. Let $(t_k,\xi_k)_k$ be a sequence in $ [1/D, D]\times M$, let $v_k \in F_k = F \big( \vek,t_k,\xi_k \big)$ and let $\phi_k = \phi_k \big(t_k,\xi_k,v_k \big) \in K_{k,t_k,\xi_k}^\perp$ be the solution of \eqref{restephik} given by Proposition \ref{propptfixe1}. In particular, for any $k$ large enough, there exist $(\lki)_{0 \le i \le n } = (\lki(t_k,\xi_k, v_k))_{0 \le i \le n}$ such that $\phi_k$ solves:
\ben \label{eqphikgros}
\bal
 \left( \triangle_g + h \right) \left( u + W_k + \phi_k \right)& - f \left( u + W_k + \phi_k \right)^{2^*-1}  - \frac{|\Lg T + \sigma|_g^2 + \pi^2}{\rho \left( u + W_k + \phi_k\right)^{2^*+1}} \\
& -  \frac{ |\Lg T_k + \sigma|_g^2 - |\Lg T + \sigma|_g^2 }{\left( u + W_k + v_k \right)^{2^*+1}}  = \sum_{i=0}^n \lki \left( \triangle_g + h \right) Z_{i,k},
\eal
\een
where the $Z_{i,k}$ are defined in \eqref{defZk} and where $T_k$ is as in \eqref{defZk1forme}. As before, in \eqref{eqphikgros} and later on we shall omit the dependence in $t_k$ and $\xi_k$ and let $W_k = W_{k,t_k,\xi_k}$ and so on. A first, obvious remark is that \eqref{eqphikgros} and standard elliptic regularity results show that $u+W_k + \phi_k$ belongs to $C^1(M)$, and then so does $\phi_k$. The proof of Proposition \ref{propC0grossier} goes through several steps.

\medskip

\noindent \textbf{Step $1$: Estimation of the $\lki$.}
We show that there holds:
\ben \label{C01step1}
\sum_{i = 0}^n |\lki| = O (\eta \vek) + O (\dk^{\frac{3}{2}}).
\een
For this, we rewrite \eqref{eqphikgros} as:
\ben \label{C01step11}
\bal
 (\triangle_g + h )R_k +& \left( \triangle_g + h \right) \phi_k -  f \left[ \left( u + W_k + \phi_k \right)^{2^*-1} - \left(u + W_k \right)^{2^*-1} \right] \\
& -  \left( |\Lg T + \sigma|_g^2 + \pi^2 \right) \left(\rho \left( u + W_k + \phi_k\right)^{-2^*-1} - (u+W_k)^{-2^*-1} \right)  \\
&= \sum_{i=0}^n \lki \left( \triangle_g + h \right) Z_{i,k},
\eal
\een
where $R_k$ is as in \eqref{propptfixe4}. Let $0 \le i \le n$ be fixed. There holds, by \eqref{propphik2}, \eqref{propptfixe91}, \eqref{propptfixe10} and H\"older's inequality that:
\ben \label{C01step12}
\left \langle R_k + \phi_k, Z_{i,k} \right \rangle_h = O (\dk^\frac{3}{2}) + O(\eta \vek).
\een
Since we have:
\[ \left| \left( u + W_k + \phi_k \right)^{2^*-1} - \left(u + W_k \right)^{2^*-1} \right| = O\left( \left(u+ W_k\right)^{2^*-2} |\phi_k| \right) + O \left(|\phi_k|^{2^*-1} \right),  \]
there holds that 
\ben \label{C01step13}
\int_M  f \left[ \left( u + W_k + \phi_k \right)^{2^*-1} - \left(u + W_k \right)^{2^*-1} \right] Z_{i,k} dv_g = O(\eta \vek).
\een
The sequence of functions $\Big(\left( |\Lg T + \sigma|_g^2 + \pi^2 \right) \left(\rho \left( u + W_k + \phi_k\right)^{-2^*-1} - (u+W_k)^{-2^*-1} \right) \Big)_k$ is uniformly bounded in $L^\infty(M)$, so we have, since $n \ge 6$:
\ben \label{C01step14}
\int_M \left( |\Lg T + \sigma|_g^2 + \pi^2 \right) \left(\rho \left( u + W_k + \phi_k\right)^{-2^*-1} - (u+W_k)^{-2^*-1} \right)Z_{i,k} dv_g = o(\dk^{\frac32}).
\een
Finally, there holds for any $0 \le j \le n$, that:
\ben \label{C01step15}
\left \langle Z_{i,k} , Z_{j,k} \right \rangle_h = \delta_{ij} \Vert \nabla V_{i,\xi} \Vert_{L^2(\RR^n)}^2 + o(1) ,
\een
where $V_{i,\xi} $ is as in \eqref{defVki}. Multiplying \eqref{C01step11} by $Z_{i,k}$, integrating and using \eqref{C01step12} -- \eqref{C01step15} yields \eqref{C01step1}.

\medskip

\noindent \textbf{Step $2$: Local behavior of $\phi_k$.}
In this step we show that:
\ben \label{C01step2}
\dk^\pui u_k \left( \exp_{\xi_k}^{g_{\xi_k}} (\dk \cdot ) \right) \longrightarrow U_{\xi_0} \textrm{ in } C^1_{loc}(\RR^n),
\een
as $k \to + \infty$, where $\xi_0 = \lim_{k \to +\infty} \xi_k$, $U_{\xi_0}$ is defined in \eqref{eqlindefU}, $\dk$ in \eqref{defdkyk} and the exponential map for $g_{\xi_k}$ is as in \eqref{confnorm} and \eqref{metconforme}. In order to prove \eqref{C01step2}, we define $\tilde{u}_k$ in $B_0(r_k / \dk)$ by:
\ben \label{C01step21}
\tilde{u}_k(x) = \dk^\pui u_k \left( \exp_{\xi_k}^{g_{\xi_k}} (\dk x ) \right).
\een
Remember that $r_k$ is a positive radius given by \eqref{relationrkmk} and that $W_k$ is supported in $B_{\xi_k}(2 r_k)$, the ball being taken for the metric $g_{\xi_k}$. It is easily seen that for any $x \in B_0(r_k/\dk)$, $\tilde{u}_k$ satisfies:
\ben \label{C01step22}
\bal
 & \triangle_{g_k} \tilde{u}_k(x) + \dk^2 h(x_k) \tilde{u}_k(x) = f(x_k) \tilde{u}_k(x)^{2^*-1} + \dk^{\frac{n+2}{2}} \frac{|\Lg T + \sigma|_g^2 + \pi^2}{ \rho \left( u + W_k + \phi_k\right)^{2^*+1}}(x_k) \\
 &+ \dk^{\frac{n+2}{2}} \frac{ |\Lg T_k + \sigma|_g^2 - |\Lg T + \sigma|_g^2 }{\left( u + W_k + v_k \right)^{2^*+1}}(x_k)  + \sum_{i=0}^n \lki \dk^{\frac{n+2}{2}} \left( \triangle_g + h \right) Z_{i,k}(x_k),
 \eal
  \een
where, in \eqref{C01step22}, we have let $x_k = \exp_{\xi_k}^{g_{\xi_k}} \left( \dk x \right)$ and $g_k =  \left( \exp_{\xi_k}^{g_{\xi_k}}\right)^*g_{\xi_k}(\dk \cdot)$. By definition of $\rho$ in \eqref{defeta} there holds:
\ben \label{C01step23}
\dk^{\frac{n+2}{2}} \frac{|\Lg T + \sigma|_g^2 + \pi^2}{ \rho \left( u + W_k + \phi_k\right)^{2^*+1}}(x_k) \longrightarrow 0 \textrm{ in } C^0_{loc}(\RR^n).
\een
Using \eqref{estLTkLT} below, \eqref{defXintro}, \eqref{relationrkmk} and the pointwise control on $v_k$ given by the definition of $F_k$ in \eqref{defFea} we also have:
\ben \label{C01step24}
\dk^{\frac{n+2}{2}} \frac{ |\Lg T_k + \sigma|_g^2 - |\Lg T + \sigma|_g^2 }{\left( u + W_k + v_k \right)^{2^*+1}}(x_k) \longrightarrow 0 \textrm{ in } C^0_{loc}(\RR^n).
\een
The Laplacian of $Z_{i,k}$ is computed using \eqref{lapZik} below, and with \eqref{C01step1} we also obtain that:
\ben \label{C01step25}
\sum_{i=0}^n \lki \dk^{\frac{n+2}{2}} \left( \triangle_g + h \right) Z_{i,k}(x_k) \longrightarrow 0 \textrm{ in } L^\infty_{loc}(\RR^n).
\een
By definition of $u_k$ in \eqref{defuk} and $\tilde{u}_k$ in \eqref{C01step21} there holds, for any $x \in \RR^n$, that: 
\[ \lim_{r \to 0} \limsup_{k \to +\infty} \int_{B_x(r)} \tilde{u}_k^{2^*} dv_{g_k} = 0, \]
and therefore an adaptation of Trudinger's standard argument (see for instance Hebey \cite{HebeyZLAM}, theorem $2.15$) along with \eqref{C01step22}, \eqref{C01step23}, \eqref{C01step24} and \eqref{C01step25} shows that $\tilde{u}_k$ converges strongly in $C^1_{loc}(\RR^n)$. By \eqref{propphik2} and since $\ve_k \to 0$, there holds
\[ \int_K \dk^{n} \left |\phi_k \left(  \exp_{\xi_k}^{g_{\xi_k}} (\dk x) \right) \right|^{2^*} dv_{g_k} = \int_{ \exp_{\xi_k}^{g_{\xi_k}} (\dk K)} |\phi_k|^{2^*} dv_g = o(1), \]
so that using the explicit expression \eqref{defuk} of $u_k$, we see that  \eqref{C01step2} holds true.

\medskip

\noindent \textbf{Step $3$: A lower-bound on $\phi_k$.}
We now show that for any sequence $(x_k)_k$ of points of $M$ there holds:
\ben \label{C01step3}
\phi_k(x_k) \ge o \big( u(x_k) \big) + o \big( W_k(x_k) \big).
\een
Let $G$ be the Green's function of the operator $\triangle_g + h $ in $M$ and let $(x_k)_k$ be a sequence of points in $M$. First, by the definition of $\rho$ in \eqref{defeta},  since $\vek \to 0$, by \eqref{propphik2} and Fatou's lemma, we have:
\ben \label{C01step31}
\bal
\int_M &\frac{|\Lg T + \sigma|_g^2 + \pi^2}{\rho \left( u + W_k + \phi_k\right)^{2^*+1}}(y) G(x_k,y) dv_g(y) \\
& \ge \int_M \frac{|\Lg T + \sigma|_g^2 + \pi^2}{u^{2^*+1}}(y) G(x_k,y) dv_g(y) + o(1)
\eal
\een
as $k \to + \infty$. Let $(A_k)_k$ denote some sequence of real numbers such that $A_k \to + \infty$ as $k \to \infty$ and such that $A_k \delta_k \to 0$ as $k \to \infty$. 
By \eqref{bulle} and \eqref{propphik2}, $W_k + \phi_k \to 0$ almost everywhere in $M \backslash B_{x_k}(A_k \delta_k)$ and thus Fatou's lemma gives again that: 
\ben \label{C01step32}
\bal
\int_{M \backslash B_{x_k}(A_k\dk)} f & \left( u + W_k + \phi_k \right)^{2^*-1}(y) G(x_k,y) dv_g(y) \\
&\ge  \int_{M} f u^{2^*-1}(y) G(x_k,y) dv_g(y) + o(1).
\eal
\een
Using \eqref{bulle}, \eqref{propphik2}, the local convergence given in \eqref{C01step2} and standard properties of Green's function (see e.g. Robert \cite{RobDirichlet}), Fatou's lemma also shows that:
\ben \label{C01step33}
\int_{B_{x_k}(A_k \dk)} f \left( u + W_k + \phi_k \right)^{2^*-1}(y) G(x_k,y) dv_g(y) \ge \big( 1 + o(1) \big) W_k(x_k) + o(1)
\een
as $k\to +\infty$ (see for instance Hebey \cite{HebeyZLAM}, proposition $6.1$). Independently, using the pointwise estimate given by \eqref{estLTkLT} below, the definition of $X$ in \eqref{defXintro}, \eqref{relationrkmk}, the fact that $\vek \to 0$ and standard properties of Green's function, one obtains that:
\ben \label{C01step34}
\int_M  \frac{ |\Lg T_k + \sigma|_g^2 - |\Lg T + \sigma|_g^2 }{\left( u + W_k + v_k \right)^{2^*+1}}(y) G(x_k,y) dv_g(y) = o(1)
\een
as $k \to +\infty$. Having in mind that $|Z_{i,k}| \lesssim W_k$ for any $0 \le i \le n$, it remains to write a Green's representation formula for the operator $\triangle_g +h $ in $M$ and to use \eqref{eqphikgros}, \eqref{C01step1}, \eqref{C01step31}, \eqref{C01step32}, \eqref{C01step33} and \eqref{C01step34} to obtain that:
\ben \label{C01step35}
\bal
u_k(x_k) \ge  \int_{M} f u^{2^*-1}(y) G(x_k,y) dv_g(y) +  \int_M \frac{|\Lg T + \sigma|_g^2 + \pi^2}{u^{2^*+1}}(y) G(x_k,y) dv_g(y) \\
+ \big( 1 + o(1) \big) W_k(x_k) + o(1).
\eal
\een
Since $u$ solves the scalar equation of \eqref{syseta}, a Green's representation formula for $u$ with \eqref{defuk} and  \eqref{C01step35} then concludes the proof of \eqref{C01step3}. 

\medskip
\noindent Note that, by the definition of $\rho$ in \eqref{defeta}, \eqref{C01step3} shows in particular that $u_k$ in \eqref{defuk} actually satisfies:
\ben \label{C01eqphik}
\bal
 \left( \triangle_g + h \right)&u_k  - f u_k^{2^*-1}  - \frac{|\Lg T + \sigma|_g^2 + \pi^2}{ u_k^{2^*+1}} \\
& -  \frac{ |\Lg T_k + \sigma|_g^2 - |\Lg T + \sigma|_g^2 }{\left( u + W_k + v_k \right)^{2^*+1}}  = \sum_{i=0}^n \lki \left( \triangle_g + h \right) Z_{i,k}.
\eal
\een

\noindent \textbf{Step $4$: A global weak estimate on $\phi_k$.}
We prove now that 
\ben \label{C01step4}
\tk(x)^\pui \left| \phi_k (x)\right| \longrightarrow 0 \textrm{ in } L^\infty(M)
\een 
as $k \to +\infty$, where we have let, for any $x \in M$:
\ben \label{defthetak}
\tk(x) = \dk + d_{g_{\xi_k}}(\xi_k, x).
\een
The proof of \eqref{C01step4} proceeds by contradiction: assume that there exists a sequence $(x_k)_k$ in $M$ such that 
\ben \label{C01step41}
\tk(x_k)^2 | \phi_k (x_k)|^{2^*-2} = \max_{x \in M} \tk(x)^2 |\phi_k(x)|^{2^*-2} \ge \ve_0
\een 
for some $\ve_0 > 0$. We start by noticing that there holds:
\ben \label{C01step43}
\tk(x_k)^2 W_k(x_k)^{2^*-2} = o(1)
\een
and 
\ben \label{C01step44}
u_k(x_k) \to + \infty
\een
as $k \to +\infty$, where $x_k$ is defined in \eqref{C01step4}. Equations \eqref{C01step43} and \eqref{C01step44} follow from a straightforward adaptation of the arguments in Hebey \cite{HebeyZLAM} (proposition $7.1$) combined with \eqref{C01step1}, \eqref{C01step2}, \eqref{C01step3}, \eqref{estLTkLT} and \eqref{C01eqphik}. We let in what follows
\be 
\tdk = u_k(x_k)^{- \frac{2}{n-2}}.
\ee
Equation \eqref{C01step44} implies in particular that $\tdk \to 0$ as $k \to +\infty$. For any $x \in B_0(i_g(M)/ \tdk)$, we let:
\ben \label{C01step46}
\tvk(x) = \tdk^\pui u_k \left( \exp_{x_k}^{g_{x_k}} (\tdk x )\right).
\een
In case $\dk = o(\tdk)$ and $d_{g_{x_k}}(x_k, \xi_k) = O(\tdk)$, let $\mathcal{S} = \{ \tilde{\xi}_0 \}$, where $\tilde{\xi}_0 = \lim_{k \to +\infty} \frac{1}{\tdk} \left( \exp_{x_k}^{g_{x_k}}\right)^{-1}(\xi_k)$. Otherwise, let $\mathcal{S} = \emptyset$. Let $K \subset \subset \RR^n \backslash \mathcal{S}$ be a compact set. For any $z \in K$, let $z_k = \exp_{x_k}^{g_{x_k}} \left( \tdk x \right)$. There holds in particular that
\ben \label{choixdek}
d_{g_{\xi_k}}(\xi_k, z_k) \ge C_0 \tdk
\een
for some positive constant $C_0$. Using \eqref{C01step41}, \eqref{C01step43} and \eqref{C01step46} one easily obtains that for any  $x \in K$,
\ben \label{C01step47}
\left| \tvk(x) -  \tdk^{\pui} W_k(z_k) \right|^{2^*-2} = O(1),
\een
and the constant in the $O(1)$ term obviously depends on $K$. We now claim that 
\ben \label{C01step48}
\tdk^\pui W_k(z_k) = o(1).
\een
By \eqref{bulle} and \eqref{choixdek}, the only case where \eqref{C01step48} is not clearly satisfied is when $\frac{1}{C} \dk \le \tdk \le C \dk$ and $d_{g_{\xi_k}} (\xi_k, z_k) \le C \dk$ for some positive constant $C$. In this case there also holds that $d_{g_{\xi_k}}(\xi_k, x_k) = O (\dk)$, hence $\liminf_{k \to + \infty} \tdk^\pui W_k(x_k) > 0$ by definition of $W_k$, which contradicts \eqref{C01step43}. Hence \eqref{C01step48} holds true. Coming back to \eqref{C01step47} with \eqref{C01step48} we have in particular:
\ben \label{C01step49}
\tvk(x) \le C_K \textrm{ for all } x \in K.
\een
By construction of $\tvk$ in \eqref{C01step46} there holds $\tvk(0) = 1$. As an easy consequence of \eqref{C01step41} and of the definition of $\mathcal{S}$ above one always has $|\tilde{\xi}_0| >0$ whenever $\tilde{\xi}_0$ is finite. Hence \eqref{C01step48} can be applied to some compact subset $K \ni 0$ and yields
\ben \label{C01step410}
\tdk^\pui W_k(x_k) = o(1)
\een 
as $k \to + \infty$ which, combined with \eqref{C01step41}, also gives:
\ben \label{C01step411}
\tk(x_k) \ge \frac{1}{C} \tdk
\een
for some positive $C$ independent of $k$, where $\tk$ is as in \eqref{defthetak}. Let now $0 \in K \subset \subset \RR^n \backslash \mathcal{S}$. By \eqref{C01eqphik} and \eqref{C01step46}, $\tvk$ satisfies, for any $x \in K$:
\ben \label{C01step412}
\bal
 \triangle_{\tilde{g}_k} \tvk(x) & + \tdk^2 h(z_k) \tvk(x) = f(z_k) \tvk(x)^{2^*-1} + \tdk^{2n}  \frac{\left( |\Lg T + \sigma|_g^2 + \pi^2\right)(z_k)}{ \tvk^{2^*+1}(x)} \\
&+ \tdk^{\frac{n+2}{2}}  \frac{ |\Lg T_k + \sigma|_g^2 - |\Lg T + \sigma|_g^2 }{\left( u + W_k + v_k \right)^{2^*+1}}(z_k) + \tdk^{\frac{n+2}{2}} \sum_{i=0}^n \lki \left( \triangle_g + h \right) Z_{i,k}(z_k),
\eal
\een
where again $z_k = \exp_{x_k}^{g_{x_k}} (\tdk x)$ and where $\tilde{g}_k = \left( \exp_{x_k}^{g_{x_k}} \right)^* g(\tdk \cdot)$. Using estimate \eqref{estLTkLT} below, \eqref{defXintro} and using the pointwise estimates on $v_k$ given by the choice $v_k \in F_k = F(\vek, t_k, \xi_k)$ it is easily seen that there holds
\ben \label{C01step413}
 \tdk^{\frac{n+2}{2}}  \frac{ |\Lg T_k + \sigma|_g^2 - |\Lg T + \sigma|_g^2 }{\left( u + W_k + v_k \right)^{2^*+1}}(z_k) = O \left( \tdk^{\frac{n+2}{2}} \right) \textrm{ for any } x \in K.
\een 
Combining \eqref{C01step48} with the pointwise expression of $(\triangle_h + h) Z_{i,k}$ given by \eqref{lapZik} below, the inequality $|Z_{i,k}| \lesssim W_k$ for $0 \le i \le n$ and with \eqref{relationrkmk} yields:
\ben \label{C01step414}
\tdk^{\frac{n+2}{2}} \sum_{i=0}^n \lki \left( \triangle_g + h \right) Z_{i,k}(z_k) = o(1) \textrm{ uniformly in } K.
\een
In the end, \eqref{C01step49}, \eqref{C01step413} and \eqref{C01step414} give with \eqref{C01step412} that $\tvk$ satisfies in $K$:
\ben \label{C01step415}
 \triangle_{\tilde{g}_k} \tvk(x) + \tdk^2 h(z_k) \tvk(x) = f(z_k) \tvk(x)^{2^*-1} + \frac{o(1)}{\tvk^{2^*+1}},
\een
where the term $o(1)$ is uniform in $K$. With \eqref{C01step49} and \eqref{C01step415}, the Harnack inequality for the Einstein-Lichnerowicz equation stated in Premoselli \cite{Premoselli4} (Proposition $6.1$) then shows that there exists a positive constant $C_K$ such that 
\ben \label{C01step416}
\frac{1}{C_K} \le \tvk \le C_K \textrm{ in } K
\een
and thus, by standard elliptic theory, that $\tvk$ converges in $C^2_{loc} (\RR^n \backslash \mathcal{S})$ towards a positive solution $\tilde{w}_0$ of 
\[ \triangle_\xi \tilde{w}_0 = f(x_0) \tilde{w}_0^{2^*-1} \textrm{ in } \RR^n \backslash \mathcal{S},\]
where $x_0 = \lim_{k \to + \infty} x_k$.  In particular, \eqref{C01step416} shows that 
\ben \label{C01step417}
\int_K \tilde{w}_0^{2^*} dy > 0 \textrm{ for any compact set } K \subset \subset \RR^n \backslash \mathcal{S}.
\een
But independently, by \eqref{C01step46}, \eqref{defuk}, \eqref{C01step48} and \eqref{propphik2}, one has that $\tvk \to 0$ in $L^{2^*}(K)$ as $k \to + \infty$, which is a contradiction with \eqref{C01step417}. This concludes the proof of \eqref{C01step4}.

\medskip

\noindent In the following we let, for any $\delta >0$:
\ben \label{defetadelta}
\nu_k(\delta) = \sup_{M \backslash B_{\xi_k}(\delta)} u_k.
\een

\noindent \textbf{Step $5$: A first set of strong pointwise estimates.}
We show that for any $\ve >0$, there exist $R_\ve >0$, $\delta_\ve >0 $ and $C_\ve > 0$ such that:
\ben \label{C01step5}
u_k(x) \le C_\ve \left( \dk^{\pui (1 - 2\ve)} d_{g_{\xi_k}}(\xi_k,x)^{(2-n)(1-\ve)} + \nu_k(\delta_\ve) d_{g_{\xi_k}}(\xi_k,x)^{(2-n)\ve} \right),
\een
for any $k$ and for any $x \in M \backslash B_{\xi_k}(R_\ve \dk)$. For $\ve >0$, we define the following function in $M$:
\ben \label{C01step51}
\Psi_{k,\ve}(x) = \dk^{\pui (1 - 2\ve)}  G (\xi_k,x)^{1-\ve} + \nu_k(\delta_\ve) G(\xi_k,x)^\ve,
\een
where $G$ is the Green's function of $\triangle_g + h$ in $M$. We let $R > 0$ and let $(x_k)_k$ be defined by:
\ben \label{C01step52}
\frac{u_k}{\Psi_{k,\ve}}(x_k) = \sup_{M \backslash B_{\xi_k}(R \dk)} \frac{u_k}{\Psi_{k,\ve}}.
\een
We now claim that, up to choosing $R$ large enough and $\delta$ small enough, there holds, for $k >> 1$, that:
\ben \label{C01step53}
x_k \in \partial B_{\xi_k}(R \dk) \textrm{ or } d_{g_{\xi_k}}(\xi_k,x_k) \ge \delta.
\een
The proof of \eqref{C01step53} proceeds by contradiction: if we assume that \eqref{C01step53} does not hold, then \eqref{C01step52} gives that:
\ben \label{C01step54}
d_{g_{\xi_k}}(\xi_k,x_k)^2 \frac{\triangle_g \Psi_{k,\ve}}{\Psi_{k,\ve}}(x_k) \le d_{g_{\xi_k}}(\xi_k,x_k)^2 \frac{\triangle_g u_k}{u_k} (x_k).
\een
Straightforward computations show that
\ben \label{C01step55}
d_{g_{\xi_k}}(\xi_k,x_k)^2 \frac{\triangle_g \Psi_{k,\ve}}{\Psi_{k,\ve}}(x_k) \ge - (1 - \ve) (1 - C \ve) d_{g_{\xi_k}}(\xi_k,x_k)^2 + C \ve(1 - \ve)
\een
for some positive constant $C$ independent of $k$. Independently, using \eqref{C01eqphik} we have:
\ben \label{C01step56} 
\bal
& d_{g_{\xi_k}}(\xi_k,x_k)^2 \frac{\triangle_g u_k}{u_k} (x_k) = - d_{g_{\xi_k}}(\xi_k,x_k)^2 h(x_k) + d_{g_{\xi_k}}(\xi_k,x_k)^2 f(x_k) u_k^{2^*-2}(x_k)  \\
& + d_{g_{\xi_k}}(\xi_k,x_k)^2  \frac{|\Lg T + \sigma|_g^2 + \pi^2}{ u_k^{2^*+2}}(x_k) +  d_{g_{\xi_k}}(\xi_k,x_k)^2 \frac{ |\Lg T_k + \sigma|_g^2 - |\Lg T + \sigma|_g^2 }{u_k \left( u + W_k + v_k \right)^{2^*+1}}(x_k)  \\
& + d_{g_{\xi_k}}(\xi_k,x_k)^2 \frac{1}{u_k}\sum_{i=0}^n \lki \left( \triangle_g + h \right) Z_{i,k}(x_k).
\eal
\een
By \eqref{defuk}, \eqref{C01step4} and since we assumed that \eqref{C01step53} does not hold, we have that
\ben \label{C01step57}
d_{g_{\xi_k}}(\xi_k,x_k)^2 f(x_k) u_k^{2^*-2}(x_k) \le o(1) + \Vert u \Vert_{L^\infty(M)}^{2^*-2} \delta^2 + C \left( 1 + \frac{f(\xi_k)}{n(n-2)} R^2 \right)^{-2}
\een
for some positive constant $C$ independent of $k$. Similarly, using in addition \eqref{C01step3}:
\ben \label{C01step58}
d_{g_{\xi_k}}(\xi_k,x_k)^2  \frac{|\Lg T + \sigma|_g^2 + \pi^2}{ u_k^{2^*+2}}(x_k) \le C' \delta^2
\een
for some positive $C'$ independent of $k$. Here, $(u,T)$ are always as in \eqref{syseta}. Using the pointwise estimates on $v_k$ given by \eqref{defFea}, \eqref{defXintro}, \eqref{C01step3}, \eqref{relationrkmk} and \eqref{estLTkLT} below there holds:
\ben \label{C01step59}
 \bal
& d_{g_{\xi_k}}(\xi_k,x_k)^2 \frac{ |\Lg T_k + \sigma|_g^2 - |\Lg T + \sigma|_g^2 }{u_k \left( u + W_k + v_k \right)^{2^*+1}} \\
 & \qquad \qquad \le \left \{
\bal
& O(\dk) \textrm{ if } d_{g_{\xi_k}}(\xi_k,x_k) \le \sqrt{\dk} \\
& C(n,g,u_0) \delta ^2 + o(1) \textrm{ if } d_{g_{\xi_k}}(\xi_k,x_k) \ge \sqrt{\dk}. \\
\eal \right.
\eal
\een
Finally, since $|Z_{i,k}| \lesssim W_k$,
by \eqref{lapZik} and \eqref{relationrkmk} one gets, mimicking the proof of \eqref{C01step57}, that:
\ben \label{C01step510}
\bal
& d_{g_{\xi_k}}(\xi_k,x_k)^2 \frac{1}{u_k}\sum_{i=0}^n \lki \left( \triangle_g + h \right) Z_{i,k}(x_k) \\
& \qquad \le C(n,g,u_0,h) \delta^2 +  C \left( 1 + \frac{f(\xi_k)}{n(n-2)} R^2 \right)^{-2} + o(1).
\eal
\een
For a fixed value of $\ve$, \eqref{C01step54}, \eqref{C01step57}, \eqref{C01step58}, \eqref{C01step59} and \eqref{C01step510} give a contradiction with \eqref{C01step55} up to choosing $R$ large enough and $\delta$ small enough. This shows that  \eqref{C01step53} holds. Then \eqref{C01step5} follows from \eqref{C01step51} and the local convergence given by \eqref{C01step2}.

\medskip

\noindent \textbf{Step $6$: End of the proof of \eqref{estC0grossiere}.} We show that there exists a sequence $\nu_k$ of positive real numbers converging to $0$ as $k \to + \infty$ such that for any $x \in M$ there holds:
\ben \label{C01step6}
|\phi_k(x)| \le \nu_k \big( W_k(x) + u(x) \big).
\een
Let $(x_k)_k$ be a sequence of points in $M$. We prove in what follows that there holds:
\ben \label{C01step61}
|\phi_k(x_k)| = o \big(W_k(x_k) \big) + o(1).
\een
Since $u > 0$ in $M$, \eqref{C01step6} follows from \eqref{C01step61}. Assume first that $d_{g_{\xi_k}}(\xi_k,x_k) = O(\dk)$. Then \eqref{C01step61} follows from \eqref{C01step2}. Assume then that $d_{g_{\xi_k}}(\xi_k,x_k) \not \to 0$ as $k \to + \infty$. Then \eqref{C01step61} follows from \eqref{C01step4}. We may therefore assume that
\ben \label{C01step61bis}
\dk << d_{g_{\xi_k}}(\xi_k,x_k) << 1 .
\een
Let $G$ denote again the Green's function of $\triangle_g + h$. By \eqref{C01step1} and since $|Z_{i,k}(x_k)| \lesssim W_k(x_k)$ for $0 \le i \le n$ we easily get that
\ben \label{C01step62}
\int_M G(x_k,y) \sum_{i=0}^n \lki \left( \triangle_g + h \right) Z_{i,k}(y) dv_g (y) = \sum_{i=0}^n \lki Z_{i,k}(x_k) = o \big( W_k(x_k) \big).
\een
Independently, since $\Big( \left( |\Lg T + \sigma|_g^2 + \pi^2\right) u_k^{-2^*-1} \Big)_k$ is uniformly bounded in $L^\infty(M)$ by \eqref{C01step3}, we have that:
\ben \label{C01step63}
\int_M G(x_k,y)  \frac{|\Lg T + \sigma|_g^2 + \pi^2}{ u_k^{2^*+1}}(y) dv_g(y) = \int_M G(x_k,y)  \frac{|\Lg T + \sigma|_g^2 + \pi^2}{ u^{2^*+1}}(y) dv_g(y) + o(1)
\een
as $k \to +\infty$. Now for some fixed $0 < \ve < \frac{2}{n+2}$ we write that
\[
\bal
&  \int_M G(x_k,y) f(y) u_k^{2^*-1}(y) dv_g(y)  \\
 & = \int_{B_{\xi_k}(R_\ve \dk)} G(x_k,y) f (y) u_k^{2^*-1}(y) dv_g(y) + \int_{M \backslash B_{\xi_k}(R_\ve \dk) } G(x_k,y) f(y) u_k^{2^*-1}(y) dv_g(y) .
 \eal
 \]
 On the one side, \eqref{C01step2} along with \eqref{C01step61bis} shows that
 \ben \label{C01step64}
    \int_{B_{\xi_k}(R_\ve \dk)} G(x_k,y) f(y) u_k^{2^*-1}(y) dv_g(y) = O \big( W_k(x_k) \big) + o(1).
 \een
On the other side, using \eqref{C01step5} we obtain that:
\ben \label{C01step65}
 \int_{M \backslash B_{\xi_k}(R_\ve \dk) } G(x_k,y) f(y) u_k^{2^*-1}(y) dv_g(y) = O \big( W_k(x_k) \big) + O(1)
\een
as $k \to + \infty$. Gathering \eqref{C01step34}, \eqref{C01step62}, \eqref{C01step63}, \eqref{C01step64}  and \eqref{C01step65} and writing a representation formula for $u_k$ gives, with \eqref{C01eqphik}, that :
\ben \label{C01step66}
|\phi_k(x_k)| \le C \big( W_k(x_k) + u(x_k) \big),
\een
for some positive constant $C$ that neither depends on $k$ nor on $\eta$ in \eqref{introeta}. In particular, $C$ in \eqref{C01step66} does not depend on the choice of $(t_k,\xi_k)_k$ and $(v_k)_k$. It remains to improve \eqref{C01step66} into \eqref{C01step61}. Using the expression of the conformal laplacian of $W_k$ given by \eqref{lapbulle} below, and since $f \in C^1(M)$, the following representation formula holds true for $W_k$:
\ben \label{C01step67}
W_k(x_k) = \int_M G(x_k,y) f(\xi_k) W_k^{2^*-1}(y) dv_g(y) + o \big( W_k(x_k) \big) + o(1), 
\een
where to obtain \eqref{C01step67} we used that there holds, by \eqref{C01step61bis} and Giraud's lemma (see \cite{HebeyZLAM}, lemma $7.5$):
\ben \label{C01step68}
\int_M G(x_k,y) W_k(y) dv_g(y)  = O \left( \dk^\pui \tk(x_k)^{4-n} \right) = O \big( \dk W_k(x_k) \big) + O(\dk),
\een
where $\tk$ is defined in \eqref{defthetak}. We now write a representation formula for $u_k - W_k - u$, where $u_k$ is defined in \eqref{defuk}. Since $u$ solves \eqref{syseta}, by \eqref{C01eqphik}, \eqref{C01step62}, \eqref{C01step63}, \eqref{C01step34} and \eqref{C01step67} there holds:
\ben \label{C01step69}
\phi_k(x_k) = \int_M G(x_k,y) f(y) \left( u_k^{2^*-1} - W_k^{2^*-1} - u^{2^*-1}\right)(y)  dv_g(y) + o \big( W_k(x_k) \big) + o(1).
\een
By \eqref{C01step2} there exists a sequence $A_k \to + \infty$ such that
\[ \left| \left| \dk^\pui u_k \left( \exp_{\xi_k}^{g_{\xi_k}}(\dk \cdot) \right) - U_{\xi_0} \right| \right|_{C^0(B_0(A_k))} \to 0 \]
as $k \to +\infty$, where $U_{\xi_0}$ is defined in \eqref{eqlindefU}. We can always choose such a sequence $(A_k)_k$ to have $A_k \dk = o(\sqrt{\dk})$, so that there holds, by the dominated convergence theorem:
\ben \label{C01step610}
\int_{B_{\xi_k}(A_k \dk)} G(x_k,y) f(y) \left( u_k^{2^*-1} - W_k^{2^*-1} - u^{2^*-1}\right)(y)  dv_g(y) = o \big( W_k(x_k) \big) + o(1).
\een
By \eqref{bulle} and \eqref{C01step66} there holds that
\[ \left| u_k^{2^*-1} - W_k^{2^*-1} - u^{2^*-1}\right| \lesssim W_k + |\phi_k| \textrm{ in } M \backslash B_{\xi_k}(\sqrt{\dk}). \]
Since by \eqref{C01step66} $\phi_k$ is uniformly bounded in $M \backslash B_{\xi_k}(\sqrt{\dk})$ we obtain with \eqref{propphik2}, \eqref{C01step68} and since $\vek = o(1)$ that 
\ben \label{C01step610bis}
\int_{M \backslash B_{\xi_k}(\sqrt{\dk})} G(x_k,y) f(y) \left( u_k^{2^*-1} - W_k^{2^*-1} - u^{2^*-1}\right)(y)  dv_g(y) = o(1) + o(W_k(x_k)).
\een
By \eqref{C01step66} we can write that there holds in $B_{\xi_k}(\sqrt{\dk})$:
\be
\left| u_k^{2^*-1} - W_k^{2^*-1} - u^{2^*-1}\right| \lesssim  W_k^{2^*-1}.
\ee 
Since $A_k \to + \infty$ as $k \to +\infty$ straightforward computations therefore show that:
\ben \label{C01step612}
\int_{B_{\xi_k}(\sqrt{\dk}) \backslash B_{\xi_k}(A_k \dk)} G(x_k,y) f(y) \left( u_k^{2^*-1} - W_k^{2^*-1} - u^{2^*-1}\right)(y)  dv_g(y) = o \big( W_k(x_k) \big) + o(1).
\een
Equation \eqref{C01step61} then follows from \eqref{C01step69}, \eqref{C01step610}, \eqref{C01step610bis} and \eqref{C01step612}.
\end{proof}

\noindent As a consequence of Proposition \ref{propC0grossier} we also obtain pointwise bounds on the gradient of $\phi_k$. More precisely there holds:
\ben  \label{estC0grossieregrad}
|\nabla \phi_k(x)| \lesssim \Big( 1 + \dk^{\pui} \tk(x)^{1-n} \Big),
\een
where $\tk$ is defined in \eqref{defthetak}. Indeed, using Proposition \ref{propC0grossier}, \eqref{defXintro} and \eqref{estLTkLT}, equation \eqref{C01eqphik} can be written as:
\[ \triangle_g \phi_k + h \phi_k = O\big(W_k^{2^*-1}\big) + O(1),\]
so that writing again a representation formula for $\phi_k$ and differentiating yields easily \eqref{estC0grossieregrad}.

\section{Second-order estimates on $\phi_k$.} \label{theorieC0ordre2}

\noindent Let $D >0$ and $(\vek)_k \in \mathcal{E}$. Assume that  $\vek >> \mk^{\frac{3}{2}}$ as $k \to + \infty$, where $\mk$ is defined in \eqref{defmk}. Let $(t_k,\xi_k)_k \in [1/D, D]\times M$, for any $k$ let $v_k \in F_k = F \big( (\vek)_k,t_k,\xi_k \big)$, and let $\phi_k = \phi_k \big(t_k,\xi_k,v_k \big) $ be given by Proposition \ref{propptfixe1}. 

\medskip

\noindent 
The asymptotic pointwise estimates on $\phi_k$ obtained in Proposition \ref{propC0grossier} are not accurate enough to even say that $\phi_k$ belongs to $F(\vek,t,\xi)$. In view of the final fixed-point argument we need to control $\nu_k$ introduced in \eqref{estC0grossiere} only in terms of $\vek$ and uniformly in the choice of $(t_k,\xi_k,v_k)_k$. This task is achieved in this section and in the following one.

\medskip

\noindent In this Section we perform a second-order pointwise expansion of $u_k$ defined in \eqref{defuk} and obtain finer pointwise estimates on $\phi_k$. 
\medskip

\noindent We keep using the notations $f_k = O(g_k)$, $f_k = o(g_k)$ and $f_k \lesssim g_k$ introduced in Section \ref{partie4}. For the sake of clarity, we will also use the following notational shorthand: if $f \in L^\infty(M)$, $\Vert f \Vert_{L^\infty(2 r_k)}$ will be used to denote the quantity $\Vert f \Vert_{L^\infty(B_{\xi_k}(2 r_k))}$. Also, the notation $\mathds{1}_{nlcf}$ will be used to denote a term which only appears when the manifold $M$ is non-locally conformally flat.

\medskip

\noindent We first obtain refined estimates on $\phi_k$ when the bubble $W_{k}$ is the dominant term in the $C^0$-decomposition of $u_k$:

\begin{prop} \label{C02pres}
Let $D >0$, $(\vek)_k \in \mathcal{E}$ and assume that  $\vek >> \mk^{\frac{3}{2}}$ as $k \to + \infty$, where $\mk$ is defined in \eqref{defmk}. Let $(t_k,\xi_k)_k$ be a sequence in $ [1/D, D]\times M$, let $v_k \in F_k = F \big( \vek,t_k,\xi_k \big)$ and let $\phi_k = \phi_k \big(t_k,\xi_k,v_k \big)$ be given by Proposition \ref{propptfixe1}. 
Let $(x_k)_k$ be any sequence of points in $B_{\xi_k}(2 \sqrt{\dk})$. There holds:
\begin{itemize}
\item If $n \ge 7$:
\ben \label{estC02pres}
\bal
\tk(x_k) \left| \nabla \phi_k (x_k) \right|  &+ |\phi_k(x_k)|  \lesssim  \Vert \phi_k \Vert_{L^\infty(\Omega_k)} +  \sqrt{\dk} \Vert \nabla \phi_k \Vert_{L^\infty(\Omega_k)} 
  + \dk  \\
& + \Bigg[  
 \dk^{\frac{n}{2}}
 + \dk \Vert \nabla f \Vert_{L^\infty(2 r_k)}  + \Vert h - c_n S_g \Vert_{L^\infty(2 r_k)} \dk^2  \left| \ln{ \left( \frac{\tk(x_k)}{\dk} \right)} \right| \\
 & + \Vert h - c_n S_g \Vert_{L^\infty(2 r_k)} \tk(x_k)^2 + \tk(x_k)^4 \mathds{1}_{nlcf} \Bigg]  W_k(x_k) + \left( \frac{\dk}{\tk(x_k)}\right)^2 .\\
\eal
\een
\item If $n = 6$:
\ben \label{estC02pres2}
\bal
& \tk(x_k) \left| \nabla \phi_k (x_k) \right| + |\phi_k(x_k)|  \lesssim   \Vert \phi_k \Vert_{L^\infty(\Omega_k)} +  \sqrt{\dk} \Vert \nabla \phi_k \Vert_{L^\infty(\Omega_k)}  +\dk \\
& + \Bigg[ \dk^3 +  \dk \Vert \nabla f \Vert_{L^\infty(2 r_k)} + \big \Vert h - \frac15 S_g - 2 f u \big \Vert_{L^\infty(2 r_k)} \Big(  \dk^2 \left| \ln{ \left( \frac{\tk(x_k)}{\dk} \right)}\right| +  \tk(x_k)^2 \Big)  \Bigg] W_k(x_k) .\\
\eal
\een
\end{itemize}
In \eqref{estC02pres} and \eqref{estC02pres2} we have let:
\ben \label{defOmegak}
\Omega_k = B_{\xi_k}(2 r_k) \backslash B_{\xi_k}(\sqrt{\dk}),
\een
and $\tk(x_k)$ is as in \eqref{defthetak}.
\end{prop}

\begin{proof}
Let $D >0$ and $(\vek)_k \in \mathcal{E}$ and assume that  $\vek >> \mk^{\frac{3}{2}}$ as $k \to + \infty$, where $\mk$ is defined in \eqref{defmk}. Let $(t_k,\xi_k)_k$ be a sequence in $ [1/D, D]\times M$, let $v_k \in F_k = F \big( \vek,t_k,\xi_k \big)$ and let $\phi_k = \phi_k \big(t_k,\xi_k,v_k \big)$ be given by Proposition \ref{propptfixe1}.
Proposition \ref{propC0grossier} applies so that $u_k$ given by \eqref{defuk} solves \eqref{vraieeqphik}. As before, we shall omit the dependence in $t_k$ and $\xi_k$ and let $W_k = W_{k,t_k,\xi_k}$ and so on.  We first obtain pointwise estimates on $\phi_k$ that we later improve into gradient estimates.

\medskip
\noindent \textbf{We first assume that $n \ge 7$.} From \eqref{vraieeqphik}, using \eqref{syseta} and \eqref{lapbulle} below, $\phi_k$ is easily seen to satisfy:
\ben \label{eqphik}
\bal
 \left( \triangle_g + h \right) \left(  \phi_k - \sum_{i=0}^n \lki Z_{i,k} \right)   &=  f \left( u_k^{2^*-1} - W_k^{2^*-1} - u^{2^*-1} \right) \\
 &+ \big(f - f(\xi_k) \big) W_k^{2^*-1}  - \big(h - c_n S_g \big) W_k  \\ 
&- c_n S_{g_{\xi_k}} \Lambda_{g_{\xi_k}}^{2^*-2} W_k + O (\dk^\pui \mathds{1}_{d_k \le 2 r_k}) + O \big(\dk^\pui r_k^{-n} \mathds{1}_{r_k \le d_k \le 2 r_k} \big) \\ 
 &  + \left(|\Lg T + \sigma|_g^2 + \pi^2 \right) \left( u_k^{-2^*-1} - u^{-2^*-1} \right) \\
& +  \frac{ |\Lg T_k + \sigma|_g^2 - |\Lg T + \sigma|_g^2 }{\left( u + W_k + v_k \right)^{2^*+1}} . \\
\eal
\een
Remember that by \eqref{relationrkmk} there holds that $r_k >> \sqrt{\dk}$ for $t \in [1/D, D]$, where $\dk$ is given by \eqref{defdkyk}. In \eqref{eqphik}  we have let, for any $x \in M$:
\ben \label{notdk}
d_k(x) = d_{g_{\xi_k}}(\xi_k, x).
\een
Also, in \eqref{eqphik}, the notation $\dk^\pui r_k^{-n} \mathds{1}_{r_k \le d_k \le 2 r_k}$ is used to denote a smooth function, supported in $B_{\xi_k}(2 r_k) \backslash B_{\xi_k}(r_k)$, and whose $C^0$ norm is bounded by $\dk^\pui r_k^{-n} $. Let $(x_k)_k $ be a sequence of points in $B_{\xi_k}(2 \sqrt{\dk})$. 
If $x_k \in \Omega_k$ as in \eqref{defOmegak} we clearly have:
\ben \label{C02s11}
|\phi_k(x_k)| \le \Vert \phi_k \Vert_{L^\infty(\Omega_k)}.
\een
Assume now that $x_k \in B_{\xi_k}(\sqrt{\dk})$. As before, we let $G$ denote the Green's function of $\triangle_g + h$ in $M$. A representation formula for $\phi_k - \sum_{i=0}^n \lki Z_{i,k}$ in $B_{\xi_k}(2 \sqrt{\dk})$ gives, using \eqref{relationrkmk} and \eqref{eqphik}, that:
\ben \label{C02s12}
\bal
\phi_k(x_k) &= \sum_{i=0}^n \lki Z_{i,k} (x_k) + O \big( \dk^{\pui} r_k^{2-n}\big) +  O \left( \Vert \phi_k \Vert_{L^\infty(\Omega_k)} \right) \\
&+ O \left( \sqrt{\dk} \Vert \nabla \phi_k \Vert_{L^\infty(\Omega_k)} \right) +  O \left( \sum_{i=0}^n |\lki| \right) \\
&+ I_1 + I_2 + I_3 + I_4 + I_5 + I_6, 
\eal
\een
where we have let:
\ben \label{I1toI6}
\bal
I_1 & = \int_{B_{\xi_k}(2\sqrt{\dk})} f(y) \left( u_k^{2^*-1} - W_k^{2^*-1} - u^{2^*-1} \right)(y) G(x_k,y) dv_g(y), \\
I_2 & = \int_{B_{\xi_k}(2\sqrt{\dk})} \big(f(y) - f(\xi_k) \big) W_k^{2^*-1}(y) G(x_k,y) dv_g(y), \\
I_3 & = - \int_{B_{\xi_k}(2\sqrt{\dk})} \big(h - c_n S_g \big)(y)  W_k(y) G(x_k,y) dv_g(y), \\
I_4 & = - \int_{B_{\xi_k}(2\sqrt{\dk})} c_n S_{g_{\xi_k}}(y) \Lambda_{g_{\xi_k}}^{2^*-2}(y) W_k(y) G(x_k,y) dv_g(y), \\
I_5 & =  \int_{B_{\xi_k}(2\sqrt{\dk})}  \left(|\Lg T + \sigma|_g^2 + \pi^2 \right) \left( u_k^{-2^*-1} - u^{-2^*-1} \right)(y) G(x_k,y) dv_g(y), \textrm{ and } \\
I_6 & =  \int_{B_{\xi_k}(2\sqrt{\dk})}  \frac{ |\Lg T_k + \sigma|_g^2 - |\Lg T + \sigma|_g^2 }{\left( u + W_k + v_k \right)^{2^*+1}}(y) G(x_k,y) dv_g(y).
\eal
\een
The definition of $W_k$ in \eqref{bulle} yields:
\ben \label{C02s1I2}
I_2 \lesssim  \dk \Vert \nabla f \Vert_{L^\infty(2 r_k)} W_k(x_k),
\een
while an application of Giraud's lemma (see \cite{HebeyZLAM}, lemma $7.5$) shows that
\ben \label{C02s1I3}
I_3 \lesssim \Vert h - c_n S_g \Vert_{L^\infty(2 r_k)} \tk(x_k)^2 W_k(x_k)
\een
and, using \eqref{propLambdaSg}, that:
\ben \label{C02s1I4}
I_4 = \left \{
\bal
& 0 &\textrm{ if } (M,g) \textrm{ is locally conformally flat} \\
&O \left( \tk(x_k)^4 W_k(x_k) \right) &\textrm{ if } n \ge 7 \textrm{ and } (M,g) \textrm{ is not l.c.f.} \\
& O \left( \dk^2 |\ln \left( \tk(x_k) \right)| \right) &\textrm{ if } n = 6 \textrm{ and } (M,g) \textrm{ is not l.c.f}, \\
\eal \right.
\een
where $\tk(x)$ is as in \eqref{defthetak}. Since $u_k$ is uniformly bounded from below by \eqref{C01step3} we also have:
\ben \label{C02s1I5}
I_5 \lesssim  \dk.
\een
With \eqref{estLTkLT} below and since $v_k \in F_k$  (defined in \eqref{defFea}), we obtain that:
\ben \label{C02s1I6}
I_6  \lesssim  \dk. 
\een
Note that \eqref{C02s1I2}, \eqref{C02s1I4}, \eqref{C02s1I5} and \eqref{C02s1I6} actually hold even if $n = 6$. By \eqref{estC0grossiere} there holds, in $B_{\xi_k}(\sqrt{\dk})$:
\ben \label{C02s1I10}
\left| u_k^{2^*-1} - W_k^{2^*-1} - u^{2^*-1} \right| \lesssim W_k^{2^*-2} |\phi_k| + W_k^{2^*-2},
\een
so that we have
\ben \label{C02s1I1}
I_1 \lesssim  \left( \frac{\dk}{\tk(x_k)} \right)^2 \left(  \Vert \phi_k \Vert_{L^\infty(B_{\xi_k}(2\sqrt{\dk}))} + 1 \right) .
\een
Gathering the estimates \eqref{C02s1I2} to \eqref{C02s1I1} in \eqref{C02s12} and using \eqref{relationrkmk} gives:
\ben \label{C02s13}
\bal
&\left| \phi_k(x_k) - \sum_{i=0}^n \lki Z_{i,k} (x_k) \right| \lesssim   \Vert \phi_k \Vert_{L^\infty(\Omega_k)} +  \sqrt{\dk} \Vert \nabla \phi_k \Vert_{L^\infty(\Omega_k)} +  \sum_{i=0}^n |\lki| \\
& \qquad + \dk +  \left( \frac{\dk}{\tk(x_k)} \right)^2  \left( 1 + \Vert \phi_k \Vert_{L^\infty(B_{\xi_k}(2\sqrt{\dk}))} \right)\\
&\qquad + \Bigg[ \dk \Vert \nabla f \Vert_{L^\infty(2 r_k)} + \Vert h - c_n S_g \Vert_{L^\infty(2 r_k)} \tk(x_k)^2 + \tk(x_k)^4 \mathds{1}_{nlcf} \Bigg] W_k(x_k)  .\\
\eal
\een
In \eqref{C02s13}, the notation $\mathds{1}_{nlcf}$ is used to denote a term which vanishes when $(M,g)$ is conformally flat. Equation \eqref{C02s13} holds for any sequence $x_k \in B_{\xi_k}(\sqrt{\dk})$. Applying it to a well-chosen set of $(n+1)$ points lying in $B_{\xi_k}(\dk)$ one obtains that:
\ben \label{C02s1estimlki}
\bal
\sum_{i=0}^n |\lki| &\lesssim \dk^\pui \Big( \Vert \phi_k \Vert_{L^\infty(\Omega_k)} +  \sqrt{\dk} \Vert \nabla \phi_k \Vert_{L^\infty(\Omega_k)} \Big)  \\
& + \dk^\pui 
 + \dk \Vert \nabla f \Vert_{L^\infty(2 r_k)} + \Vert h - c_n S_g \Vert_{L^\infty(2 r_k)} \dk^2 \\
& + \dk^\pui \Vert \phi_k \Vert_{L^\infty(B_{\xi_k}(2\sqrt{\dk}))} + \dk^4 \mathds{1}_{nlcf}. \\
\eal
\een
Plugging \eqref{C02s1estimlki} into \eqref{C02s13}, using that $|Z_{i,k}| \lesssim W_k$ and using \eqref{C02s11} gives the improved estimate:
\ben \label{C02s14}
\bal
|\phi_k(x_k)| & \lesssim  \Vert \phi_k \Vert_{L^\infty(\Omega_k)} +  \sqrt{\dk} \Vert \nabla \phi_k \Vert_{L^\infty(\Omega_k)} \\ 
&  + \dk +  \left( 1 +  \Vert \phi_k \Vert_{L^\infty(B_{\xi_k}(2\sqrt{\dk}))} \right) \left( \frac{\dk}{\tk(x_k)} \right)^2 \\ 
& + \Bigg[ \dk \Vert \nabla f \Vert_{L^\infty(2 r_k)} + \Vert h - c_n S_g \Vert_{L^\infty(2 r_k)} \tk(x_k)^2 +  \tk(x_k)^4 \mathds{1}_{nlcf}\Bigg] W_k(x_k),  \\
\eal
\een
which now holds for any sequence $x_k \in B_{\xi_k}(2 \sqrt{\dk})$. 
We now claim that the following holds:

\begin{claim} \label{claimNphik}
There holds
\ben \label{eqclaimNphik}
\Vert \phi_k \Vert_{L^\infty(B_{\xi_k}(2\sqrt{\dk}))} \lesssim \max \left(1, M_k \right),
\een
where we have let
\ben \label{C02s1defMk}
\bal
M_k &=  \Vert \phi_k \Vert_{L^\infty(\Omega_k)} +  \sqrt{\dk} \Vert \nabla \phi_k \Vert_{L^\infty(\Omega_k)} 
 +  \dk  \\ 
& + \dk^{2 - \frac{n}{2}}\Vert \nabla f \Vert_{L^\infty(2 r_k)} + \dk^{3 - \frac{n}{2}} \Vert h - c_n S_g \Vert_{L^\infty(2 r_k)}  + \dk^{5 - \frac{n}{2}} \mathds{1}_{nlcf} ,\\
\eal
\een
and where $\Omega_k$ is as in \eqref{defOmegak}.
\end{claim}

\begin{proof}[Proof of Claim \ref{claimNphik}]

\noindent Let $(x_k)_k$ be a sequence of points of $B_{\xi_k}(2 \sqrt{\dk})$ satisfying $|\phi_k(x_k)| = \max_{B_{\xi_k}(2 \sqrt{\dk})} |\phi_k|$. Claim \ref{claimNphik} is trivially satisfied if $x_k \in \Omega_k$, so we assume in the following that $x_k \in B_{\xi_k}(\sqrt{\dk})$. We proceed by contradiction and assume that there holds:
\ben \label{C02claimcontra}
|\phi_k(x_k)| = \Vert \phi_k \Vert_{L^\infty(B_{\xi_k}(2\sqrt{\dk}))} >> \max \left(1, M_k \right)
\een
as $k \to + \infty$, where $M_k$ is defined in \eqref{C02s1defMk}. Let $(y_k)_k$ be any other sequence of points in $B_{\xi_k}(\sqrt{\dk})$. Assumption \eqref{C02claimcontra} implies in particular that $\Vert \phi_k \Vert_{L^\infty(B_{\xi_k}(2\sqrt{\dk}))} >> 1$, so that applying \eqref{C02s14} at $y_k$ yields, with \eqref{C02claimcontra}:
\ben \label{C02s15}
 |\phi_k(y_k)|   
\lesssim \left(  \left( \frac{\dk}{\tk(y_k)} \right)^2 + o(1) \right) \Vert \phi_k \Vert_{L^\infty(B_{\xi_k}(2\sqrt{\dk}))}. 
\een
Using \eqref{C02s15}, we now compute again $I_1$ in \eqref{I1toI6} and obtain, since $\dk \le \tk(y_k)$:
\[ I_1  \lesssim  \left( \left( \frac{\dk}{\tk(y_k)} \right)^{\frac{7}{2}} + o(1) \right)\Vert \phi_k \Vert_{L^\infty(B_{\xi_k}(2\sqrt{\dk}))} +   \left( \frac{\dk}{\tk(y_k)} \right)^2,  \]
which in turn, with \eqref{C02s12} and \eqref{C02claimcontra} implies that there holds:
\[  |\phi_k(y_k)| \lesssim \left(  \left( \frac{\dk}{\tk(y_k)} \right)^{\frac{7}{2}} + o(1) \right) \Vert \phi_k \Vert_{L^\infty(B_{\xi_k}(2\sqrt{\dk}))}. \]
Replugging the latter estimate in the computation of $I_1$ improves again, and after a finite number of iterations one obtains:
\ben \label{C02s17}
|\phi_k(y_k)| \lesssim  \left(  \left( \frac{\dk}{\tk(y_k)} \right)^{n-2} + o(1) \right) \Vert \phi_k \Vert_{L^\infty(B_{\xi_k}(2\sqrt{\dk}))} \textrm{ for any } y_k \in B_{\xi_k}(\sqrt{\dk}). 
\een
In particular, \eqref{C02s17} applied to the sequence $x_k$ given by \eqref{C02claimcontra} yields:
\be 
\tk(x_k) \lesssim \dk,
\ee
where $\tk$ is as in \eqref{defthetak}. We define now, for $y \in B_0(2\dk^{- \frac12})$:
\be 
\tpk(x) = \Vert \phi_k \Vert_{L^\infty(B_{\xi_k}(2\sqrt{\dk}))}^{-1} \phi_k \left( \exp_{\xi_k}^{g_{\xi_k}} (\dk y )\right),
\ee
and let $\tilde{x}_k = \frac{1}{\dk} \left( \exp_{\xi_k}^{g_{\xi_k}} \right)^{-1}(x_k)$. Hence, $\tilde{x}_k \to \tilde{x}_0 \in \RR^n$ as $k \to + \infty$. We also have $\Vert \tpk \Vert_{L^\infty(2 \dk^{- \frac12})} \le 1$, and using \eqref{eqphik}, \eqref{estC0grossiere}, \eqref{C02claimcontra} and standard elliptic theory, we obtain that $\tpk$ converges in $C^1_{loc}(\RR^n)$, up to a subsequence, to some function $\tpz$ satisfying $|\tpz(\tilde{x}_0)| =1$ and
\ben \label{C02s19}
\triangle_\xi \tpz - (2^*-1) f(\xi_0) U_{\xi_0}^{2^*-2} \tpz = \sum_{i=0}^n \tilde{\lambda}_i \triangle_\xi V_{i,\xi_0},
\een 
where $\xi_0 = \lim_{k \to + \infty} \xi_k$, $U_{\xi_0}$ is defined in \eqref{eqlindefU}, $V_{i,\xi_0}$ is as in \eqref{defVki} and where we have let, up to a subsequence:
\[ \tilde{\lambda}_i = \lim_{k \to + \infty} \frac{\lki}{\dk^\pui \Vert \phi_k \Vert_{L^\infty(B_{\xi_k}(2\sqrt{\dk}))}}. \]
Note that this limit exists, up to a subsequence, by \eqref{C02s1estimlki} and \eqref{C02claimcontra}. Also, in \eqref{C02s19}, $\xi$ stands for the Euclidean metric in $\RR^n$. Passing \eqref{C02s17} to the limit also gives that:
\ben \label{C02s19bis}
|\tpz(y)| \lesssim \left( 1 + |y| \right)^{2-n} \textrm{ for any } y\in \RR^n.
\een
With the latter estimate we can integrate \eqref{C02s19} against $V_{i,\xi_0}$ for any $0 \le i \le n$. Since $V_{i,\xi_0}$ solves \eqref{eqlin} there holds $\tilde{\lambda}_i = 0$ and thus $\tpz$ satisfies:
\ben \label{C02s110}
\triangle_\xi \tpz = (2^*-1) f(\xi_0) U_{\xi_0}^{2^*-2} \tpz . ,
\een 
Again with \eqref{C02s19bis} this shows that $\tpz \in H^1(M)$ and then the Bianchi-Egnell \cite{BianchiEgnell} classification result applies and gives that
\ben \label{C02s111}
\tpz \in \textrm{Span} \{ V_{i,\xi_0}, 0 \le i \le n \}.
\een

\medskip
\noindent To conclude the proof of Claim \ref{claimNphik} we now show that $\tpz \in \textrm{Span} \{ V_{i,\xi}, 0 \le i \le n \}^\perp$. With \eqref{C02s111} this will show that $\tpz \equiv 0$, thus contradicting the fact that $|\tpz(\tilde{x}_0)| = 1$. By \eqref{propphik1}, $\phi_k \in K_k^{\perp}$, where $K_k = K_{k,t_k,\xi_k}$ is defined in \eqref{noyau}. Hence for any $0 \le i \le n$ there holds:
\be 
\int_M \left( \left \langle \nabla Z_{i,k}, \nabla \phi_k \right \rangle_g + h Z_{i,k} \phi_k \right) dv_g = 0.
\ee
Let $R >0$ and $0 \le i \le n$. Integrating by parts the latter equation gives
\ben  \label{C02s112}
\bal
&\int_{B_{\xi_k}(R \dk)} \left \langle \nabla Z_{i,k}, \nabla \phi_k \right \rangle_g + h Z_{i,k} \phi_k dv_g = - \int_{\partial B_{\xi_k}(R \dk)} \phi_k \partial_\nu Z_{i,k} d\sigma_g \\
&- \int_{M \backslash B_{\xi_k}(R \dk)} (h-c_n S_g) Z_{i,k} \phi_k dv_g - \int_{M \backslash B_{\xi_k}(R \dk)} \left( \triangle_g + c_nS_g \right)Z_{i,k} \phi_k dv_g.
\eal
\een
Using \eqref{propphik2} and H\"older's inequality we get, with \eqref{C02claimcontra}, that:
\ben \label{C02s113}
\bal
\left|  \int_{M \backslash B_{\xi_k}(R \dk)} (h-c_n S_g) Z_{i,k} \phi_k dv_g  \right| &\lesssim \Vert h - c_n S_g \Vert_{L^\infty(2 r_k)} \eta \dk^2 \vek \\
& = o \left( \dk^\pui \Vert \phi_k \Vert_{L^\infty(B_{\xi_k}(2\sqrt{\dk}))} \right).
\eal
\een
Using \eqref{C02s17} we have:
\ben \label{C02s114}
\left| \int_{\partial B_{\xi_k}(R \dk)} \phi_k \partial_\nu Z_{i,k} d\sigma_g \right| \lesssim  \left( \frac{1}{\left(1 + R \right)^{n-2}}  + o(1) \right) \dk^\pui \Vert \phi_k \Vert_{L^\infty(B_{\xi_k}(2\sqrt{\dk}))}. 
\een
We now compute the third integral in \eqref{C02s112} by using  \eqref{lapZik} below. By Proposition \ref{propC0grossier} and \eqref{relationrkmk} we have
\ben \label{C02s115}
\bal
\int_{M \backslash B_{\xi_k}(R \dk)} \dk^{\frac{n}{2}} r_k^{-n-1} \mathds{1}_{r_k \le d_k \le 2 r_k} |\phi_k| dv_g + \int_{M \backslash B_{\xi_k}(R \dk)} \dk^\pui r_k^{-n} \mathds{1}_{r_k \le d_k \le 2 r_k} |\phi_k| dv_g  \\
+ \int_{M\backslash B_{\xi_k}(R \dk)} \dk^{\pui} \phi_k dv_g = o(\dk^\pui) = o \left( \dk^\pui \Vert \phi_k \Vert_{L^\infty(B_{\xi_k}(2\sqrt{\dk}))} \right),
\eal
\een
where the last equality is due to \eqref{C02claimcontra}. By \eqref{propLambdaSg} there holds
\[ |S_{g_{\xi_k}}|(x) \lesssim d_k(x)^2 \textrm{ in } B_{\xi_k}(2r_k), \]
where $d_k$ is the Riemannian distance defined in \eqref{notdk}, so that by \eqref{C02claimcontra} and \eqref{C02s17} we have:
\ben \label{C02s116}
\bal
 \int_{M \backslash B_{\xi_k}(R \dk)} c_n S_{g_{\xi_k}} \Lambda_{\xi_k}^{2^*-2} Z_{i,k} \phi_k dv_g  & =  \int_{B_{\xi_k}(2 r_k) \backslash B_{\xi_k}(\sqrt{\dk})} c_n S_{g_{\xi_k}} \Lambda_{\xi_k}^{2^*-2} Z_{i,k} \phi_k dv_g \\
 &+  \int_{B_{\xi_k}(\sqrt{\dk}) \backslash B_{\xi_k}(R \dk)} c_n S_{g_{\xi_k}} \Lambda_{\xi_k}^{2^*-2} Z_{i,k} \phi_k dv_g \\
& = o \left( \dk^\pui \Vert \phi_k \Vert_{L^\infty(B_{\xi_k}(2\sqrt{\dk}))} \right), \\
\eal
\een
where the last equality is again given by \eqref{C02claimcontra}. 
In case where $(M,g)$ is not locally conformally flat we have, using \eqref{C02claimcontra}, that:
\ben \label{C02s1171}
\bal
\left| \int_{M \backslash B_{\xi_k}(\sqrt{\dk})}  \dk^{\frac{n}{2}} \tk(\cdot)^{2-n} \mathds{1}_{d_k \le 2 r_k} |\phi_k| dv_g \right| & \lesssim \dk^{\frac{n}{2}}  \Vert \phi_k \Vert_{L^\infty(\Omega_k)}\\
& = o \left( \dk^\pui \Vert \phi_k \Vert_{L^\infty(B_{\xi_k}(2\sqrt{\dk}))} \right), \\
\eal
\een
and using \eqref{C02s17} we also obtain that:
\ben \label{C02s1172}
\bal
\left| \int_{B_{\xi_k}(\sqrt{\dk}) \backslash B_{\xi_k}(R \dk)} \dk^{\frac{n}{2}} \tk(\cdot)^{2-n} \mathds{1}_{d_k \le 2 r_k} |\phi_k| dv_g\right| 
 = o \left( \dk^\pui \Vert \phi_k \Vert_{L^\infty(B_{\xi_k}(2\sqrt{\dk}))} \right).
\eal
\een
Finally, by \eqref{C02claimcontra} we can write that: 
\ben \label{C02s118}
\bal
\left| \int_{M \backslash B_{\xi_k}(\sqrt{\dk})} f(\xi_k) W_k^{2^*-2} Z_{i,k} \phi_k dv_g \right| & \lesssim \dk^{\frac{n}{2}} \Vert \phi_k \Vert_{L^\infty( \Omega_k)}, \\
& = o \left( \dk^\pui \Vert \phi_k \Vert_{L^\infty(B_{\xi_k}(2\sqrt{\dk}))} \right) \\
\eal
\een
where $\Omega_k$ is as in \eqref{defOmegak}, and using \eqref{C02s17} one gets that
\ben \label{C02s119}
\bal
&\left| \int_{B_{\xi_k}(\sqrt{\dk}) \backslash B_{\xi_k}(R \dk)} f(\xi_k) W_k^{2^*-2} Z_{i,k} \phi_k dv_g \right| \\
 &\qquad  \lesssim  \left( \frac{1}{\left(1 + R \right)^{n-2}} + o(1) \right)  \dk^\pui \Vert \phi_k \Vert_{L^\infty(B_{\xi_k}(2\sqrt{\dk}))}.
\eal
\een
Combining \eqref{C02s115}, \eqref{C02s116}, \eqref{C02s1171}, \eqref{C02s1172}, \eqref{C02s118} and \eqref{C02s119} in \eqref{lapZik} we obtain that:
\ben \label{C02s120}
\left| \int_{M \backslash B_{\xi_k}(R \dk)} \left( \triangle_g + c_nS_g \right)Z_{i,k} \phi_k dv_g \right| \lesssim  \left( \frac{1}{\left(1 + R \right)^{n-2}} + o(1) \right)  \dk^\pui \Vert \phi_k \Vert_{L^\infty(B_{\xi_k}(2\sqrt{\dk}))}.
\een
We now divide \eqref{C02s112} by $\dk^\pui \Vert \phi_k \Vert_{L^\infty(B_{\xi_k}(2\sqrt{\dk}))}$. We first use \eqref{C02s113}, \eqref{C02s114}, \eqref{C02s120}, the $C^1$ convergence of $\tpk$ towards $\tpz$ and the expression of $Z_{i,k}$ in \eqref{defZk} to pass to the limit as $k \to + \infty$. We then use \eqref{C02s19bis} to pass to the limit as $R \to +\infty$, to obtain that:
\[ \int_{\RR^n} \left \langle \nabla \tpz, \nabla V_{i,\xi_0} \right \rangle_{eucl} dx = 0 \textrm{ for all } 0 \le i \le n, \]
where $\xi_0 = \lim_{k \to +\infty} \xi_k$. We thus have in the end that: 
\[ \tpz \in \textrm{Span} \{ V_{i,\xi_0}, 0 \le i \le n \}^{\perp}.\]
With \eqref{C02s111}, this implies that $\tpz \equiv 0$ which contradicts the fact that $|\tpz(\tilde{x}_0)| = 1$, and therefore concludes the proof of Claim \ref{claimNphik}.   
\end{proof}

\noindent Hence, equation \eqref{eqclaimNphik} holds true. Assume first that there holds $M_k \lesssim 1$, where $M_k$ is defined in \eqref{C02s1defMk}. Then \eqref{C02s14} shows that, for any sequence $x_k \in B_{\xi_k}(2 \sqrt{\dk})$, there holds:
\ben \label{C02conclu1}
\bal
|\phi_k(x_k)| & \lesssim  \Vert \phi_k \Vert_{L^\infty(\Omega_k)} +  \sqrt{\dk} \Vert \nabla \phi_k \Vert_{L^\infty(\Omega_k)} 
  + \dk +  \left( \frac{\dk}{\tk(x_k)} \right)^2\\ 
& + \Bigg[ \dk \Vert \nabla f \Vert_{L^\infty(2 r_k)} + \Vert h - c_n S_g \Vert_{L^\infty(2 r_k)} \tk(x_k)^2 + \tk(x_k)^4 \mathds{1}_{nlcf} \Bigg] W_k(x_k).  \\ 
\eal
\een
Assume now that $M_k >> 1$. Then, plugging \eqref{C02s1defMk} in \eqref{C02s14} yields:
\ben \label{C02conclu21}
\bal
|\phi_k(x_k)| & \lesssim  \Vert \phi_k \Vert_{L^\infty(\Omega_k)} +  \sqrt{\dk} \Vert \nabla \phi_k \Vert_{L^\infty(\Omega_k)}  
  + \dk + \left( \frac{\dk}{\tk(x_k)} \right)^2 \\ 
& + \Bigg[ \dk \Vert \nabla f \Vert_{L^\infty(2 r_k)} + \Vert h - c_n S_g \Vert_{L^\infty(2 r_k)} \tk(x_k)^2 + \tk(x_k)^4 \mathds{1}_{nlcf}\Bigg] W_k(x_k)  \\
&  + \Bigg[  \dk^3   
 + \dk^{4 - \frac{n}{2}} \Vert \nabla f \Vert_{L^\infty(2 r_k)} + \dk^{5 - \frac{n}{2}} \Vert h - c_n S_g \Vert_{L^\infty(2 r_k)} \\
 & \qquad \qquad \qquad \qquad + \dk^{7 - \frac{n}{2}} \mathds{1}_{nlcf} \Bigg] \tk(x_k)^{-2} .\\
\eal
\een
We proceed again as we did to obtain \eqref{C02s17}: we use \eqref{C02conclu21} to obtain a better estimate of $I_1$ in \eqref{I1toI6}, which in turn gives with \eqref{C02s12} an improved pointwise estimate of $|\phi_k|$. After a finite number of iterations one obtains:
\ben \label{sansgrad7}
\bal
& |\phi_k(x_k)|  \lesssim  \Vert \phi_k \Vert_{L^\infty(\Omega_k)} +  \sqrt{\dk} \Vert \nabla \phi_k \Vert_{L^\infty(\Omega_k)} 
  + \dk  \\
& + \Bigg[  
 \dk^{\frac{n}{2}} 
 + \dk \Vert \nabla f \Vert_{L^\infty(2 r_k)}  + \Vert h - c_n S_g \Vert_{L^\infty(2 r_k)} \dk^2  \left| \ln{ \left( \frac{\tk(x_k)}{\dk} \right)} \right| \\
 & + \Vert h - c_n S_g \Vert_{L^\infty(2 r_k)} \tk(x_k)^2 + \tk(x_k)^4 \mathds{1}_{nlcf} \Bigg]  W_k(x_k) + \left( \frac{\dk}{\tk(x_k)}\right)^2 .\\
\eal
 \een

\medskip

\noindent \textbf{We now assume that $n = 6$.}
The proof closely follows the $n \ge 7$ case so we only highlight the main differences. From \eqref{vraieeqphik}, using \eqref{syseta}, Proposition \ref{propC0grossier} and \eqref{lapbulle}, it is easily seen that $\phi_k$ satisfies:
\ben \label{eqphik6}
\bal
 \left( \triangle_g + h \right) \left(  \phi_k - \sum_{i=0}^6 \lki Z_{i,k} \right)   & =  (2+o(1))  f \big( u+ W_k \big) \phi_k \\
 &+ \big(f - f(\xi_k) \big) W_k^{2}  - \big(h - \frac{1}{5} S_g  - 2 f u\big) W_k  \\ 
&- \frac{1}{5} S_{g_{\xi_k}} \Lambda_{g_{\xi_k}} W_k + O (\dk^2 \mathds{1}_{d_k \le 2 r_k}) + O \big(\dk^2 r_k^{-6} \mathds{1}_{r_k \le d_k \le 2 r_k} \big) \\ 
 &  + \left(|\Lg T + \sigma|_g^2 + \pi^2 \right) \left( u_k^{-4} - u^{-4} \right) \\
& +  \frac{ |\Lg T_k + \sigma|_g^2 - |\Lg T + \sigma|_g^2 }{\left( u + W_k + v_k \right)^{4}} . \\
\eal
\een
Let 
 $(x_k)_k $ be a sequence of points in $B_{\xi_k}(2 \sqrt{\dk})$. 
If $x_k \in \Omega_k$ we have as before:
\ben \label{C02s126}
|\phi_k(x_k)| \le \Vert \phi_k \Vert_{L^\infty(\Omega_k)}.
\een
Assume now that $x_k \in B_{\xi_k}(\sqrt{\dk})$. Mimicking \eqref{C02s1I3}, there holds
\[ \int_{B_{\xi_k}(2 \sqrt{\dk})} \big( h - \frac{1}{5}S_g - 2 fu \big) W_k(y) G(x_k,y) dv_g(y) \lesssim \Vert  h - \frac{1}{5}S_g - 2 fu \Vert_{L^\infty(2 r_k)} \left( \frac{\dk}{\tk(x_k)}\right)^2,  \]
so that using \eqref{C02s1I2}--\eqref{C02s1I6}, a representation formula for $ \phi_k - \sum_{i=0}^6 \lki Z_{i,k}$ in $B_{\xi_k}(2 \sqrt{\dk})$ gives, with \eqref{eqphik6}, that:
\ben \label{C02s136}
\bal
& \left| \phi_k - \sum_{i=0}^6 \lki Z_{i,k}\right|(x_k)  \lesssim   \Vert \phi_k \Vert_{L^\infty(\Omega_k)} +  \sqrt{\dk} \Vert \nabla \phi_k \Vert_{L^\infty(\Omega_k)} +  \sum_{i=0}^6 |\lki| + \dk \\
 & + \dk \Vert \nabla f \Vert_{L^\infty(2 r_k)}  W_k(x_k)  + \left(  \big \Vert h - \frac15 S_g - 2 f u \big \Vert_{L^\infty(2 r_k)} +  \Vert \phi_k \Vert_{L^\infty(B_{\xi_k}(2\sqrt{\dk}))}\right) \left( \frac{\dk}{\tk(x_k)} \right)^2  .\\
\eal
\een
As before, the $|\lki|$ are estimated by 
\ben \label{C02s1estimlki6}
\bal
\sum_{i=0}^6 |\lki| &\lesssim \dk^2 \Big( \Vert \phi_k \Vert_{L^\infty(\Omega_k)} +  \sqrt{\dk} \Vert \nabla \phi_k \Vert_{L^\infty(\Omega_k)} \Big) + \dk^3   \\
&
 + \dk \Vert \nabla f \Vert_{L^\infty(2 r_k)} + \big \Vert h -  \frac15 S_g - 2 fu \big \Vert_{L^\infty(2 r_k)} \dk^2 + \dk^2 \Vert \phi_k \Vert_{L^\infty(B_{\xi_k}(2\sqrt{\dk}))},  \\
\eal
\een
so that plugging \eqref{C02s1estimlki6} in \eqref{C02s136} and using \eqref{C02s126} gives, for any sequence of points $x_k \in B_{\xi_k}(2 \sqrt{\dk})$:
\ben \label{C02s146}
\bal
|\phi_k(x_k)| & \lesssim  \Vert \phi_k \Vert_{L^\infty(\Omega_k)} +  \sqrt{\dk} \Vert \nabla \phi_k \Vert_{L^\infty(\Omega_k)} + \dk \Vert \nabla f \Vert_{L^\infty(2 r_k)} W_k(x_k) + \dk \\
& +  \left(  \big \Vert h - \frac15 S_g - 2 f u \big \Vert_{L^\infty(2 r_k)} +  \Vert \phi_k \Vert_{L^\infty(B_{\xi_k}(2\sqrt{\dk}))}\right)\left( \frac{\dk}{\tk(x_k)} \right)^2. \\
\eal
\een
We now claim that there holds:
\ben \label{eqclaimNphik6}
\Vert \phi_k \Vert_{L^\infty(B_{\xi_k}(2\sqrt{\dk}))} \lesssim N_k
\een
where we have let:
\ben \label{C02s1defMk6}
\bal
N_k &=  \Vert \phi_k \Vert_{L^\infty(\Omega_k)} +  \sqrt{\dk} \Vert \nabla \phi_k \Vert_{L^\infty(\Omega_k)} 
 +  \dk  + \dk^{-1}\Vert \nabla f \Vert_{L^\infty(2 r_k)} + \big \Vert h - \frac15 S_g - 2 f u \big \Vert_{L^\infty(2 r_k)},\\
\eal
\een
and where $\Omega_k$ is as in \eqref{defOmegak}. To prove \eqref{eqclaimNphik6}, let $(x_k)_k$ be a sequence of points in $B_{\xi_k}(2 \sqrt{\dk})$ satisfying $|\phi_k(x_k)| = \max_{B_{\xi_k}(2 \sqrt{\dk})} |\phi_k|$. Estimate \eqref{eqclaimNphik6} is trivially satisfied if $x_k \in \Omega_k$, so we assume in the following that $x_k \in B_{\xi_k}(\sqrt{\dk})$. We proceed by contradiction and assume that there holds:
\ben \label{C02claimcontra6}
|\phi_k(x_k)| = \Vert \phi_k \Vert_{L^\infty(B_{\xi_k}(2\sqrt{\dk}))} >>  N_k
\een
as $k \to + \infty$, where $N_k$ is defined in \eqref{C02s1defMk6}. Let $(y_k)_k$ be any other sequence of points in $B_{\xi_k}(\sqrt{\dk})$. Proceeding as in \eqref{C02s15}--\eqref{C02s17} we obtain that there holds:
\ben \label{C02s176}
|\phi_k(y_k)| \lesssim  \left(  \left( \frac{\dk}{\tk(y_k)} \right)^{4} + o(1) \right) \Vert \phi_k \Vert_{L^\infty(B_{\xi_k}(2\sqrt{\dk}))} ,  
\een
so that \eqref{C02s176} applied to the sequence $x_k$ given by \eqref{C02claimcontra6} yields:
\be 
\tk(x_k) \lesssim \dk,
\ee
where $\tk$ is as in \eqref{defthetak}. Define now, for $y \in B_0(2 \dk^{-\frac12})$:
\be  
\tpk(x) = \Vert \phi_k \Vert_{L^\infty(B_{\xi_k}(2\sqrt{\dk}))}^{-1} \phi_k \left( \exp_{\xi_k}^{g_{\xi_k}} (\dk y )\right),
\ee
and let $\tilde{x}_k = \frac{1}{\dk} \left( \exp_{\xi_k}^{g_{\xi_k}} \right)^{-1}(x_k)$. Here again,
$\tilde{x}_k \to \tilde{x}_0 \in \RR^n$ as $k \to + \infty$. The arguments that led to \eqref{C02s111} adapt therefore with no modification and we obtain in the end that 
 $\tpk$ converges in $C^1_{loc}(\RR^n)$, up to a subsequence, to some function $\tpz$ satisfying $|\tpz(\tilde{x}_0)| =1$ and
\ben \label{C02s1116}
\tpz \in \textrm{Span} \{ V_{i,\xi_0}, 0 \le i \le 6 \}.
\een

\medskip
\noindent To conclude the proof of \eqref{eqclaimNphik6} we now show that $\tpz \in \textrm{Span} \{ V_{i,\xi}, 0 \le i \le 6 \}^\perp$. 
By \eqref{propphik1}, $\phi_k \in K_k^{\perp}$, where $K_k = K_{k,t_k,\xi_k}$ is defined in \eqref{noyau}. Hence for any $0 \le i \le 6$ and for any $R > 0$ fixed, there holds:
\ben  \label{C02s1126}
\bal
&\int_{B_{\xi_k}(R \dk)} \left \langle \nabla Z_{i,k}, \nabla \phi_k \right \rangle_g + h Z_{i,k} \phi_k dv_g = - \int_{\partial B_{\xi_k}(R \dk)} \phi_k \partial_\nu Z_{i,k} d\sigma_g \\
&- \int_{M \backslash B_{\xi_k}(R \dk)} (h-\frac15 S_g) Z_{i,k} \phi_k dv_g - \int_{M \backslash B_{\xi_k}(R \dk)} \left( \triangle_g + \frac15 S_g \right)Z_{i,k} \phi_k dv_g.
\eal
\een
By \eqref{C02claimcontra6} and \eqref{defZk} one has that: 
\[ \int_{M \backslash B_{\xi_k}(\sqrt{\dk})} (h-\frac15 S_g) Z_{i,k} \phi_k dv_g = o( \dk^2 \Vert \phi_k \Vert_{L^\infty(B_{\xi_k}(2\sqrt{\dk}))}),\]
while using \eqref{C02s176} there holds that:
\[   \int_{B_{\xi_k}(\sqrt{\dk}) \backslash B_{\xi_k}(R \dk)} (h-c_n S_g) Z_{i,k} \phi_k dv_g  = o \big( \dk^2  \Vert \phi_k \Vert_{L^\infty(B_{\xi_k}(2\sqrt{\dk}))} \big), \]
that
\[ \left| \int_{\partial B_{\xi_k}(R \dk)} \phi_k \partial_\nu Z_{i,k} d\sigma_g \right| \lesssim  \left( \frac{1}{\left(1 + R \right)^{4}}  + o(1) \right) \dk^2 \Vert \phi_k \Vert_{L^\infty(B_{\xi_k}(2\sqrt{\dk}))}, \]
and, with \eqref{relationrkmk} and \eqref{C02claimcontra6}, that
\[ \bal
\int_{M \backslash B_{\xi_k}(R \dk)} \dk^{3} r_k^{-7} \mathds{1}_{r_k \le d_k \le 2 r_k} |\phi_k| dv_g + \int_{M \backslash B_{\xi_k}(R \dk)} \dk^2 r_k^{-6} \mathds{1}_{r_k \le d_k \le 2 r_k} |\phi_k| dv_g  \\
+ \int_{M\backslash B_{\xi_k}(R \dk)} \dk^{2} \phi_k dv_g = o \left( \dk^2 \Vert \phi_k \Vert_{L^\infty(B_{\xi_k}(2\sqrt{\dk}))} \right).
\eal \]
Using \eqref{propLambdaSg}, \eqref{C02claimcontra6} and \eqref{C02s176} we get also that:
\[ \bal
\int_{M \backslash B_{\xi_k}(R \dk)} & \frac15 S_{g_{\xi_k}} \Lambda_{g_{\xi_k}} Z_{i,k} \phi_k dv_g \\
& =  \int_{M \backslash B_{\xi_k}(\sqrt{\dk})} \frac15 S_{g_{\xi_k}} \Lambda_{g_{\xi_k}} Z_{i,k} \phi_k dv_g + \int_{B_{\xi_k}(\sqrt{\dk}) \backslash B_{\xi_k}(R \dk)}\frac15 S_{g_{\xi_k}} \Lambda_{g_{\xi_k}} Z_{i,k} \phi_k dv_g \\
&  =  o \left( \dk^2 \Vert \phi_k \Vert_{L^\infty(B_{\xi_k}(2\sqrt{\dk}))} \right).
\eal
\]
Similarly, we get that
\[ \int_{M \backslash B_{\xi_k}(R \dk)}  \dk^{3} \tk(\cdot)^{-4} \mathds{1}_{ d_k \le 2 r_k} |\phi_k| dv_g  
= o \left( \dk^2 \Vert \phi_k \Vert_{L^\infty(B_{\xi_k}(2\sqrt{\dk}))} \right)
\]
that
\[ \bal
\left| \int_{M \backslash B_{\xi_k}(\sqrt{\dk})} f(\xi_k) W_k Z_{i,k} \phi_k dv_g \right| & \lesssim \dk^3 \Vert \phi_k \Vert_{L^\infty( \Omega_k)}, \\
& = o \left( \dk^2 \Vert \phi_k \Vert_{L^\infty(B_{\xi_k}(2\sqrt{\dk}))} \right), \\
\eal \]
and, using \eqref{C02s176}, that
\[ \left| \int_{B_{\xi_k}(\sqrt{\dk}) \backslash B_{\xi_k}(R \dk)} f(\xi_k) W_k Z_{i,k} \phi_k dv_g \right| \lesssim \left( \frac{1}{(1+R)^4} + o(1)\right) \dk^2 \Vert \phi_k \Vert_{L^\infty(B_{\xi_k}(2\sqrt{\dk}))},\] 
so that passing to the limit as $k \to + \infty$ and then as $R \to + \infty$ gives, with \eqref{C02s1126}:
\[ \tilde{\phi}_0  \in \textrm{Span} \{ V_{i,\xi_0}, 0 \le i \le 6 \}^{\perp},\]
which is a contradiction with \eqref{C02s1116} since $|\tpz(\tilde{x}_0)| = 1$. Therefore \eqref{eqclaimNphik6} is proven, and an iteration argument as for the $n \ge 7$ case gives:
\ben \label{sansgrad6}
\bal
& |\phi_k(x_k)|  \lesssim   \Vert \phi_k \Vert_{L^\infty(\Omega_k)} +  \sqrt{\dk} \Vert \nabla \phi_k \Vert_{L^\infty(\Omega_k)}  +\dk  + \Bigg[ \dk^3 +  \dk \Vert \nabla f \Vert_{L^\infty(2 r_k)} \\
& + \big \Vert h - \frac15 S_g - 2 f u \big \Vert_{L^\infty(2 r_k)}  \dk^2 \left| \ln{ \left( \frac{\tk(x_k)}{\dk} \right)}\right| + \big \Vert h - \frac15 S_g - 2 f u \big \Vert_{L^\infty(2 r_k)} \tk(x_k)^2  \Bigg] W_k(x_k) .\\
\eal
\een

\medskip

\noindent To conclude the proof of Proposition \ref{C02pres} it remains to improve \eqref{sansgrad7} and \eqref{sansgrad6} into the final estimates \eqref{estC02pres} and \eqref{estC02pres2}. We only prove the $n \ge 7$ case, as the six-dimensional case follows from similar arguments. Let $(x_k)_k$ be any sequence of points in $B_{\xi_k}(2 \sqrt{\dk})$. We write a Green representation formula and differentiate it at $x_k$. As in \eqref{C02s1I2}--\eqref{C02s1I6} one obtains that: 
\ben 
\bal
& \left|  \nabla \left( \psi_k - \sum_{j=0}^n \lki Z_{i,k} \right) (x_k) \right| \lesssim 
\frac{1}{\sqrt{\dk}} \Vert \phi_k \Vert_{L^\infty(\Omega_k)} +  \Vert \nabla \phi_k \Vert_{L^\infty(\Omega_k)} 
  + \sqrt{\dk}  \\
& + \Bigg[  
 \dk^{\frac{n}{2}} 
 + \dk \Vert \nabla f \Vert_{L^\infty(2 r_k)}  + \Vert h - c_n S_g \Vert_{L^\infty(2 r_k)} \tk(x_k)^2 + \tk(x_k)^4 \mathds{1}_{nlcf} \Bigg]  \dk^\pui \tk(x_k)^{1-n} \\
 & + \int_{B_{\xi_k}(2 \sqrt{\dk})} f \left| u_k^{2^*-1} - W_k^{2^*-1} - u^{2^*-1} \right|(y) \nabla_x G(x_k,y)dv_g(y). \\
\eal
\een
The last term is then estimated with \eqref{sansgrad7} to obtain \eqref{estC02pres}.
\end{proof}

\noindent Note that using Claim \ref{claimNphik} and \eqref{eqclaimNphik6} in \eqref{C02s1estimlki} and \eqref{C02s1estimlki6} yields:
\begin{itemize}
\item If $n \ge 7$:
\ben \label{grosseestimlki}
\bal
\sum_{i=0}^n |\lki| &\lesssim \dk^\pui \Big( \Vert \phi_k \Vert_{L^\infty(\Omega_k)} +  \sqrt{\dk} \Vert \nabla \phi_k \Vert_{L^\infty(\Omega_k)} \Big) + \dk^\pui \\
& + \dk \Vert \nabla f \Vert_{L^\infty(2 r_k)} + \Vert h - c_n S_g \Vert_{L^\infty(2 r_k)} \dk^2 + \dk^4 \mathds{1}_{nlcf}. \\
\eal
\een
\item If $n = 6$:
\ben \label{grosseestimlki6}
\bal
\sum_{i=0}^6 |\lki| &\lesssim \dk^2 \Big( \Vert \phi_k \Vert_{L^\infty(\Omega_k)} +  \sqrt{\dk} \Vert \nabla \phi_k \Vert_{L^\infty(\Omega_k)} \Big) + \dk^3 \\
& + \dk \Vert \nabla f \Vert_{L^\infty(2 r_k)} + \big \Vert h - \frac15 S_g - 2 f u\big \Vert_{L^\infty(2 r_k)} \dk^2 . \\
\eal
\een
\end{itemize}

\medskip

\noindent As a consequence of these refined local estimates we are now in position to obtain global pointwise estimates on $\phi_k$ in the whole manifold $M$:

\begin{prop} \label{propestglob}
Let $D >0$ and $(\vek)_k \in \mathcal{E}$ and assume that  $\vek >> \mk^{\frac{3}{2}}$ as $k \to + \infty$, where $\mk$ is defined in \eqref{defmk}. Let $(t_k,\xi_k)_k$ be a sequence in $ [1/D, D]\times M$, let $v_k \in F_k = F \big( \vek,t_k,\xi_k \big)$ and let $\phi_k = \phi_k \big(t_k,\xi_k,v_k \big)$ be given by Proposition \ref{propptfixe1}. 
Let $(x_k)_k$ be any sequence of points in $M$. There holds then:
\ben \label{ponctuelphik}
|\phi_k(x_k)| \le C \left( \dk + \eta \vek \right) \Big( u(x_k) + W_k(x_k) \Big),
\een
where $\eta$ is as in \eqref{introeta}, for some positive constant $C$ independent of $\eta$ and $k$. As a consequence, we have the following gradient estimate:
\be
 \left| \nabla \phi_k(x_k) \right| \le C \left( \dk + \eta \vek \right) \left( 1 + \dk^\pui \tk(x_k)^{1-n} \right).
 \ee
\end{prop}
\noindent Note that the constant $C$ appearing in the statement of Proposition \ref{propestglob} \emph{a priori} depends on $g, u_0$ (as in \eqref{defu0}), $h,f, \sigma, \pi$, but does not depend on the choice of the sequences $(\vek)_k, (t_k)_k, (\xi_k)_k$ and $(v_k)_k$.

\begin{proof}
Let $D >0$ and $(\vek)_k \in \mathcal{E}$ and assume that  $\vek >> \mk^{\frac{3}{2}}$ as $k \to + \infty$, where $\mk$ is defined in \eqref{defmk}. Let $(t_k,\xi_k)_k$ be a sequence in $ [1/D, D]\times M$, let $v_k \in F_k = F \big( \vek,t_k,\xi_k \big)$ and let $\phi_k = \phi_k \big(t_k,\xi_k,v_k \big)$ be given by Proposition \ref{propptfixe1}. 
Define:
\ben \label{supphikbulle}
\nu_k = \left| \left| \frac{\phi_k}{u + W_k} \right| \right|_{L^\infty(M)}.
\een
Proposition \ref{propC0grossier} shows that $\nu_k = o(1)$ as $k \to + \infty$. 
We let in the following, for any $v \in H^1(M)$:
\ben \label{Ldefop}
L_u (v) = \triangle_g v + \left[ h - (2^*-1) f u^{2^*-2} + (2^*+1) \frac{|\Lg T + \sigma|_g^2 + \pi^2}{u^{2^*+2}}\right] v.
\een
$L_u$ is the linearized operator of the scalar equation of \eqref{syseta} at $u$. 

\medskip

\noindent \textbf{As before, assume first that $n \ge 7$.} It is easily seen from \eqref{eqphik} that $\phi_k$ satisfies:
\ben \label{nouveqphik}
\bal
 L_u \phi_k & =  \sum_{i=0}^n \lki \left( \triangle_g + h \right)Z_{i,k}  +  f \left( u_k^{2^*-1} - W_k^{2^*-1} - (2^*-1) u^{2^*-2} \phi_k - u^{2^*-1} \right) \\
 &+ \big(f - f(\xi_k) \big) W_k^{2^*-1}  - \big(h - c_n S_g \big) W_k  \\ 
&- c_n S_{g_{\xi_k}} \Lambda_{g_{\xi_k}}^{2^*-2} W_k + O (\dk^\pui \mathds{1}_{dk \le 2 r_k} ) + O \big(\dk^\pui r_k^{-n} \mathds{1}_{r_k \le d_k \le 2 r_k} \big) \\ 
 &  + \left(|\Lg T + \sigma|_g^2 + \pi^2 \right) \left( u_k^{-2^*-1} - u^{-2^*-1} + (2^*+1) u^{-2^*-2} \phi_k \right) \\
& +  \frac{ |\Lg T_k + \sigma|_g^2 - |\Lg T + \sigma|_g^2 }{\left( u + W_k + v_k \right)^{2^*+1}} . \\\eal
\een
Let $(x_k)_k$ be any sequence of points in $M$. Let $G_u$ be the Green's function of the operator $L_u$ defined in \eqref{Ldefop}. By \eqref{uetastableX}, the operator $L_u$ is coercive, and therefore its Green function satisfies  (see for instance Robert \cite{RobDirichlet}):
\ben \label{estLu}
\frac{1}{C} d_{g_{\xi_k}}(x,y)^{2-n} \le G_u(x,y) \le C d_{g_{\xi_k}}(x,y)^{2-n},
\een
for some positive constant $C$. Note that since we assumed $|\sigma|_g + \pi > 0$ somewhere in $M$, the constant $C$ in \eqref{estLu} does not depend on $\eta$ in \eqref{introeta}, provided $\eta$ is small enough.  First, by \eqref{grosseestimlki}, 
 by Proposition \ref{propC0grossier} and since $|Z_{i,k}| \lesssim W_k$ for all $0 \le i \le n$, we have:
\ben \label{C030}
\bal
 \int_M G_u(x_k,y) & \sum_{i=0}^n \lki \left( \triangle_g + h \right)Z_{i,k} (y) dv_g (y)  \\
& \lesssim \Bigg[ 
\dk^\pui 
 + \dk \Vert \nabla f \Vert_{L^\infty(2 r_k)} + \Vert h - c_n S_g \Vert_{L^\infty(2 r_k)} \dk^2 + \dk^4 \mathds{1}_{nlcf} \Bigg] \dk^\pui \tk(x_k)^{2-n}. \\
\eal
\een
Independently, straightforward computations using \eqref{estLu} show that there holds:
\ben \label{C031}
\bal
&\int_M G_u(x_k,y) \Big[ \big(f(y) - f(\xi_k) \big) W_k^{2^*-1}(y) + \big(h - c_n S_g \big)  W_k(y)  + c_n S_{g_{\xi_k}}(y) W_k(y) \Big]  dv_g(y) \\
& \quad \lesssim 
\Big( \dk \Vert \nabla f \Vert_{L^\infty(2 r_k)} + \Vert h - c_n S_g \Vert_{L^\infty(2 r_k)} \tk(x_k)^2 \Big) \dk^{\pui} \tk(x_k)^{2-n} + \dk^\pui \tk(x_k)^{6-n} \mathds{1}_{nlcf} .\\
\eal
\een
Similarly, there holds:
\ben \label{C032}
\int_M G_u(x_k,y) \left[  \dk^\pui+ \dk^\pui r_k^{-n} \mathds{1}_{r_k \le d_k \le 2 r_k} \right] dv_g(y) =  O \big( \dk^\pui r_k^{2-n} \big),
\een
where $d_k$ is as in \eqref{notdk}. Now, with \eqref{defuk}, straightforward computations using Proposition \ref{propC0grossier} yield that:
\ben \label{C033}
\bal
 &\int_{B_{\xi_k}(\sqrt{\dk})}  G_u(x_k,y)\left(|\Lg T + \sigma|_g^2 + \pi^2 \right) \\
 & \times \left( u_k^{-2^*-1} - u^{-2^*-1} + (2^*+1) u^{-2^*-2} \phi_k \right)(y) dv_g(y) \lesssim \dk^{\frac{n}{2}} \tk(x_k)^{2-n} + \dk^{\pui} \tk(x_k)^{4-n},
\eal
\een
while using \eqref{defuk}, Proposition \ref{propC0grossier} and the definition of $\nu_k$ in \eqref{supphikbulle} there holds:
\ben \label{C034}
\bal
 \int_{M \backslash B_{\xi_k}(\sqrt{\dk})}&  G_u(x_k,y)\left(|\Lg T + \sigma|_g^2 + \pi^2 \right) \\
 & \times \left( u_k^{-2^*-1} - u^{-2^*-1}  + (2^*+1) u^{-2^*-2} \phi_k \right)(y) dv_g(y) \\
 & \lesssim  \dk^\pui \tk(x_k)^{4-n}  + \nu_k^2 .
\eal
\een
Similarly, using \eqref{defuk} and the definition of $\nu_k$ in \eqref{supphikbulle} one obtains that:
\ben \label{C034bis}
\bal
\int_{M \backslash B_{\xi_k}(\sqrt{\dk})}  & G_u(x_k,y) f \left( u_k^{2^*-1} - W_k^{2^*-1} - (2^*-1) u^{2^*-2} \phi_k - u^{2^*-1} \right)(y) dv_g(y) \\
& \lesssim  
\dk^\pui \tk(x_k)^{4-n} + \nu_k^2.
\eal
\een
Using Proposition \ref{C02pres}  
and estimating $\Vert \phi_k \Vert_{L^\infty(\Omega_k)} + \sqrt{\dk} \Vert \nabla \phi_k \Vert_{L^\infty(\Omega_k)}$ with Proposition \ref{propC0grossier} gives that:
\ben \label{C035}
\bal
 \int_{B_{\xi_k}(\sqrt{\dk})} & G_u(x_k,y) f \left( u_k^{2^*-1} - W_k^{2^*-1} - (2^*-1) u^{2^*-2} \phi_k - u^{2^*-1} \right)(y) dv_g(y) \\
 & \lesssim
  \left( \frac{\dk}{\tk(x_k)}\right)^2  + \Bigg[ \dk^{\frac{n}{2}} + \dk \Vert \nabla f \Vert_{L^\infty(2 r_k)} + \Vert h - c_n S_g \Vert_{L^\infty(2 r_k)} \dk^2 \left| \ln(\dk) \right| \\
  &+ \dk^2 \tk(x_k) \mathds{1}_{nlcf} \Bigg]  \dk^{\pui} \tk(x_k)^{2-n}. \\ \eal
\een
Finally, using again \eqref{estLTkLT} below and since $v_k \in F_k$, there holds that:
\ben \label{C036}
\int_{B_{\xi_k}(\sqrt{\dk})}  G_u(x_k,y)  \frac{ |\Lg T_k + \sigma|_g^2 - |\Lg T + \sigma|_g^2 }{\left( u + W_k + v_k \right)^{2^*+1}}(y)  dv_g(y) \lesssim \dk^{\frac{n}{2}} \tk(x_k)^{2-n} + \dk \vek
\een
and that
\ben \label{C037}
\int_{M \backslash B_{\xi_k}(\sqrt{\dk})} G_u(x_k,y)  \frac{ |\Lg T_k + \sigma|_g^2 - |\Lg T + \sigma|_g^2 }{\left( u + W_k + v_k \right)^{2^*+1}}(y)  dv_g(y) \lesssim \dk^{\frac{n}{2}} \tk(x_k)^{2-n} + \dk^\pui \tk(x_k)^{4-n} 
 + \eta \vek,
\een
where $\eta$ is as in \eqref{introeta}. Writing a Green's representation formula for \eqref{nouveqphik} together with \eqref{C030} --
\eqref{C037}, with Proposition \ref{propC0grossier} and with \eqref{relationrkmk} gives then:
\ben \label{C038}
\bal
|\phi_k(x_k)| &\lesssim   \eta \vek+ \nu_k^2 +  \dk \Big( u(x_k) + W_k(x_k) \Big)
\eal
\een
where $\nu_k$ is as in \eqref{supphikbulle}. 
Also, to obtain \eqref{C038} we used that there holds, for any $x \in M$:
\ben \label{asymptotermes}
\left( \frac{\dk}{\tk(x_k)} \right)^2 + \dk^{\pui} \tk(x)^{4-n} + \dk^{\frac{n}{2}} \tk(x)^{2-n} \le C \dk \Big( u(x)  + W_k(x) \Big),
\een
for some positive constant $C$ independent of $\eta$ and $k$. Coming back to the definition of $\nu_k$ in \eqref{supphikbulle}, it remains to apply \eqref{C038} at the sequence $(x_k)_k$ of points of $M$ where $\nu_k$ is attained. Since, by Proposition \ref{propC0grossier}, there holds that $\nu_k \to 0$ as $k \to + \infty$, we obtain in the end that
\[ \nu_k \lesssim \dk + \eta \vek, \]
which concludes the proof of \eqref{ponctuelphik} when $n \ge 7$. 

\medskip
\noindent \textbf{Assume now that $n = 6$.} 
Rewrite \eqref{eqphik} as:
\ben \label{nouveqphik6}
\bal
 L_u \phi_k & =  \sum_{i=0}^6 \lki \left( \triangle_g + h \right)Z_{i,k}  +  2 f  W_k \phi_k + f \phi_k^2 + \big(f - f(\xi_k) \big) W_k^{2} \\
 & - \big(h - \frac15 S_g - 2f u \big) W_k  - \frac15 S_{g_{\xi_k}} \Lambda_{\xi_k} W_k + O (\dk^2 \mathds{1}_{dk \le 2 r_k} ) + O \big(\dk^2 r_k^{-6} \mathds{1}_{r_k \le d_k \le 2 r_k} \big) \\ 
 &  + \left(|\Lg T + \sigma|_g^2 + \pi^2 \right) \left( u_k^{-4} - u^{-4} + 4 u^{-5} \phi_k \right) \\
& +  \frac{ |\Lg T_k + \sigma|_g^2 - |\Lg T + \sigma|_g^2 }{\left( u + W_k + v_k \right)^{4}} . \\
\eal
\een
Let $(x_k)_k$ be any sequence of points in $M$. As before, a Green's representation formula for \eqref{nouveqphik6} using Propositions \ref{propC0grossier} and \ref{C02pres} and \eqref{relationrkmk} gives now the following estimate:
\ben \label{sticazzi6}
 \bal
 |\phi_k(x_k)| & \lesssim \left( \frac{\dk}{\tk(x_k)}\right)^2 + \dk + \vek + \nu_k^2 + \left( \dk + \dk^2 \left| \ln{\dk} \right| \Vert h - \frac15 S_g - 2f u \big \Vert_{L^\infty(2 r_k)} + \nu_k^2 \right) \dk^2 \tk(x_k)^{-4} \\
 & \lesssim \left( \dk + \eta \vek + \nu_k^2 \right) \big( u + W_k \big)(x_k).
 \eal 
 \een
 Applying \eqref{sticazzi6} at the sequence of points where $\nu_k$ as in \eqref{supphikbulle} is attained concludes the proof of \eqref{ponctuelphik} for the $6$-dimensional case.
 
 \medskip
 
 \noindent The gradient estimates in Proposition \ref{propestglob} are obtained from \eqref{ponctuelphik} by a representation formula argument as before.
\end{proof}

\noindent Proposition \ref{C02pres} provides more information than needed to just prove Proposition \ref{propestglob}. The precision of estimates \eqref{estC02pres} and \eqref{estC02pres2} will turn out to be crucial in section \ref{DL} to obtain precise asymptotic expansions of the $\lki$.

\medskip

\noindent
It is worth noting that Propositions \ref{C02pres} and \ref{propestglob} do depend on the choice of $X$ given by \eqref{defXintro} (mostly to obtain pointwise estimates on the source term) but do not use the specific form of the other coefficients, e.g. of $h$ and $f$ defined in \eqref{defhintro}, \eqref{defhintro6} and \eqref{deffintro}. In particular, Propositions \ref{C02pres} and \ref{propestglob} still hold true for arbitrary choices of $h, \pi \in C^0(M)$, $f \in C^1(M)$, $\sigma \in C^0(M)$ and $Y \in C^0(M)$.

\section{Global fixed-point argument and resolution of the reduced problem} \label{pointfixage}

\noindent 
\noindent Proposition \ref{propestglob} shows that we have pointwise estimates on the remainder $\phi_k$ defined by Proposition \ref{propptfixe1} which only depend on the data $\mk$ and $\vek$, and not on the chosen sequences $(t_k,\xi_k)_k$ or $(v_k)_k$. 

\medskip

\noindent In this section we crucially use this result to show that Banach-Picard's fixed-point theorem applies to the mapping $v_k \mapsto \phi_k$ and yields a solution to the reduced problem for system \eqref{intro1}. The main result of this section is the following:

\begin{prop} \label{propsystreduit}
Let $D >0$. Assume that $\eta$ and $\alpha$ defined in \eqref{introeta} and \eqref{defalpha} are small enough. There exists $k_0 \in \mathbb{N}$ such that for any sequence $(t_k,\xi_k)_k \in [1/D, D]\times M$ and for any $k \ge k_0$, there exists a function $\phi_k = \phi_k \big(t_k,\xi_k \big) \in K_{k,t_k, \xi_k}^\perp$ that satisfies the following system of equations:
\ben \label{systreduit}
\left \{ 
\bal
& \Pi_{K_{k,t_k,\xi_k}^\perp} \Bigg [ u_{k} - \left( \triangle_ g+ h \right)^{-1} \left( f  u_{k}^{2^*-1}  + \frac{|\Lg T_{k} + \sigma|_g^2 + \pi^2}{ u_{k}^{2^*+1}} \right) \Bigg] = 0, \\
& \Dg T_{k} = u_{k}^{2^*} X + Y,\\
\eal \right. 
\een
where, as in \eqref{defuk}, we have let $u_{k} = u + W_{k,t_k,\xi_k} + \phi_k(t_k,\xi_k)$. In addition, there exists a positive constant $C$, independent of $(t_k,\xi_k)_k$ such that there holds:
\ben \label{eqsystreduit}
 \Vert \phi_k(t_k,\xi_k) \Vert_{H^1(M)} \le C \dk \textrm{ and } |\phi_k(t_k,\xi_k)| \le C \dk \Big( u + W_{k,t_k,\xi_k} \Big) \textrm{ in } M,
 \een
and such that $\phi_k(t_k,\xi_k)$ is the unique solution of \eqref{systreduit} in $K_{k,t_k,\xi_k}^\perp$ satisfying in addition \eqref{eqsystreduit}. Also, for any $k$, the mapping $(t,\xi) \mapsto \phi_k(t,\xi) \in C^1(M)$ is continuous.
\end{prop}
\noindent  In Proposition \ref{propsystreduit}, $K_{k,t_k,\xi_k}^\perp$ is again as in \eqref{noyau}. The smallness assumption on $\eta$ and $\alpha$ is made clear in the course of the proof.
\begin{proof}
We let $C_0 = C_0(n,g,u_0,h,f,\sigma,\pi,D)$ denote the smallest of the two positive constants appearing in equations \eqref{ponctuelphik} and \eqref{propphik2}, and we let:
\ben \label{choixvek}
\vek = 4 C_0 \dk.
\een
Assume that $\eta$ in \eqref{introeta} is chosen so that there holds:
\ben \label{conditioneta}
C_0 \eta \le \frac{1}{4}.
\een
Let $(t_k,\xi_k)_k$ be a sequence in $[1/D, D] \times M$ and let $F_k = F(\vek,t_k,\xi_k)$, where $F(\vek,t_k,\xi_k)$ is defined in \eqref{defFea} and $\vek$ is as in \eqref{choixvek}. For any $k$, we define the following mapping:
\ben \label{defmapping}
\Psi_k: \left \{ 
\bal
& F(\vek,t_k,\xi_k) \to F(\vek,t_k,\xi_k) \cap  \Big \{ \vp \in K_{k,t_k,\xi_k}^\perp, \Vert \vp \Vert_{H^1(M)} \le \frac{\vek}{4} \Big \} \\
& v \longmapsto \phi_k = \phi_k(t_k,\xi_k,v) 
\eal \right.,
\een
where $\phi_k(t_k,\xi_k,v)$ is given by Proposition \ref{propptfixe1}. For any $k$, we endow the set $F_k$ with the norm
\ben \label{normeFvek}
\Vert v \Vert_{F_k} = \left| \left|  \frac{v}{u + W_{k,t_k,\xi_k}} \right| \right|_{C^0(M)}.
\een 
That the mapping $\Psi_k$ in \eqref{defmapping} is well-defined for $k \ge k_0$ is a consequence of \eqref{choixvek}, \eqref{conditioneta}, Proposition \ref{propestglob} and \eqref{propphik2}, which also show that the value of such a $k_0$ is independent of the choice of the sequence $(t_k,\xi_k)_k$. We now show that for $k$ sufficiently large the mapping $\Psi_k$ is a contraction for the norm \eqref{normeFvek}. 

\medskip

\noindent  Let, for any $k$, $v_k^1, v_k^2 \in F_k$ and denote by $\phi_k^1$ and $\phi_k^2$ the associated images by $\Psi_k$. Using coherent notations, for $i = 1,2$, we will let $u_k^i = u + W_{k,t_k,\xi_k} + \phi_k^i$ and $T_k^i$ will denote the unique solution of $\Dg T_k^i = {u_k ^i}^{2^*} X + Y$ in $M$. As before, we shall omit the dependence in $t_k$ and $\xi_k$ in the computations. By Proposition \ref{propptfixe1}, for any $k$ there exist $(\lkja)_{0 \le j \le n}$ and $(\lkjb)_{0 \le j \le n}$ such that $\phi_k^1 - \phi_k^2$ satisfies:
\ben \bal \label{eqp1p2}
\Big( \triangle_g + h \Big)& \Big( \phi_k^1 - \phi_k^2 \Big)  = f \left[ \Big(u + W_k + \phi_k^1 \Big)^{2^*-1} - \Big(u + W_k + \phi_k^2 \Big)^{2^*-1} \right] \\
&+ \Big( \pi^2 + \left| \sigma + \Lg T \right|_g^2 \Big) \left[ \Big(u + W_k + \phi_k^1 \Big)^{-2^*-1} - \Big(u + W_k + \phi_k^2 \Big)^{-2^*-1} \right] \\
&+ \left( \left| \sigma + \Lg T_k^1 \right|_g^2 - \left| \sigma + \Lg T \right|_g^2 \right) \left[ \Big(u + W_k + v_k^1 \Big)^{-2^*-1} - \Big(u + W_k + v_k^2 \Big)^{-2^*-1}\right] \\
&+ \Big(u + W_k + v_k^2 \Big)^{-2^*-1} \left( \left| \sigma + \Lg T_k^1 \right|_g^2 - \left| \sigma + \Lg T_k^2 \right|_g^2 \right) \\
&+ \sum_{j=0}^n \left( \lkja - \lkjb\right) \Big( \triangle_g + h \Big) Z_{j,k}.
\eal \een
We first estimate the $H^1$-norm of $\phi_k^1 - \phi_k^2$. Since $v_k^i \in F_k$ there holds, with \eqref{introeta}:
\[ \big| \Dg(T_k^1 - T_k^2) \big| \lesssim \big(u + W_k \big)^{2^*} \eta \Vert v_k^1 - v_k^2 \Vert_{F_k}, \]
so that mimicking the proof of Proposition \ref{controleLTkdessus} yields, with \eqref{defXintro}: for any $x \in M$,
\ben \label{estLT1mT2}
 \big| \Lg (T_k^1 - T_k^2) \big|(x) \lesssim \eta \big( \dk^{\frac{n-1}{2}} \tk(x)^{1-n} + 1 \big) \Vert v_k^1 - v_k^2 \Vert_{F_k}. 
 \een
We now apply $(\triangle_g + h)^{-1}$ to \eqref{eqp1p2} and project on $K_k^{\perp} = K_{k,t_k,\xi_k}^{\perp}$. Using \eqref{estLTkLTordre2}, Proposition \ref{invoplin} and the techniques developed in the proof of Proposition \ref{propptfixe1} we then get that:
\ben \label{p1p2H1}
\Vert \phi_k^1 - \phi_k^2 \Vert_{H^1(M)} \lesssim \eta \Vert v_k^1 - v_k^2 \Vert_{F_k}.
\een
Let now $0 \le j \le n$ and integrate \eqref{eqp1p2} against $Z_{j,k}$. With \eqref{p1p2H1} one obtains:
\ben \label{estl1l2} 
\bal
\sum_{j=0}^n \left| \lkja - \lkjb\right| \lesssim \eta \Vert v_k^1 - v_k^2 \Vert_{F_k}.
\eal
\een

\medskip

\noindent 
We now show that, up to suitably choosing $\eta$ and $\alpha$ as in \eqref{introeta} and \eqref{defalpha}, for $k$ large enough there always holds: 
\ben \label{12obj}
\Vert  \phi_k^1 - \phi_k^2 \Vert_{F_k} \le \frac{1}{2}  \Vert v_k^1 - v_k^2 \Vert_{F_k}.
\een
For this, we let $\Phi_k = \phi_k^1 - \phi_k^2$ and let $x_k$ be the point where the $F_k$-norm of $\Phi_k$ is attained:
\ben  \label{p1p2defxk}
\left | \frac{\Phi_k}{u + W_k} (x_k) \right| = \sup_{x \in M} \left | \frac{\Phi_k}{u + W_k} (x) \right| = \Vert \Phi_k \Vert_{F_k}.
\een

\bigskip

\noindent We distinguish between two cases. We first assume that $d_{g_{\xi_k}}(\xi_k, x_k) = O(\dk)$. We then let, for any $x \in B_0(i_g(M)/\dk)$:
\ben \label{p1p2deftphi}
\tilde{\Phi}_k(x) = \frac{\dk^{\pui} \Phi_k}{\Vert \Phi_k \Vert_{F_k} +  \Vert v_k^1 - v_k^2 \Vert_{F_k}} \big( \exp_{\xi_k}^{g_{\xi_k}} (\dk x) \big).
\een
We also let $\tilde{x}_k = \frac{1}{\dk}{ \exp_{\xi_k}^{g_{\xi_k}}}^{-1}(x_k)$. There holds then: $\tilde{x}_k \to \tilde{x}_0 \in \RR^n$ as $k \to + \infty$ with $|\tilde{x}_0| = R_0$ for some $R_0 \ge 0$. It is easily seen that $\Vert \tilde{\Phi}_k \Vert_{L^\infty} \le 1$ and 
by \eqref{eqp1p2}, using Proposition \ref{propestglob} and standard elliptic theory, we get that the sequence $\tilde{\Phi}_k$ converges in $C^{1}_{loc}(\RR^n)$ to some function $\tilde{\Phi}_0$  which is a solution of:
\ben \label{p1p2eqtp0}
\triangle_\xi \tilde{\Phi}_0 = (2^*-1) f(\xi_0) U_{\xi_0}^{2^*-2} + \sum_{j = 0}^n \tilde{\lambda}_{0,j} \triangle_\xi V_{j,\xi_0},
\een
and which satisfies, for any $x \in \RR^n$:
\ben  \label{p1p2maxphi0}
|\tilde{\Phi}_0(x)| \le |\tilde{\Phi}_0(\tilde{x}_0)| = \left( \lim_{k \to + \infty} \frac{\Vert \Phi_k \Vert_{F_k}}{\Vert \Phi_k \Vert_{F_k} +  \Vert v_k^1 - v_k^2 \Vert_{F_k}} \right) \left(1 + \frac{f(\xi_0)}{n(n-2)} R_0^2 \right)^{1 - \frac{n}{2}}.
\een
In \eqref{p1p2eqtp0} $U_{\xi_0}$ and the $V_{j,\xi_0}$ are defined in \eqref{eqlindefU} and \eqref{defVki}, and we have let
\[ \tilde{\lambda}_{0,j} = \lim_{k \to + \infty} \frac{ \lkja - \lkjb}{\Vert \Phi_k \Vert_{F_k} +  \Vert v_k^1 - v_k^2 \Vert_{F_k}}. \]
This limit exists, up to a subsequence, by \eqref{estl1l2}. A first thing to notice is that $\tilde{\Phi}_0 \in L^{2^*}(\RR^n)$. This is a consequence of \eqref{p1p2H1}, of the scaling invariance of the $L^{2^*}$ norm and of the definition of $\tilde{\Phi}_k$ in \eqref{p1p2deftphi}. Therefore, we can integrate \eqref{p1p2eqtp0} against $V_{j,\xi_0}$ for all $ 0 \le j \le n$ and \eqref{eqlin} shows that there holds $\tilde{\lambda}_{0,j} = 0$ for all $0 \le j \le n$. Since $\tilde{\Phi}_0 \in L^{2^*}(\RR^n)$, there holds then $\tilde{\Phi }_0 \in H^1(\RR^n)$ and the Bianchi-Egnell \cite{BianchiEgnell} classification result applies and shows that 
\ben \label{p1p2noyau}
\tilde{\Phi}_0 \in \textrm{Span} \big\{ V_{j,\xi_0}, 0 \le j \le n \big\}.
\een
Similarly to what we did in the proof of Claim \ref{claimNphik}, we will now show that $\tilde{\Phi}_0 \in \textrm{Span} \big\{ V_{j,\xi_0}, 0 \le j \le n \big\}^{\perp}$. Since $\Phi_k \in K_k^{\perp}$ we can write, for any $R > 0$ and $0 \le j \le n$, that 
\ben  \label{p1p2ortho1}
\bal
&\int_{B_{\xi_k}(R \dk)} \left \langle \nabla Z_{j,k}, \nabla \Phi_k \right \rangle_g + h Z_{j,k} \Phi_k dv_g =  -  \int_{\partial B_{\xi_k}(R \dk)} \Phi_k \partial_\nu Z_{j,k} d\sigma_g \\
&- \int_{M \backslash B_{\xi_k}(R \dk)} (h-c_n S_g) Z_{j,k} \Phi_k dv_g - \int_{M \backslash B_{\xi_k}(R \dk)} \left( \triangle_g + c_nS_g \right)Z_{j,k} \Phi_k dv_g.
\eal
\een
Mimicking the computations that led to \eqref{C02s113} -- \eqref{C02s120} and using \eqref{lapZik} and the definition of the $\Vert \cdot \Vert_{F_k}$-norm in \eqref{normeFvek} one obtains that, for any $0 \le j \le n$:
\ben \label{p1p2ortho2}
\left| \int_{B_{\xi_k}(R \dk)} \left \langle \nabla Z_{j,k}, \nabla \Phi_k \right \rangle_g + h Z_{j,k} \Phi_k dv_g \right| \lesssim \Big( o(1) + R^{-2} \Big) \Vert \Phi_k \Vert_{F_k}.
\een
Dividing both sides of \eqref{p1p2ortho2} by $\Vert \Phi_k \Vert_{F_k} +  \Vert v_k^1 - v_k^2 \Vert_{F_k}$, using the definition of $\tilde{\Phi}_k$ in \eqref{p1p2deftphi} and the convergence of $\tilde{\Phi}_k$ to $\tilde{\Phi}_0$, letting $k \to +\infty$ and then $R \to + \infty$, we then obtain that $\tilde{\Phi}_0 \in \textrm{Span} \big\{ V_{j,\xi_0}, 0 \le j \le n \big\}^{\perp}$. With \eqref{p1p2noyau}, this gives that $\tilde{\Phi}_0 \equiv 0$. Using \eqref{p1p2maxphi0} this gives in turn that
\[ \Vert \phi_k^1 - \phi_k^2 \Vert_{F_k} = o \Big( \Vert v_k^1 - v_k^2 \Vert_{F_k} \Big), \]
which proves \eqref{12obj} in this case.

\bigskip 

\noindent  We now assume that there holds, up to a subsequence, that 
\ben \label{p1p2hypoxk}
\frac{d_{g_{\xi_k}}(\xi_k, x_k)}{\dk} \to + \infty
\een
as $ k \to + \infty$, where $x_k$ is given by \eqref{p1p2defxk}. Recall the definition of the operator $L_u$ introduced in \eqref{Ldefop}. Then, using \eqref{ponctuelphik} and \eqref{choixvek}, we rewrite \eqref{eqp1p2} as:
\ben \label{p1p2grand1}
\bal
L_u(\Phi_k) & = O \Bigg( \min\big( W_k^{2^*-2}, W_k \big) |\Phi_k| \Bigg) + o(|\Phi_k|) \\
&+ \left( \left| \sigma + \Lg T_k^1 \right|_g^2 - \left| \sigma + \Lg T \right|_g^2 \right) \left[ \Big(u + W_k + v_k^1 \Big)^{-2^*-1} - \Big(u + W_k + v_k^2 \Big)^{-2^*-1}\right] \\
&+ \Big(u + W_k + v_k^2 \Big)^{-2^*-1} \left( \left| \sigma + \Lg T_k^1 \right|_g^2 - \left| \sigma + \Lg T_k^2 \right|_g^2 \right) \\
&+ \sum_{j=0}^n \left( \lkja - \lkjb\right) \Big( \triangle_g + h \Big) Z_{j,k}. \\
\eal
\een
By \eqref{uetastableX}, the Green's function $G_u$ of $L_u$ satisfies the pointwise bounds \eqref{estLu}. We now write a Green representation formula for $L_u$ at $x_k$. Using \eqref{p1p2grand1}, \eqref{estLT1mT2}, \eqref{estl1l2}, \eqref{estLTkLT} below together with \eqref{defXintro} and \eqref{estLu} we get that there holds:
\ben \label{p1p2grand2}
\bal
|\Phi_k(x_k)| & \le o \Bigg( \Vert \Phi_k \Vert_{F_k} \Big(u + W_k \Big)(x_k)  \Bigg) + D_0 \big(\alpha + \eta \big) \Vert v_k^1 - v_k^2 \Vert_{F_k} \\
& + \int_M G_u(x_k,y)  \min\Big( W_k^{2^*-2}(y), W_k(y) \Big) |\Phi_k|(y) dv_g(y)
\eal
\een
for some positive $D_0$ that does not depend on $k$, and where $\eta$ and $\alpha$ are as in \eqref{introeta} and \eqref{defalpha}. Let $R >0$ be fixed. We have, using \eqref{p1p2hypoxk}, that:
\ben \label{p1p2grand3}
\bal
 \Bigg| \int_{M \backslash B_{\xi_k}(R \dk)} & G_u(x_k,y)  \min\Big( W_k^{2^*-2}(y), W_k(y) \Big) |\Phi_k|(y) dv_g(y) \Bigg| \\
 & \lesssim  \int_{M \backslash B_{\xi_k}(\sqrt{\dk})} G_u(x_k,y)  W_k(y) \Vert \Phi_k \Vert_{F_k} dv_g(y) \\
 & + \int_{B_{\xi_k}(\sqrt{\dk}) \backslash B_{\xi_k}(R \dk)} G_u(x_k,y)  W_k^{2^*-1}(y) \Vert \Phi_k \Vert_{F_k} dv_g(y) \\
 & \le D_1 \Big( \frac{1}{R^2} + \dk \Big) \Big( u + W_k \Big)(x_k) \Vert \Phi_k \Vert_{F_k},
\eal
\een
for some positive constant $D_1$ which does not depend on $k$ or on $R$. Independently, if we let $p > \frac{n}{4}$ be fixed, two H\"older inequalities together with \eqref{p1p2hypoxk} and \eqref{p1p2H1} show that
\ben \label{p1p2grand4}
\bal
\int_{B_{\xi_k}(R \dk)} G_u(x_k,y)  W_k^{2^*-2} |\Phi_k(y)| dv_g(y) \lesssim D_p R^{\frac{n+2}{2} - \frac{n}{p}} \eta \Vert v_k^1 - v_k^2 \Vert_{F_k} \dk^\pui d_{g_{\xi_k}}(\xi_k,x_k)^{2-n},
\eal
\een
for some positive constant $D_p$ which depends on $p$ but not on $k$ or $R$. Choose now $R_0 > 0$ such that $D_1 R_0^{-2} = \frac18$, where $D_1$ is given by \eqref{p1p2grand3}. Assume then that $\alpha$ and $\eta$ are small enough to have $D_p {R_0}^{\frac{n+2}{2} - \frac{n}{p}} \eta \le \frac18$ and $D_0 \big(\alpha + \eta \big) \le \frac18$, where $D_0$ and $D_p$ are given by \eqref{p1p2grand2} and \eqref{p1p2grand4}. Then, plugging \eqref{p1p2grand3} and \eqref{p1p2grand4} in \eqref{p1p2grand2} and using the definition \eqref{p1p2defxk} of $x_k$ gives:
\[ \Vert \phi_k^1 - \phi_k^2 \Vert_{F_k} \le \frac{1}{2} \Vert v_k^1 - v_k^2 \Vert_{F_k},\]
thus concluding the proof of \eqref{12obj}. The fact that \eqref{12obj} holds for large $k$ independent on the choice of $(t_k, \xi_k)_k$ follows by a standard contradiction argument, up to passing to a subsequence. 

\medskip

\noindent Now, the uniqueness property that defines $\phi(t_k,\xi_k,v_k)$ (and is stated in Proposition \ref{propptfixe1}) shows that, for any $k$ large enough, a function $\vp \in F(\vek,t_k,\xi_k) \cap  \{ \vp \in K_{k,t_k,\xi_k}^\perp, \Vert \vp \Vert_{H^1(M)} \le \frac{\vek}{4} \}$ solves \eqref{systreduit} if and only if it is a fixed-point of $\Psi_k$ defined in \eqref{defmapping}. Using \eqref{12obj}, Banach-Picard's fixed-point theorem asserts, for any $k$, the existence of such a fixed-point $\phi_k(t_k,\xi_k)$ as well as its uniqueness in $F(C \dk, t_k, \xi_k) \cap K_{k,t_k,\xi_k}^\perp \cap B_{H^1(M)}(0, C \dk)$. The estimates in \eqref{eqsystreduit} follow then from \eqref{propphik2}, \eqref{propestglob} and \eqref{choixvek}. Finally, the continuity of the mapping $(t, \xi) \in (0, + \infty) \times M \mapsto \phi_k(t,\xi) \in C^1(M)$ follows from direct arguments, using \eqref{C01step1}, standard elliptic theory and the uniqueness property of $\phi_k(t,\xi) $ in $F(C \dk, t_k, \xi_k) \cap K_{k,t_k,\xi_k}^\perp \cap B_{H^1(M)}(0, C \dk)$, thus concluding the proof of Proposition \ref{propsystreduit}.
\end{proof}

\medskip

\noindent One might surprised by the use of Banach-Picard's fixed-point theorem in the proof of Proposition \ref{propsystreduit}. In view of Proposition \ref{propestglob}, an application of Schauder's fixed-point theorem would seem preferable to construct a solution of \eqref{systreduit} -- and the proof would indeed be simpler. The reason for using Banach-Picard's theorem is that Schauder's theorem does not provide a preferred solution of \eqref{systreduit}, and in particular yields no uniqueness property on the remainder $\phi_k(t,\xi)$ thus constructed. A striking consequence is that if $\phi_k(t,\xi)$ is not uniquely determined in some way, one is not able to prove its continuity in the choice of $(t,\xi)\in (0, + \infty) \times M$. 

\noindent This remark is best understood anticipating a little on Sections \ref{DL} and \ref{argumentconclusif}. If we were to construct a remainder $\phi_k(t,\xi)$ by Schauder's fixed-point theorem it would still satisfy \eqref{eqsystreduit} and the (analogue of the) estimates of Proposition \ref{C02pres}, so that the expansions we perform in Section \ref{DL} would still hold true. In particular, one would obtain an exact analogue of \eqref{grosseeqlambda}, but where the error terms $R_k^{i}(t,p)$, $0 \le i \le n$, would only be bounded functions whose uniform bound goes to zero as $k \to +\infty$. Since $(t, \xi) \mapsto \phi_k(t,\xi)$ is not continuous, these $R_k^i(t,p)$ would be non-continuous too, and the usual final annihilating arguments would fail.

\medskip

\noindent The pointwise estimates obtained in Proposition  \ref{propestglob} still hold true for the function $\phi_k(t,\xi)$ given by Proposition \ref{propsystreduit}, when substituting $\vek$ by $\mk$. Estimate \eqref{estLTkLT} also remains true for the field of $1$-forms $T_{k,t,\xi}$ associated to $u + W_{k,t,\xi} + \phi_k(t,\xi)$ by \eqref{defZk1forme}. Proposition \ref{C02pres} gives the following control:

\begin{prop} \label{proplocalmieux}
Let $D >0$, $(t_k,\xi_k)_k$ be a sequence in $ [1/D, D]\times M$, and let $\phi_k = \phi_k \big(t_k,\xi_k \big)$ denote the function given by Proposition \ref{propsystreduit}. 
Let $(x_k)_k$ be any sequence of points in $B_{\xi_k}(2 \sqrt{\dk})$. Then there holds:
\begin{itemize}
\item If $n \ge 7$:
\be
\bal
|\phi_k(x_k)| & \lesssim \dk + \dk \Vert \nabla f \Vert_{L^\infty(2 r_k)} + \dk \Vert h - c_n S_g \Vert_{L^\infty(2 r_k)} +\dk \\
&  + \left( \frac{\dk}{\tk(x_k)}\right)^2 + \Bigg[ \dk^{\frac{n}{2}}  +  \dk \Vert \nabla f \Vert_{L^\infty(2 r_k)} \\
& +  \Vert h - c_n S_g \Vert_{L^\infty(2 r_k)} \Big( \tk(x_k)^2 + \dk^2 \ln \left( \frac{\tk(x_k)}{\dk} \right) \Big)  + \tk(x_k)^4 \mathds{1}_{nlcf} \Bigg] W_k(x_k). \\
\eal
\ee
\item If $n = 6$:
\be
\bal
|\phi_k(x_k)| & \lesssim \dk + \dk \Vert \nabla f \Vert_{L^\infty(2 r_k)} + \dk \big \Vert h - \frac15 S_g - 2f u \big \Vert_{L^\infty(2 r_k)}   \\
& + \Bigg[ \dk^{3}  +  \dk \Vert \nabla f \Vert_{L^\infty(2 r_k)}  + \big \Vert h - \frac15 S_g - 2fu \big \Vert_{L^\infty(2 r_k)} \Big( \tk(x_k)^2 + \dk^2 \ln \left( \frac{\tk(x_k)}{\dk}\right) \Big)\Bigg] W_k(x_k). \\
\eal
\ee
\end{itemize}
\end{prop}
\begin{proof}
Proposition \ref{propestglob} shows that there holds:
\[ \Vert \phi_k \Vert_{L^\infty(B_{\xi_k} (2 r_k) \backslash B_{\xi_k}( \sqrt{\dk}))} + \sqrt{\dk} \Vert \nabla \phi_k \Vert_{L^\infty(B_{\xi_k} (2 r_k) \backslash B_{\xi_k}( \sqrt{\dk}))}\lesssim \dk ,\]
and the result then follows from an application of  Proposition \ref{C02pres}.
\end{proof}

\noindent To be able to perform the asymptotic expansion along $K_{k,t,\xi}^\perp$ in the next section, we will need more precise estimates than \eqref{ponctuelphik} on the behaviour of $\phi_k(t,\xi)$ at distances from $\xi_k$ which are large compared to $\sqrt{\mk}$. We quantify more precisely the fall-off of $\phi_k(t,\xi)$ far away from the center of the bubbles in the next Proposition:

\begin{prop} \label{proploinmieux}
Let $D >0$, $(t_k,\xi_k)_k$ be a sequence in $ [1/D, D]\times M$, and let $\phi_k = \phi_k \big(t_k,\xi_k \big)$ denote the function given by Proposition \ref{propsystreduit}. Let $(R_k)_k$, $R_k \ge 1$, denote a sequence of positive numbers. There holds:
\ben \label{falloffphik}
\Vert \phi_k \Vert_{L^\infty(M \backslash B_{\xi_k}(R_k \sqrt{\dk}))} \lesssim \frac{\dk}{R_k^2} + R_k^2 \dk^{2} + \dk^{\pui} r_k^{-n}.
\een
\end{prop}

\begin{proof}
Let $(t_k,\xi_k)_k$ be a sequence in $ [1/D, D]\times M$ and let $(R_k)_k$, $R_k \ge 1$, denote a sequence of positive numbers. Let $(x_k)_k$ be a sequence of points in $M$ satisfying $d_{g_{\xi_k}} (\xi_k, x_k) \ge R_k \sqrt{\dk}$. Let again $G_u$ denote the Green's function of the operator $L_u$ defined in \eqref{Ldefop}. We use \eqref{C036} and \eqref{estLTkLTordre2} below to write that there holds, for some positive constant $C$ that neither depends on $k$ nor on $\eta$ as in \eqref{introeta}, that:
\ben \label{plm1}
\bal
\int_{M} G_u(x_k,y)  \frac{ |\Lg T_k + \sigma|_g^2 - |\Lg T + \sigma|_g^2 }{\left( u + W_k + v_k \right)^{2^*+1}}(y)  dv_g(y) \le C \Big( \dk^{\frac{n}{2}} \tk(x_k)^{2-n} \\
 + \eta \Vert \phi_k \Vert_{L^\infty(M \backslash B_{\xi_k}(R_k \sqrt{\dk}))} + \dk^2 + \eta R_k^2 \dk^{2} + \dk^{\frac{n-1}{2}} \tk(x_k)^{3-n} \Big).
\eal
\een
Here, $T_k = T_{k,t_k,\xi_k}$ is the solution of the $1$-form equation in \eqref{systreduit}. Assume now that $\eta$ in \eqref{introeta} is small enough so as to have
\be 
C\eta \le \frac{1}{2},
\ee
where $C$ is the constant appearing in \eqref{plm1}. Estimate \eqref{falloffphik} then follows from \eqref{plm1} by choosing the sequence $(x_k)_k$ to be such that
\[ |\phi_k(x_k)| =  \Vert \phi_k \Vert_{L^\infty(M \backslash B_{\xi_k}(R_k \sqrt{\dk}))} \]
and by writing a representation formula for $L_u$ at $x_k$.  The representation formula is written for \eqref{nouveqphik} and the terms appearing in it are estimated by using \eqref{C030}--\eqref{C035} and \eqref{eqsystreduit}.
\end{proof}
\noindent The considerations in the remark following Proposition \ref{propestglob} apply here too: Propositions \ref{propsystreduit}, \ref{proplocalmieux} and \ref{proploinmieux} still hold true for arbitrary choices of $h, \pi \in C^0(M)$, $f \in C^1(M)$, $\sigma \in C^0(M)$ and $Y \in C^0(M)$ once $X$ is given by \eqref{defXintro}.

\section{Expansion of the Kernel coefficients} \label{DL}

\noindent Let $D >0$. We let, for any $p \in \overline{B_0(1)}$:
\ben \label{definitionyk}
y_k = \exp_{\xi_k}^{g_{\xi_k}}(\beta_k p),
\een
where $\beta_k$ is defined in \eqref{propbetak} and $(\xi_k)_k$ is the sequence chosen in  Section \ref{notations}. Throughout this section the functions $W_{k,t,y_k}$ and $Z_{i,k,t,y_k}$ defined in \eqref{bulle} and \eqref{defZk} will be denoted by $W_{k,t,p}$ and $Z_{i,k,t,p}$, with $\dk(t)$ again given by \eqref{defdkyk}. For any $(t,p) \in [1/D, D] \times \overline{B_0(1)}$ we will denote by $\phi_k(t,p)$ the function $\phi_k(t,y_k)$ given by Proposition \ref{propsystreduit} when $y_k$ is given by \eqref{definitionyk}, and as before we let  $u_{k,t,p} = u + W_{k,t,p} + \phi_k(t,p)$. Proposition \ref{propsystreduit} shows that there exist real numbers $(\lki(t,p))_{0 \le i \le n}$ such that $u_{k,t,p}$ satisfies:
\ben \label{systnoyaupres}
\left \{ 
\bal
& \left( \triangle_ g+ h \right) u_{k,t,p} = f  u_{k,t,p}^{2^*-1}  + \frac{|\Lg T_{k,t,p} + \sigma|_g^2 + \pi^2}{ u_{k,t,p}^{2^*+1}} + \sum_{i=0}^n \lki(t,p) \left( \triangle_g + h \right)Z_{i,k,t,p}, \\
& \Dg T_{k,t,p} = u_{k,t,p}^{2^*} X + Y.\\
\eal \right. 
\een
Since $\phi_k(t,p)$, $W_{k,t,p}$ and $Z_{i,k,t,p}$ are continuous in the choice of $(t,p)$, then so are the $(\lki(t,p))_{0\le i \le n}$.

\medskip

\noindent In this section we conclude the proof of Theorem \ref{thprincipal} by showing that for any $k$ there exists $(t_k,p_k)$ such that $\lki(t_k,p_k) = 0$ for any $0 \le i \le n$. By \eqref{systnoyaupres}, the function $u_{k,t_k,p_k}$ will therefore provide the desired solution of \eqref{intro1} and this will conclude the proof of Theorem \ref{thprincipal}. We do this by performing an asymptotic expansion of the $(\lki(t,p))_{0 \le i \le n}$ as $k \to + \infty$ towards some limiting function that possesses zeroes. As before, we distinguish between the $n = 6$ and $n \ge 7$ cases. All the expansions obtained in this section will be uniform in the choice of $(t,p)\in [1/D, D] \times \overline{B_0(1)}$. In this section we will often apply previously obtained results, such as Proposition \ref{propsystreduit}. They will always be applied for the sequence $(t,y_k)$ where $y_k$ is given by \eqref{definitionyk}. As mentioned in the Introduction, the expansions in this section do not rely on the assumption $r_k \to 0$.

\subsection{The $n \ge 7$ case.} \label{DLenergie}

We start by re-writing the scalar equation in \eqref{systnoyaupres} as:
\ben \label{DL1}
\bal
 &\sum_{i=0}^n \lki(t,p) \left( \triangle_g + h \right) Z_{i,k,t,p} = \\
 & \left( \triangle_g + h \right)W_{k,t,p} - f(y_k) W_{k,t,p}^{2^*-1} + \left( f(y_k) - f \right) W_{k,t,p}^{2^*-1} \\
 & - f \Bigg[ \big( u + W_{k,t,p} + \phi_k(t,p)\big)^{2^*-1} - \big( u + W_{k,t,p} \big)^{2^*-1} - (2^*-1) \big(u + W_{k,t,p} \big)^{2^*-2} \phi_k(t,p)\Bigg] \\
 & - f \Bigg[ \big(u + W_{k,t,p} \big)^{2^*-1} - u^{2^*-1} - W_{k,t,p}^{2^*-1} \Bigg] \\
  & + \left( \triangle_g + h \right)\phi_k(t,p) - (2^*-1) f(y_k) W_{k,t,p}^{2^*-2} \phi_k(t,p) \\
 & - (2^*-1) f \Bigg[ \big( u + W_{k,t,p} \big)^{2^*-2} - W_{k,t,p}^{2^*-2} \Bigg] \phi_k(t,p) \\
 & + (2^*-1) \big( f(y_k) - f \big) W_{k,t,p}^{2^*-2} \phi_k(t,p) \\
 & + \frac{|\Lg T + \sigma|_g^2 + \pi^2}{u^{2^*+1}} - \frac{|\Lg T_{k,t,p} + \sigma |_g^2 + \pi^2}{\big(u + W_{k,t,p} + \phi_k(t,p) \big)^{2^*+1}} . \\
\eal
\een 
We integrate each side of equation \eqref{DL1} against $Z_{i,k,t,p}$, for a given $0 \le i \le n$. For any $0 \le i \le n$ we let:
\ben \label{DL2}
\bal
I_{1,i} & = \int_M \left( \left( \triangle_g + h \right)W_{k,t,p} - f(y_k) W_{k,t,p}^{2^*-1} + \left( f(y_k) - f \right) W_{k,t,p}^{2^*-1} \right)Z_{i,k,t,p} dv_g, \\
I_{2,i} & = - \int_M f \Bigg[ \big( u + W_{k,t,p} + \phi_k(t,p)\big)^{2^*-1} - \big( u + W_{k,t,p} \big)^{2^*-1} \\
& \qquad \qquad \qquad \qquad- (2^*-1) \big(u + W_{k,t,p} \big)^{2^*-2} \phi_k(t,p)\Bigg]  Z_{i,k,t,p} dv_g, \\
I_{3,i} & = - \int_M  f \Bigg[ \big(u + W_{k,t,p} \big)^{2^*-1} - u^{2^*-1} - W_{k,t,p}^{2^*-1} \Bigg] Z_{i,k,t,p} dv_g, \\
I_{4,i} & = \int_M \Bigg[  \left( \triangle_g + h \right)\phi_k(t,\xi) - (2^*-1) f(\xi_k) W_{k,t,\xi}^{2^*-2} \phi_k(t,\xi) \Bigg] Z_{i,k,t,\xi} dv_g, \\
I_{5,i} & = - (2^*-1) \int_M f \Bigg[ \big( u + W_{k,t,p} \big)^{2^*-2} - W_{k,t,p}^{2^*-2} \Bigg] \phi_k(t,p) Z_{i,k,t,p} dv_g, \\
I_{6,i} & = (2^*-1) \int_M  \big( f(y_k) - f \big) W_{k,t,p}^{2^*-2} \phi_k(t,p) Z_{i,k,t,p} dv_g, \\
I_{7,i} & = \int_M \left(  \frac{|\Lg T + \sigma|_g^2 + \pi^2}{u^{2^*+1}} - \frac{|\Lg T_{k,t,p} + \sigma |_g^2 + \pi^2}{\big(u + W_{k,t,p} + \phi_k(t,p) \big)^{2^*+1}} \right) Z_{i,k,t,p} dv_g.
\eal
\een
We will need precise asymptotic expansions of the $\lki(t,p)$ and so will always distinguish in the following between the $i = 0$ and $1 \le i \le n$ cases, to take into account the different decay of the $Z_{i,k,t,p}$. Our aim is to show that the dominant contributions in the expansion of the $\lki(t,p)$ come from the integrals $I_{1,i}$ and $I_{3,i}$ in \eqref{DL2}.

\bigskip

\noindent The integrals in \eqref{DL2} are computed in a series of Claims. We start with the computations of $I_{1,i}$:
\begin{claim}
There holds:
\ben \label{I11}
\bal
I_{1,0} = & \frac{8(n-1)}{(n-2)(n-4)} K_n^{-n} f(\xi_0)^{- \frac{n}{2}} \tau_k \mk^2 H(p) t^2 \\
&- \frac{1}{3} \frac{n (n-2)}{(n-4)(n-6)} f(\xi_0)^{-1 - \frac{n}{2}} K_n^{-n} |W_g(\xi_0)|_g^2 \mk^4 t^4 + o(\tau_k \dk^2) +  o(\dk^4 \mathds{1}_{nlcf}) + o(\dk^\pui),
\eal
\een
and, for any $1 \le i \le n$:
\ben \label{I12}
\bal
I_{1,i} =&\frac{2n(n-1)}{n-4} K_n^{-n} f(\xi_0)^{- \frac{n}{2}} \frac{\tau_k}{\beta_k} \mk^3 \nabla_i H(p) t^3 \\
& - f(\xi_0)^{-1-\frac{n}{2}} K_n^{-n} \frac{n^2 (n-2)^2}{24(n-4)(n-6)} \nabla_i (|W(\cdot)|_g^2)(\xi_0) \mk^5 t^5 + o(\dk^5 \mathds{1}_{nlcf}) + o(\dk^{\frac{n}{2}}).\\
\eal
\een
Here $\xi_0 = \lim_{k \to + \infty} \xi_k$ and $\tau_k$ and $\mk$ are as in \eqref{defmk}.
\end{claim}
\noindent In \eqref{I11} and \eqref{I12} $W_g(\xi_0)$ denotes the Weyl tensor at $\xi_0$ (which vanishes if $(M,g)$ is locally conformally flat) and $K_n^{-n}$ is the $H^1$-energy of the standard bubble defined by:
\ben \label{defKnn}
K_n = \sqrt{\frac{4}{n(n-2) \omega_n^{\frac{2}{n}}}}, 
\een
where $\omega_n$ is the volume of the standard unit $n$-sphere.
\begin{proof}
Using the definition of conformal normal coordinates as in \eqref{confnorm} and \eqref{propLambda}, and using \eqref{bulle} and \eqref{defZk}, there holds, for any fixed $t >0$ and $\xi \in M$, that:
\[ \bal
\frac{\partial}{\partial t} W_{k,\delta, \xi} & = \frac{n-2}{2 t} Z_{0,k,t,\xi}, \\
\frac{\partial}{\partial \xi_i} W_{k,t,\xi} & = \frac{f(\xi)}{n \delta_k(t)} Z_{i,k,t,\xi} + O\big( \dk^\pui r_k^{1-n} \mathds{1}_{r_k \le d_{g_\xi} \le 2 r_k } \big) + O \big( |\nabla f(\xi)|_g W_{k,t,\xi} \big) + O \big( d_{g_\xi}(\xi, \cdot) W_{k,t,\xi}).
\eal \]
These expressions are for instance obtained from the estimations in Appendix A in Esposito-Pistoia-V\'etois \cite{EspositoPistoiaVetois} and those in Section $4$ of Robert-V\'etois \cite{RobertVetois}. Estimates \eqref{I11} and \eqref{I12} follow then from straightforward computations using \eqref{defhintro}, \eqref{deffintro} and \eqref{relationrkmk}.
\end{proof}

\bigskip

\begin{claim}
There holds:
\ben \label{I22bis} 
I_{2,0} = o(\dk^{\pui}) \textrm{ if } (M,g) \textrm{ is l.c.f. or if } n\le 10,
\een
\ben \label{I23bis}
I_{2,0} = o(\dk^4) \textrm{ if } (M,g) \textrm{ is not l.c.f. and } n\ge 11, 
\een
and, for $1 \le i \le n$,
\ben \label{I24}
I_{2,i} = o(\dk^{\frac{n}{2}}) + o(\dk^5 \mathds{1}_{nlcf}) . 
\een
\end{claim}
\begin{proof}
First, we have that:
\be
\bal
 f \Bigg[ \big( u + W_{k,t,p} + \phi_k(t,p)\big)^{2^*-1} - \big( u + W_{k,t,p} \big)^{2^*-1} - (2^*-1) \big(u + W_{k,t,p} \big)^{2^*-2} \phi_k(t,p)\Bigg] \\
  \lesssim \big(u + W_{k,t,p} \big)^{2^*-3} |\phi_k(t,p)|^2. 
\eal
 \ee
 Using \eqref{eqsystreduit} and since $n \ge 7$ and $ |Z_{0,k,t,p}|  \lesssim W_{k,t,p}$, we write that:
\[ \bal
\int_{M \backslash B_{y_k}(\sqrt{\dk})} \big(u + W_{k,t,p} \big)^{2^*-3} |\phi_k(t,p)|^2 |Z_{0,k,t,p}| dv_g \lesssim \dk^2 \int_{B_{y_k}(2 r_k) \backslash B_{y_k}(\sqrt{\dk})}  |Z_{0,k,t,p}| dv_g \\
+ \dk^2 \int_{B_{y_k}(2 r_k) \backslash B_{y_k}(\sqrt{\dk})}  W_{k,t,p}^{2^*} dv_g 
\eal\]
which gives then
\ben \label{I21}
\int_{M \backslash B_{y_k}(\sqrt{\dk})} \big(u + W_{k,t,p} \big)^{2^*-3} |\phi_k(t,p)|^2 |Z_{0,k,t,p}| dv_g = o(\dk^{\frac{n}{2}}).
\een
We now use Proposition \ref{proplocalmieux} and \eqref{deffintro} to write that there holds:
\ben \label{I22}
\bal
\int_{ B_{y_k}(\sqrt{\dk})} W_{k,t,p}^{2^*-3} |\phi_k(t,p)|^2 |Z_{0,k,t,p}| dv_g & \lesssim o(\dk^{\frac{n}{2}}) + o(\dk^4 \mathds{1}_{nlcf}) +  \dk^4  \Vert h - c_n S_g \Vert_{L^\infty(2 r_k)}^2  .
\eal
\een
Assume first that $(M,g)$ is locally conformally flat or that $n \le 10$. Then, using \eqref{defhintro} there holds $\Vert h - c_n S_g \Vert_{L^\infty(2 r_k)}^2 = \dk^{n-6}$ so that with \eqref{I21} we obtain \eqref{I22bis}.
Assume then that $n \ge 11$ and that $(M,g)$ is not locally conformally flat. Then \eqref{defhintro} shows that $\Vert h - c_n S_g \Vert_{L^\infty(2 r_k)}^2 = \dk^{4}$ so that gathering \eqref{I21} and \eqref{I22} and since $n \ge 7$, we obtain \eqref{I23bis}.

\medskip 

\noindent Let now $1 \le i \le n$ be fixed. On one side, using \eqref{eqsystreduit} and since $n \ge 7$ and $|Z_{i,k,t,p}| \lesssim W_{k,t,p}$ we have that
\[ \int_{M \backslash B_{y_k}(\sqrt{\dk})} \big(u + W_{k,t,p} \big)^{2^*-3} |\phi_k(t,p)|^2 |Z_{i,k,t,p}| dv_g = o(\dk^{\frac{n}{2}}). \]
On the other side, using again Proposition \ref{proplocalmieux}, \eqref{deffintro} and \eqref{defhintro} shows that 
\[ \int_{ B_{y_k}(\sqrt{\dk})} W_{k,t,p}^{2^*-3} |\phi_k(t,p)|^2 |Z_{i,k,t,p}| dv_g = o (\dk^{\frac{n}{2}}) + o(\dk^5 \mathds{1}_{nlcf}). \]
Combining the latter estimates we get \eqref{I24}.
\end{proof}
\bigskip

\begin{claim}
There holds:
\ben \label{I31}
I_{3,0} = - \frac{1}{2} (n-2)^2 \left( n(n-2)\right)^{\pui} \omega_{n-1} f(\xi_0)^{1 - \frac{n}{2}} u(\xi_0) \dk^{\pui} + o(\dk^{\pui})
\een
and, for any $1 \le i \le n$,
\ben \label{I32}
I_{3,i} = - C(n) f(\xi_0)^{- \frac{n}{2}} \nabla_i u(\xi_0) \dk^{\frac{n}{2}} + o(\dk^{\frac{n}{2}}),
\een
where $C(n)$ is some explicit, positive, numerical constant only depending on $n$.
\end{claim}
\begin{proof}
\noindent Write that:
\[ 
\bal 
\left |\int_{M \backslash B_{y_k}(\sqrt{\dk})} f \Bigg[ \big(u + W_{k,t,p} \big)^{2^*-1} - u^{2^*-1} - W_{k,t,p}^{2^*-1} \Bigg] Z_{i,k,t,p} dv_g \right| & \lesssim \int_{M \backslash B_{y_k}(\sqrt{\dk})} W_{k,t,p} |Z_{i,k,t,p}| dv_g   \\
& \lesssim \left \{
\bal
& o(\dk^{\pui}) \textrm{ if } i = 0, \\
& o( \dk^{\frac{n}{2}}) \textrm{ if } 1 \le i \le n .\\
\eal \right.
\eal
\]
The integral over $B_{y_k}(\sqrt{\dk})$ is then easily computed using the explicit expressions of $W_{k,t,p}$ and $Z_{i,k,t,p}$ given in \eqref{bulle} and \eqref{defZk}. Straightforward computations give \eqref{I31} and \eqref{I32}.
\end{proof}

\bigskip

\begin{claim}
There holds:
\ben \label{I40}
 I_{4,0} = o (\dk^\pui) + o(\dk^4 \mathds{1}_{nlcf}), 
\een
and, for any $1 \le i \le n$:
\ben \label{Int4}
I_{4,i} =  o(\dk^{\frac{n}{2}})  + o(\dk^5 \mathds{1}_{nlcf}).
\een
\end{claim}
\begin{proof}
\noindent 
Integrating by parts, $I_{4,i}$ is best written as:
\ben \label{I4prep}
 I_{4,i}  = \int_M \Bigg[  \left( \triangle_g + h \right) Z_{i,k,t,p} - (2^*-1) f(y_k) W_{k,t,p}^{2^*-2} Z_{i,k,t,p} \Bigg] \phi_k(t,p)  dv_g. 
 \een
Again, we start with $I_{4,0}$. Using \eqref{eqsystreduit}, there holds: 
\[ \int_M \left( \dk^\pui r_k^{-n} \mathds{1}_{r_k \le d_{g_{y_k}} \le 2 r_k } +  \dk^{\pui} \mathds{1}_{nlcf} \right) |\phi_k(t, \xi)| dv_g   = o(\dk^\pui). \]
Using  \eqref{defhintro} and \eqref{eqsystreduit} it is easily seen that there also holds:
\[ \int_M (h - c_n S_g) Z_{0,k,t,p} \phi_k(t,p) dv_g = o(\dk^\pui) + o(\dk^4 \mathds{1}_{nlcf}). \]
Also, \eqref{propLambdaSg} shows that $|S_{g_{y_k}}| = O(d_{g_{y_k}}(y_k,\cdot)^2)$, so that with \eqref{eqsystreduit} we obtain:
\[ \int_M \Lambda_{y_k}^{2^*-2} c_n S_{g_{y_k}} Z_{0,k,t,p} \phi_k(t,p) = o(\dk^\pui) + o(\dk^4 \mathds{1}_{nlcf}).\]
Combining \eqref{lapZik} with \eqref{I4prep} yields in the end \eqref{I40}.

\medskip

\noindent Let now $1 \le i \le n$. Because of \eqref{relationrkmk} and \eqref{eqsystreduit} we have:
\ben \label{I4i1}
 \int_M  \dk^{\frac{n}{2}} r_k^{-n-1} \mathds{1}_{r_k \le d_{g_{y_k}} \le 2 r_k }  |\phi_k(t,p)| dv_g = o(\dk^{\frac{n}{2}}). 
 \een
There holds that 
\ben \label{I4i2} 
\bal
\int_M \Big| &(h-c_n S_g) Z_{i,k,t,p} \phi_k(t,p)\Big| dv_g \\
& = \int_{M \backslash B_{y_k}(\sqrt{\dk})} \left| (h-c_n S_g) Z_{i,k,t,p} \phi_k(t,p)\right| dv_g + \int_{ B_{y_k}(\sqrt{\dk})} \left| (h-c_n S_g) Z_{i,k,t,p} \phi_k(t,p)\right| dv_g \\
& = o(\dk^{\frac{n}{2}}) + o(\dk^5 \mathds{1}_{nlcf}), \\
\eal
  \een
where the first integral is estimated using \eqref{eqsystreduit} and the second one is estimated using Proposition \ref{proplocalmieux} and \eqref{defhintro}. Assume that $(M,g)$ is not locally conformally flat. Write again that:
\ben \label{I4i4} \bal
 \int_M  c_n & \Lambda_{y_k} S_{g_{y_k}} Z_{i,k,t, p} \phi_k(t,p) dv_g \\
 & = \int_{M \backslash B_{y_k}(\sqrt{\dk})}  c_n \Lambda_{y_k} S_{g_{y_k}} Z_{i,k,t, p} \phi_k(t,p) dv_g + \int_{ B_{y_k}(\sqrt{\dk})}  c_n \Lambda_{y_k} S_{g_{y_k}} Z_{i,k,t, p} \phi_k(t,p) dv_g \\
 & = o(\dk^{\frac{n}{2}}) + o(\dk^5 \mathds{1}_{nlcf}),
\eal
\een
where as before the first integral is estimated using \eqref{eqsystreduit} and the second one using Proposition \ref{proplocalmieux}, \eqref{defhintro} and \eqref{deffintro}. Note, as a simple computation shows, that \eqref{eqsystreduit} alone would not be enough to estimate the integral over $B_{y_k}(\sqrt{\dk})$ with the desired precision. Similarly, one obtains that
\ben \label{I4i5}
 \int_M \dk^{\frac{n}{2}} \tk(y)^{2-n} |\phi_k(t,p)|(y) dv_g(y) = o(\dk^{\frac{n}{2}}) +  o(\dk^5 \mathds{1}_{nlcf}).  
 \een
Combining \eqref{lapZik} with \eqref{I4i1}, \eqref{I4i2}, \eqref{I4i4} and \eqref{I4i5} one obtains \eqref{Int4}.
\end{proof}

\bigskip

\noindent Let us now estimate $I_{5,i}$. For $n \ge 6$, there always holds:
\[ \left| \big(u+ W_{k,t,p} \big)^{2^*-2} - W_{k,t,p}^{2^*-2} \right| = O(1). \]
Hence, for any $0 \le i \le n$:
\[ \left| I_{5,i} \right| \lesssim \int_M |Z_{i,k,t,p}| |\phi_k(t,p)| dv_g . \]
Mimicking the computations that led to \eqref{I4i2} one obtains, here also, that
\ben \label{I50}
I_{5,0} = o(\dk^{\frac{n-2}{2}})  + o(\dk^4 \mathds{1}_{nlcf})
\een
and that:
\ben  \label{Int5i}
I_{5,i} = o(\dk^{\frac{n}{2}}) + o(\dk^5 \mathds{1}_{nlcf}).
\een

\bigskip

\noindent By constrast with the previous integrals, $I_{6,i}$ is easily estimated using \eqref{deffintro} and \eqref{eqsystreduit}. Straightforward computations give indeed that for any $0 \le i \le n$:
\ben \label{I6}
I_{6,i} = o(\dk^{\frac{n}{2}}).
\een

\bigskip
\begin{claim}
There holds that:
\ben \label{I76}
I_{7,i} = \left \{ 
\bal
& o(\dk^\pui) \textrm{ if } i =0 \\
& o(\dk^{\frac{n}{2}}) \textrm{ if } 1 \le i \le n. \\
\eal \right.
\een
\end{claim}
\begin{proof}
For any $0 \le i \le n$, we write:
\ben \label{I71}
\bal 
I_{7,i} & = \int_M  \left( |\Lg T + \sigma|_g^2 + \pi^2 \right) \left( u^{-2^*-1} - \big(u + W_{k,t,p} + \phi_k(t,p) \big)^{-2^*-1} \right) Z_{i,k,t,p} dv_g  \\
& + \int_M \big(u + W_{k,t,p} + \phi_k(t,p) \big)^{-2^*-1}\left(  |\Lg T + \sigma|_g^2 - |\Lg T_{k,t,p} + \sigma |_g^2 \right) Z_{i,k,t,p} dv_g \\
& := I_{7,i}^1 + I_{7,i}^2. \\
\eal
\een
We split the integral $I_{7,i}^1$ into an integral in $B_{y_k}(\sqrt{\dk})$ and another in $M \backslash B_{y_k}(\sqrt{\dk})$. Since the integrand of $I_{7,i}^1$ is bounded by $|Z_{i,k,t,p}|$  in  $B_{y_k}(\sqrt{\dk})$  (up to some positive constant that does not depend on $k$) and since 
\[ \left| u^{-2^*-1} - \big(u + W_{k,t,p} + \phi_k(t,p) \big)^{-2^*-1} \right| \lesssim W_{k,t,p} + |\phi_k(t,p)| \textrm{ in } M \backslash B_{y_k}(\sqrt{\dk}), \]
there holds with \eqref{eqsystreduit} that:
\ben \label{I72}
I_{7,0}^1 = o(\dk^\pui),
\een
and that, for $1 \le i \le n$,
\ben \label{I73}
I_{7,i}^1 = o(\dk^{\frac{n}{2}}).
\een
To compute $I_{7,i}^2$, we again split the integration domain into $B_{y_k}(\sqrt{\dk})$ and  $M \backslash B_{y_k}(\sqrt{\dk})$. On the one side, using \eqref{eqsystreduit} and \eqref{estLTkLT}, we obtain that:
\ben \label{I74}
 \bal
 \int_{B_{y_k}(\sqrt{\dk})}  \big(u + W_{k,t,p} + \phi_k(t,p) \big)^{-2^*-1}& \left(  |\Lg T + \sigma|_g^2 - |\Lg T_{k,t,p} + \sigma |_g^2 \right) Z_{i,k,t,p} dv_g \\
& = \left \{
\bal
& o(\dk^\pui) \textrm{ if } i =0 \\
& o(\dk^{\frac{n}{2}}) \textrm{ if } 1 \le i \le n. \\
\eal \right.
\eal \een
On the other side, using again \eqref{estLTkLT}, we get:
\ben  \label{I75}
\bal
\int_{M \backslash B_{y_k}(\sqrt{\dk})}  \big(u + W_{k,t,p} + \phi_k(t,p) \big)^{-2^*-1}& \left(  |\Lg T + \sigma|_g^2 - |\Lg T_{k,t,p} + \sigma |_g^2 \right) Z_{i,k,t,p} dv_g \\
& = \left \{ 
\bal
& o(\dk^\pui) \textrm{ if } i =0 \\
& o(\dk^{\frac{n}{2}}) \textrm{ if } 1 \le i \le n. \\
\eal \right.
\eal
\een
Combining \eqref{I71}--\eqref{I75} gives \eqref{I76}.
\end{proof}

\subsection{The $n = 6$ case.}

In the $6$-dimensional case we need to take into account the compensation phnomenon for system \eqref{intro1}, since we need to push the asymptotic estimates of the $\lki(t,p)$ one order further. For this, we write the scalar equation of \eqref{systnoyaupres} as:
\[ \bal
\sum_{i=0}^6 \lki(t,p) &\left( \triangle_g + h \right) Z_{i,k,t,p} = \\
& + \left( \triangle_g + h-2fu \right)W_{k,t,p} - f(y_k) W_{k,t,p}^{2} + \left( f(y_k) - f \right) W_{k,t,p}^{2} \\
& + f \phi_k(t,p)^2 \\
& +   \left( \triangle_g + h - 2 fu \right)\phi_k(t,p) - 2 f(y_k) W_{k,t,p} \phi_k(t,p) \\
& + 2 \big( f(y_k) - f \big) W_{k,t,p} \phi_k(t,p) \\
& +  \frac{|\Lg T + \sigma|_g^2 + \pi^2}{u^{4}} - \frac{|\Lg T_{k,t,p} + \sigma |_g^2 + \pi^2}{\big(u + W_{k,t,p} + \phi_k(t,p) \big)^{4}}, \\
\eal\]
and we let, for $0 \le i \le n$:
\[
\bal
J_{1,i} & = \int_M \Bigg[ \left( \triangle_g + h-2fu \right)W_{k,t,p} - f(y_k) W_{k,t,p}^{2} + \left( f(y_k) - f \right) W_{k,t,p}^{2} \Bigg] Z_{i,k,t,p} dv_g, \\
J_{2,i} & = \int_M f \phi_k(t,p)^2 Z_{i,k,t,p} dv_g, \\
J_{3,i} & = \int_M \Bigg[  \left( \triangle_g + h - 2 fu \right)\phi_k(t,p) - 2 f(y_k) W_{k,t,p} \phi_k(t,p)  \Bigg] Z_{i,k,t,p}  dv_g, \\
J_{4,i} & = \int_M 2 \big( f(y_k) - f \big) W_{k,t,p} \phi_k(t,p) Z_{i,k,t,p} dv_g, \\
J_{5,i} & = \int_M  \left(  \frac{|\Lg T + \sigma|_g^2 + \pi^2}{u^{4}} - \frac{|\Lg T_{k,t,p} + \sigma |_g^2 + \pi^2}{\big(u + W_{k,t,p} + \phi_k(t,p) \big)^{4}} \right) Z_{i,k,t,p} dv_g. \\
\eal
\]
\noindent These integrals are again computed in a series of Claims. The situation in the $6$-dimensional case is different, since an additional contribution in the expansion of the $\lki(t,p)$ comes from the integral $J_{5,i}$. 

\medskip

\noindent The computations of $J_{1,i}$ and $J_{2,i}$ follows the exact same lines than in the $n \ge 7$ case. Using \eqref{bulle}, \eqref{defZk}, \eqref{deffintro}, \eqref{relationrkmk} and \eqref{eqsystreduit} it is esily seen that there holds:
\ben \label{J10} 
J_{1,0} =  - 5 K_6^{-6} f(\xi_0)^{-3} H(p) \tau_k \dk^2  + o(\dk^3) ,
\een
and, for $1 \le i \le 6$,
\ben \label{J1i}
J_{1,i} = - 30 K_6^{-6} f(\xi_0)^{-3} \nabla_i H(p) \frac{\tau_k}{\beta_k} \dk^3 + o(\dk^4),
\een
where $K_6^{-6}$ is defined in \eqref{defKnn}. 
Also, mimicking the computations that led to \eqref{I24} gives here as well:
\ben \label{J2}
J_{2,i} = \left \{
\bal
& o(\dk^3), \textrm{ if } i =0, \\
& o(\dk^4), \textrm{ if } 1 \le i \le 6.
\eal \right.
\een

\bigskip

\begin{claim}
There holds:
\ben \label{J40}
J_{3,0} = o(\dk^3),
\een
and, for any $1 \le i \le 6$:
\ben \label{J4i}
J_{3,i} = o(\dk^4).
\een
\end{claim}
\begin{proof}
\noindent As for the $n \ge 7$ case, we rewrite $J_{3,i}$ as:
\[ J_{3,i} = \int_M \Big( \left[  \triangle_g + (h-2 f u) \right]  Z_{i,k, t,p} - 2 f(y_k) W_{k,t,p} Z_{i,k,t,p}\Big) \phi_k(t,p) dv_g.\]
Compared to the $n \ge 7$ case, the estimation of $J_{3,i}$ requires the additional information on $\phi_k(t,p)$ given by Proposition \ref{proploinmieux}. Namely, using \eqref{falloffphik} for a suitable choice of a sequence $(R_k)_k$ yields:
\[ \int_M \left(  \dk^2 r_k^{-6} \mathds{1}_{r_k \le d_{g_{y_k}} \le 2 r_k } +  \dk^{2} \mathds{1}_{nlcf})\right) \phi_k(t,p) dv_g = o(\dk^3), \]
while using \eqref{defhintro} and \eqref{eqsystreduit} gives:
\[ \int_M \left( h - \frac15 S_g - 2 fu \right) Z_{0,k,t,p} \phi_k(t,p) dv_g = o(\dk^3). \]
If $(M,g)$ is not locally conformally flat there holds, for some $R_k$ satisfying $R_k \sqrt{\dk} = o(1)$:
\ben \label{J3mean}
 \bal
\int_{M }&  \frac15 \Lambda_{y_k} S_{g_{y_k}} Z_{0,k,t,p} \phi_k(t,p)  dv_g \\
& = \int_{M \backslash B_{y_k}(R_k \sqrt{\dk})} \frac{1}{5}  \Lambda_{y_k}S_{g_{y_k}} Z_{0,k,t,p} \phi_k(t,p)  dv_g  + \int_{ B_{y_k}(R_k \sqrt{\dk}) \backslash B_{y_k}(\sqrt{\dk})} \frac15  \Lambda_{y_k}S_{g_{y_k}} Z_{0,k,t,p} \phi_k(t,p)  dv_g\\ 
& + \int_{B_{y_k}( \sqrt{\dk})}  \frac{1}{5}  \Lambda_{y_k}S_{g_{y_k}} Z_{0,k,t,p} \phi_k(t,p)  dv_g \\
& = o(\dk^3), 
\eal
\een
where we used again \eqref{propLambdaSg} and we estimated the first integral using \eqref{falloffphik}, the second one using \eqref{eqsystreduit} and the third one using Proposition \ref{proplocalmieux}. With \eqref{lapZik} we therefore obtain \eqref{J40}.

\medskip
\noindent Let now $1 \le i \le 6$. Similarly, using \eqref{falloffphik} for a suitable radius $R_k$ and using \eqref{relationrkmk} yields:
\[ \int_M \dk^3 r_k^{-7} |\phi_k(t,\xi)| dv_g = o(\dk^4),\]
and as in \eqref{J3mean} \eqref{falloffphik}, \eqref{eqsystreduit}, Proposition \ref{proplocalmieux} and \eqref{propLambdaSg} show that if $(M,g)$ is not locally conformally flat, then:
\[ \int_M  \frac15 \Lambda_{y_k} S_{g_{y_k}} Z_{i,k,t,p} \phi_k(t,p) dv_g = o (\dk^4). \]
Independently, there holds that:
\[ \bal
 \int_M & \big(h- \frac15 S_g - 2 f u \big) Z_{i,k,t,p} \phi_k(t,p) dv_g \\
  =& \int_{B_{y_k}(\sqrt{\dk})}  \big(  h- \frac15 S_g - 2 f u \big) Z_{i,k,t,p} \phi_k(t,p) dv_g \\
 &+ \int_{M \backslash B_{y_k}(\sqrt{\dk})}  \big( h- \frac15 S_g - 2 f u \big) Z_{i,k,t,p} \phi_k(t,p) dv_g \\
=&  o(\dk^4),
 \eal \]
 where we used \eqref{defhintro} and Proposition \ref{proplocalmieux} to estimate the first integral and \eqref{defhintro} and \eqref{eqsystreduit} to estimate the second one. Finally, in case $(M,g)$ is not locally conformally flat, mimicking the computations that led to \eqref{J3mean} we get that:
  \[ \int_M  \dk^3 \tk(\cdot)^{-4} |\phi_k(t,p)| dv_g = o(\dk^4). \]
Combining the above estimates with \eqref{lapZik} gives in the end \eqref{J4i}. 
\end{proof}

\bigskip

\noindent Here again, $J_{4,i}$ is easily estimated : straightforward computations using \eqref{deffintro} give indeed that, for any $0 \le i \le 6$:
\ben \label{autreJ4i}
J_{4,i} = o(\dk^4).
\een

\noindent As already mentioned, and unlike in the higher-dimensional case, the coupling field $X$ enters the expansion of $J_{5,i}$ for $n=6$:
\begin{claim}
There holds:
\ben \label{J50final}
 J_{5,0} = \kappa \dk^3 + o(\dk^3),
 \een
 where the constant $\kappa$ is explicitly given by \eqref{defkappa} below, and, for any $1 \le i \le 6$:
 \ben \label{J5i}
J_{5,i} = O(\dk^{\frac72}).
\een
\end{claim}
\begin{proof}
 We again write that there holds, for any $0 \le i \le 6$:
\ben \label{J51}
\bal 
J_{5,i} & = \int_M  \left( |\Lg T + \sigma|_g^2 + \pi^2 \right) \left( u^{-4} - \big(u + W_{k,t,p} + \phi_k(t,p) \big)^{-4} \right) Z_{i,k,t,p} dv_g  \\
& + \int_M \big(u + W_{k,t,p} + \phi_k(t,p) \big)^{-4}\left(  |\Lg T + \sigma|_g^2 - |\Lg T_{k,t,p} + \sigma |_g^2 \right) Z_{i,k,t,p} dv_g \\
& := J_{5,i}^1 + J_{5,i}^2. \\
\eal
\een
Let $R > 0$.
Since there holds
\[  \Big| u^{-4} - \big(u + W_{k,t,p} + \phi_k(t,p) \big)^{-4} \Big| \lesssim \left( W_{k,t,p} + |\phi_k(t,p)| \right) \textrm{ in } M \backslash B_{y_k}(R \sqrt{\dk}), \]
then \eqref{falloffphik} shows that:
\ben \label{J52}
 \left| \int_{M \backslash B_{y_k}(R \sqrt{\dk})}  \left( |\Lg T + \sigma|_g^2 + \pi^2 \right) \left( u^{-4} - \big(u + W_{k,t,p} + \phi_k(t,p) \big)^{-4} \right) Z_{0,k,t,p} dv_g \right| \lesssim \frac{\dk^3}{R^2} + o(\dk^3). 
 \een
Independently, by \eqref{eqsystreduit}, Lebesgue's dominated convergence theorem shows that there holds:
\ben \label{J53}
\bal
\dk^{-3} & \int_{B_{y_k}(R \sqrt{\dk})}  \left( |\Lg T + \sigma|_g^2 + \pi^2 \right) \left( u^{-4} - \big(u + W_{k,t,p} + \phi_k(t,p) \big)^{-4} \right) Z_{0,k,t,p} dv_g \\
& =\frac{(24)^2}{f(\xi_0)}\big( |\Lg T + \sigma|_g^2 + \pi^2 \big)(\xi_0) \int_{B_0(R)} \Bigg[ u(\xi_0)^{-4} - \left( u(\xi_0) + \left( \frac{24}{f(\xi_0)} \right)^2 |y|^{-4} \right)^4 \Bigg] |y|^{-4} dy + o(1), \\
\eal
\een
as $k \to + \infty$, so that \eqref{J52} and \eqref{J53} together show that there holds
\ben \label{J54}
\bal
J_{5,0}^1 & = \frac{(24)^2}{f(\xi_0)}\big( |\Lg T + \sigma|_g^2 + \pi^2 \big)(\xi_0) \\
& \times \int_{\RR^6} \Bigg[ u(\xi_0)^{-4} - \left( u(\xi_0) + \left( \frac{24}{f(\xi_0)} \right)^2 |y|^{-4} \right)^4 \Bigg] |y|^{-4} dy \cdot \dk^3 + o(\dk^3).
\eal
\een

\noindent Using \eqref{estLTkLT0} with \eqref{choixvek} one gets that:
\[ |\Lg T + \sigma|_g^2 - |\Lg T_{k,t,\xi} + \sigma |_g^2 = - |\Lg \Theta_k|_g^2 + O\big( |\Lg \Theta_k|_g \big) + O(\dk),\]
where $\Theta_k$ is defined in \eqref{defchampsapp} below. Using the asymptotic \eqref{asymptoThetak} together with \eqref{defXintro} and the dominated convergence theorem show that:
\ben \label{J55}
\bal
 & J_{5,0}^2 = - C(6) f(\xi_0)^{-8} \alpha^2  \\
 & \times \int_{\RR^6} \left( u(\xi_0) + \left( \frac{f(\xi_0)}{24}\right)^{-2} |y|^{-4}\right)^{-4} \left(2 + 28 \left| \left \langle \frac{Z(0)}{| Z(0)|_{eucl}}, \frac{y}{|y|} \right \rangle \right|^2 \right) |y|^{-14} dy \cdot \dk^3,
\eal
 \een
 where we have let $C(6) = 2^{32} 3^9 5^{-4}$ and where $\alpha$ is as in \eqref{defalpha}. Combining \eqref{J54} and \eqref{J55} in \eqref{J51} one obtains in the end \eqref{J50final}, where $\kappa$ is given by:
  \ben \label{defkappa}
 \bal
 \kappa &= \frac{(24)^2}{f(\xi_0)}\big( |\Lg T + \sigma|_g^2 + \pi^2 \big)(\xi_0)  \times \int_{\RR^6} \Bigg[ u(\xi_0)^{-4} - \left( u(\xi_0) + \left( \frac{24}{f(\xi_0)} \right)^2 |y|^{-4} \right)^4 \Bigg] |y|^{-4} dy \\
 &- C(6) f(\xi_0)^{-8} \alpha^2   \times \int_{\RR^6} \left( u(\xi_0) + \left( \frac{f(\xi_0)}{24}\right)^{-2} |y|^{-4}\right)^{-4} \left(2 + 28 \left| \left \langle \frac{Z(0)}{| Z(0)|_{eucl}}, \frac{y}{|y|} \right \rangle \right|^2 \right) |y|^{-14} dy.
 \eal
 \een
 In particular, up to choosing $\alpha$ as in \eqref{defalpha} small enough, we have $\kappa > 0$.

\medskip

\noindent Let now $1 \le i \le 6$. Mimicking the proof of \eqref{J52} and \eqref{J53} gives, by \eqref{falloffphik} and the dominated convergence theorem, that
\be 
J_{5,i}^1 = o( \dk^{\frac72}).
\ee
Finally, using the expansion \eqref{estLTkLT0} along with \eqref{asymptoThetak} shows, again by dominated convergence, that there holds:
\be
J_{5,i}^2 = O(\dk^{\frac72}).
\ee
With \eqref{J51} we obtain that \eqref{J5i} holds. 
\end{proof}

\section{Conclusive argument} \label{argumentconclusif}

\noindent In this section we conclude the proof of Theorem \ref{thprincipal}. We use the expansions of the $\lki(t,p)$ obtained in Section \ref{DL} to show that, for any $k$, there exist $(t_k,p_k) \in [1/D,D] \times \overline{B_0(1)}$ such that $\lki(t_k,p_k) = 0$ for any $0 \le i \le n$.  We use here the notations of Section \ref{DL}. In particular, $\dk$ and $y_k$ are defined by \eqref{defdkyk} and \eqref{definitionyk}.

\medskip

\noindent First, for any $0 \le i \le n$, there holds:
\ben \label{DLNZi}
 \int_M \left \langle \nabla Z_{i,k,t,p}, \nabla Z_{j,k,t,p} \right \rangle_g + h Z_{i,k,t,p} Z_{j,k,t,p} dv_g  = \delta_{ij} \left \| \nabla V_{i,y_k} \right\|_{L^2(\RR^n)}^2 + O(\dk), 
 \een
uniformly in the choice of $h, t$ and $p$, and where $V_{i,y_k}$ is defined in \eqref{defVki}.

\noindent Assume first that $(M,g)$ is locally conformally flat or that $7 \le n \le 10$. Then, combining \eqref{I11}, \eqref{I12}, \eqref{I22bis}, \eqref{I23bis}, \eqref{I24}, \eqref{I31}, \eqref{I32}, \eqref{I40}, \eqref{Int4}, \eqref{I50}, \eqref{Int5i}, \eqref{I6} and \eqref{I76} in \eqref{DL2} yields, with \eqref{DLNZi} and \eqref{propbetak}, that:
\ben \label{grosseeqlambda}
\bal
& \Big(I_{n+1} + O(\dk) \Big) \left( 
\bal
\lambda_k^0&(t,p) \\
&\vdots \\
\lambda_k^n&(t,p)
\eal
\right) \\
& =  \left( 
\bal 
& \bal 
 \mk^{\pui} 
\Bigg[  
\frac{8(n-1)}{(n-2)(n-4)} K_n^{-n} f(\xi_0)^{- \frac{n}{2}} H(p) t^2 - \frac{1}{2} (n-2)^2 \left( n(n-2)\right)^{\pui} \omega_{n-1} f(\xi_0)^{1 - \frac{n}{2}} u(\xi_0) t^{\pui} \\
-  \frac{10}{9} f(\xi_0)^{-4} K_{10}^{-10} |W_g(\xi)|_g^2 t^4 \mathds{1}_{n=10} + R_k^0(t,p)
\Bigg] \eal
 \\
& \frac{\mk^{\frac{n}{2}}}{\beta_k} \Bigg[  \frac{2n(n-1)}{n-4} K_n^{-n} f(\xi_0)^{- \frac{n}{2}} \nabla_i H(p) t^3 + R_k^i(t,p) \Bigg]
\eal
\right),
\eal
\een
where, for $0 \le i \le n$, $R_k^{i}(t,p)$ denotes a function which converges to zero in $C^0 \big([1/D, D] \times \overline{B_0(1)} \big)$ as $k \to +\infty$, where $K_n^{-n}$ is as in \eqref{defKnn} and where $\beta_k$ is as in \eqref{propbetak}. The continuity of $R_k^{i}$, $0 \le i \le n$, is a direct consequence of the continuity of $\phi_k$ as stated in Proposition \ref{propsystreduit}. Let $F$ be the function defined in $[1/D, D] \times \overline{B_0(1)}$ by:
\[ F(t,p) = \left(
\bal
&\bal
 \frac{8(n-1)}{(n-2)(n-4)} K_n^{-n} f(\xi_0)^{- \frac{n}{2}} H(p) t^2 - \frac{1}{2} (n-2)^2 \left( n(n-2)\right)^{\pui} \omega_{n-1} f(\xi_0)^{1 - \frac{n}{2}} u(\xi_0) t^{\pui} \\
-  \frac{10}{9} f(\xi_0)^{- 4 } K_{10}^{-10} |W_g(\xi)|_g^2 t^4 \mathds{1}_{n=10} 
\eal
 \\
& \frac{2n(n-1)}{n-4} K_n^{-n} f(\xi_0)^{- \frac{n}{2}} \nabla_i H(p) t^3
\eal
\right) 
\]
and let $t_0 > 0$ be the unique solution of:
\[ \bal
\frac{8(n-1)}{(n-2)(n-4)} K_n^{-n} f(\xi_0)^{- \frac{n}{2}} t^2 = \frac{1}{2} (n-2)^2 \left( n(n-2)\right)^{\pui} \omega_{n-1} f(\xi_0)^{1 - \frac{n}{2}} u(\xi_0) t^{\pui} \\
+  \frac{10}{9} f(\xi_0)^{- 4} K_{10}^{-10} |W_g(\xi)|_g^2 t^4 \mathds{1}_{n=10}.
\eal
\]
Using the assumption that $0$ is a non-degenerate critical point of $H$, it is easily checked that the differential of $F$ at $(t_0,0)$ is invertible. Since there holds $F(t_0,0) = 0$ by definition of $t_0$, and since the $R_k^{i}(t,p)$ appearing in \eqref{grosseeqlambda} uniformly converge to $0$ as $k \to + \infty$, standard degree-theoretic arguments yield the existence of a sequence $(t_k,p_k) \in ]1/D, D[ \times B_0(1)$ of interior points such that
\[ F(t_k,p_k) + \left( 
\bal
& R_k^0(t_k,p_k) \\
 \Big( & R_k^i(t_k,p_k) \Big)_{1 \le i \le n} \\
\eal
\right) = 0
\]
for any $k$. Note that throughout this argument we assumed that $D$ is chosen large enough to have $t_0 \in [2/D, D/2]$. Coming back to \eqref{grosseeqlambda}, this amounts to say that $\lki(t_k,p_k) = 0$ for any $0 \le i \le n$. And with Proposition \ref{propsystreduit} and \eqref{systnoyaupres} this shows that the function $u_{k,t_k,p_k}$ is a solution of system \eqref{intro1}, and concludes the proof of Theorem \ref{thprincipal} in this case.

\medskip

\noindent If now we assume that $n \ge 11$ and $(M,g)$ is not locally conformally flat, the same arguments lead to the following expansion for the $\lki(t,p)$:
\be 
\bal
& \Big(I_{n+1} + O(\dk) \Big) \left( 
\bal
\lambda_k^0&(t,p) \\
&\vdots \\
\lambda_k^n&(t,p)
\eal
\right) \\
& = K_n^{-n} f(\xi_0)^{- \frac{n}{2}} \left( 
\bal 
& \bal 
 \mk^{4} 
\Bigg[  
\frac{8(n-1)}{(n-2)(n-4)} H(p) t^2 - \frac{1}{3} \frac{n(n-2)}{(n-4)(n-6)} f(\xi_0)  |W_g(\xi_0)|_g^2 t^4  + R_k^0(t,p)
\Bigg] \eal
 \\
& \frac{\mk^5}{\beta_k} \Bigg[  \frac{2n(n-1)}{n-4} \nabla_i H(p) t^3 + R_k^i(t,p) \Bigg]
\eal
\right).
\eal
\ee
\noindent While in the $6$-dimensional case, we end up with:
\be 
\bal
& \Big(I_{7} + O(\dk) \Big) \left( 
\bal
\lambda_k^0&(t,p) \\
&\vdots \\
\lambda_k^6&(t,p)
\eal
\right) =  \left( 
\bal 
& \bal 
 \mk^{3} 
\Bigg[  
-5 K_6^{-6} f(\xi_0)^{-3} H(p) t^2 + \kappa t^3  + R_k^0(t,p)
\Bigg] \eal
 \\
& \frac{\mk^4}{\beta_k} \Bigg[  - 30 K_6^{-6} f(\xi_0)^{-3} \nabla_i H(p) t^3 + R_k^i(t,p) \Bigg]
\eal
\right),
\eal
\ee
where the constant $\kappa$ is defined in \eqref{defkappa}. The conclusion in these cases follows then from the exact same arguments than in the previous case, thus concluding the proof of Theorem \ref{thprincipal}.

\section{Technical results} \label{technicalresults}

\subsection{Conformal laplacian of  $W_{k,t,\xi}$ and of the $Z_{i,k,t,\xi}$, $0 \le i \le n$.}

Below are given the expressions of the conformal laplacian of the functions $W_k$ and $Z_{i,k,t,\xi}$ defined in \eqref{bulle} and \eqref{defZk}. Let $D > 0$ and let $(t_k,\xi_k)_k$ be a sequence of points in $ [1/D, D] \times M$. Then, for any function $h \in C^\infty(M)$, there holds:
\ben \label{lapbulle}
\bal
\left( \triangle_g + h \right) W_{k,t_k,\xi_k} = f(\xi_k) W_{k,t_k,\xi_k}^{2^*-1} + (h - c_n S_g) W_{k,t_k,\xi_k} + c_n \Lambda_{\xi_k}^{2^*-2} S_{g_{\xi_k}} W_{k,t_k,\xi_k} \\
+ O(\dk^{\pui} \mathds{1}_{d_k \le 2 r_k}) + O(\dk^{\pui} r_k^{-n} \mathds{1}_{r_k \le d_k \le 2 r_k}), \,
\eal
\een
and, for any $1 \le i \le n$:
\ben \label{lapZik}
\bal
\left(  \triangle_g + h \right)& Z_{0,k,t_k,\xi_k} = (2^*-1) f(\xi_k) W_{k,t_k, \xi_k}^{2^*-2} Z_{0,k,t_k,\xi_k} + (h - c_n S_g) Z_{0,k,t_k,\xi_k} \\
& + c_n \Lambda_{\xi_k}^{2^*-2} S_{g_\xi} Z_{0,k,t_k,\xi_k} + O\big( \dk^\pui r_k^{-n} \mathds{1}_{r_k \le d_{g_{\xi_k}} \le 2 r_k } \big) + O(  \dk^{\pui} \mathds{1}_{nlcf, d_{g_{\xi_k}} \le 2 r_k}), \\
 \left(  \triangle_g + h \right)& Z_{i,k, t_k,\xi_k} =  (2^*-1) f(\xi_k) W_{k,t_k,\xi_k}^{2^*-2} Z_{i,k,t_k,\xi_k} + (h - c_n S_g) Z_{i,k,t_k,\xi_k} \\
 & + c_n \Lambda_{\xi_k}^{2^*-2} S_{g_{\xi_k}} Z_{i,k,t_k, \xi_k}  + O\big( \dk^{\frac{n}{2}} r_k^{-n-1} \mathds{1}_{r_k \le d_{g_{\xi_k}} \le 2 r_k } \big) + O \big( \dk^{\frac{n}{2}} \tk(\cdot)^{2-n} \mathds{1}_{nlcf,d_{g_{\xi_k}} \le 2 r_k}\big) . \\
\eal
\een
In \eqref{lapbulle} and \eqref{lapZik} $\Lambda_{\xi_k}$ is as in \eqref{propLambda} and $d_k = d_{g_{\xi_k}}(\xi_k, \cdot)$. The notation $``O(f)''$ denotes a smooth function which can be uniformly bounded in $C^0(M)$ by $C_0 | f| $, where $C_0$ is some positive constant that does not depend on $k$. Also, the notational shorthand $\mathds{1}_{nlcf}$ is used to indicate that the corresponding term vanishes if $(M,g)$ is locally conformally flat in a fixed neighbourhood of $\xi_k$. The notations $\mathds{1}_{r_k \le d_{g_{\xi_k}}} \le 2 r_k$ and $\mathds{1}_{nlcf, d_{g_{\xi_k}} \le 2 r_k}$ are defined similarly.

\medskip
\noindent These expressions are obtained from \eqref{bulle} and \eqref{defZk}, using the conformal invariance property of the conformal laplacian and the properties of the normal conformal factor $\Lambda_{\xi_k}$. 
See for instance Esposito-Pistoia-V\'etois \cite{EspositoPistoiaVetois} and Lee-Parker \cite{LeeParker} for more details.

\subsection{Pointwise estimates for solutions of the $1$-form equation.}

Let $(M,g)$ be a closed Riemannian manifold possessing no conformal Killing fields. This means that the operator $\Dg$, when acting on $H^1$ fields of $1$-forms, has zero kernel. We start by recalling that the operator $\Dg$ always possesses local Green fields:

\begin{prop} \label{faitGreen}
Let $x_0\in M$ and $\delta < \frac{i_g(M)}{2}$. For any $x\in B_{x_0}(\delta)$ there exist $n$ fields of $1$-forms $G_1(x,\cdot), \cdots, G_n(x,\cdot)$ defined in $B_{x_0}(\delta) \backslash \{x\}$ which form a fundamental solution for the operator $\Dg$ in $B_{x_0}(\delta)$ in the following sense: for any $1$-form $X \in \Gamma(T B_{x_0}(\delta))$ with $X \equiv 0$ on $\partial B_{x_0}(\delta)$, there holds:
\ben \label{GreenvectDg}
 \Big(X - \pi(X) \Big)_i(x) = \int_{B_{x_0}(\delta)} \left \langle G_i(x,y), \Dg X(y) \right \rangle_{g(y)} dv_g(y),
\een
where $\pi$ denotes the orthogonal projection for the $L^2(B_{x_0}(\delta))$-scalar product on the set
\[ \left\{ X \in H^1 \big( B_{x_0}(\delta)\big), \Lg X = 0 \textrm{ in } B_{x_0}(\delta) \right \} ,\]
and where the coordinates in \eqref{GreenvectDg} are taken in any chart in $B_{x_0}(\delta)$. The Green fields satisfy in addition: for any $1 \le i \le n$, for any $x \in B_{x_0}(\delta/2), y \in B_{x_0}(\delta)$,
\ben \label{propGreenvect}
d_g(x,y) \left| \nabla G_i (x,y) \right|_g + \left| G_i(x,y) \right|_g \le C(n,g) d_g(x,y)^{2-n},
\een
where the derivative in \eqref{propGreenvect} can be taken with respect to $x$ or to $y$.
\end{prop}
\noindent The construction of these Green fields is carried out in the Euclidean case in Druet-Premoselli \cite{DruetPremoselli} in dimension $3$ and in Premoselli \cite{Premoselli4}. The extension to the Riemannian case shows no additional difficulty. The proof just consists in adapting the steps that lead to the construction of the Riemannian Green function for the Laplace-Beltrami operator on a Riemannian manifold starting from the euclidean one (see for instance Appendix A in Druet-Hebey-Robert \cite{DruetHebeyRobert} or Robert \cite{RobDirichlet}), and we will not detail it here. Using the Euclidean expression computed in Premoselli \cite{Premoselli4} one finds that these Green fields $G_i$ have the following expansion:
\ben \label{expansionGi}
\bal
 G_i\big(x, \exp_x(y)\big)_j =&  - \frac{1}{4(n-1) \omega_{n-1}} |y|^{2-n} \left((3n-2) \delta_{ij} + (n-2) \frac{y_i y_j}{|y|^2} \right) \Big(1 + O\big(|y|\big) \Big),
 \eal
 \een
for any $x \in M$, where the exponential map is taken here for the background metric $g$ and $\omega_{n-1}$ is the volume of the standard $(n-1)$-sphere in $\RR^n$. If $X$ is a smooth $1$-form in $B_{x_0}(\delta)$ which vanishes on $\partial B_{x_0}(\delta)$ formula \eqref{GreenvectDg} can be differentiated to obtain: for any $1 \le i, j \le n$, for any $x \in B_{x_0}(\delta)$,
\ben \label{LgGreenvectDg}
\Lg X_{ij}(x) = \int_{B_{x_0}(\delta)} \left \langle H_{ij}(x,y), \Dg X(y) \right \rangle_{g(y)} dv_g(y),
\een
where we have let:
\ben \label{defHij}
H_{ij}(x) = \nabla_i G_j(x,y) + \nabla_j G_i(x,y) - \frac{2}{n} g^{kl}(x) \nabla_k G_l(x,y) g_{ij}(x),
\een
and the covariant derivatives are all taken here with respect to $x$.

\medskip
\noindent We now use these Green fields to derive optimal pointwise estimates on solutions of the $1$-form equation. Let $u$ be a smooth positive function in $M$. Let $D >0$ and let $(t_k,\xi_k)_k \in [\frac{1}{D} , D] \times M$ be a sequence of points, and consider the function $W_{k,t_k,\xi_k}$ given by \eqref{bulle}, where $\dk$ is given by \eqref{defdkyk} and $(\mk)_k$ denotes some sequence of positive numbers which converge to zero. Recall that the functions $W_{k,t_k,\xi_k}$ -- that shall now be abbreviated as $W_k$ -- are compactly supported on a ball centered at $\xi_k$ and of radius $2r_k$ given by \eqref{relationrkmk} (the ball is taken here with respect to the metric $g_{\xi_k}$). Let $(\ve_k)_k$ denote a sequence of positive numbers and, for any $k$, let $v_k$ be a continuous function satisfying
\ben \label{hypov}
\left \Vert \frac{v}{u + W_{k}} \right \Vert_{C^0(M)} \le \ve_k.
\een
Let $X$ and $Y$ be two $1$-forms in $M$ and, for any $k$, let $T_k, \Theta_k$ and $T$ be the unique solutions of the following equations in $M$:
\ben \label{defchampsapp}
 \bal
 &\Dg T_k  = \left( u + W_{k} + v_k \right)^{2^*} X + Y, \\
 &\Dg \Theta_k  = W_{k}^{2^*}X, \\
 &\Dg T  = u^{2^*}X + Y. 
 \eal
 \een
Then the following pointwise asymptotic estimate on $\Lg T_k$ holds:
\begin{prop} \label{controleLTkdessus}
For any sequence $(x_k)_k$ of points in $M$, there holds:
\ben \label{estLTkLT0}
 \Lg T_k(x_k) = \Big(1 + O(\dk) \Big) \Lg \Theta_k(x_k) + \Lg T(x_k)  + O\Big( \Vert X \Vert_{C^0(M)} \vek  + |X(\xi_k)|_g + \dk \Vert \nabla X \Vert_{L^\infty(2 r_k)} \Big), 
 \een
uniformly in $k$. As a consequence, there holds, for any $x \in M$:
\ben \label{estLTkLT}
|\Lg T_k - \Lg T|_g(x) \le C \Bigg( \Big[ |X(\xi_k)|_g + \dk \Vert \nabla X \Vert_{L^\infty(2 r_k)} \Big] \tk(x)^{1-n}  + \Vert X \Vert_{C^0(M)} \vek   \Bigg) ,
\een
where $C$ is a positive constant independent of $k$ and $\Vert X \Vert_{C^0(M)}$ and where $\tk(x)$ is as in \eqref{defthetak}. 
\end{prop}

\begin{proof}
Let $G_i, 1 \le i \le n$, be the family of Green fields given by Proposition \ref{faitGreen} and defined in $B_{\xi_k}(\delta) \times B_{\xi_k}(\delta) \backslash \{x = y\}$. Let $\eta \in C_c^\infty(\RR^+)$ be a smooth cut-off function equal to $1$ in $[0, \delta/4]$ and equal to $0$ outside of $[0, \delta/2]$. We define the following $1$-forms, compactly supported in $B_{\xi_k}(\delta / 2)$:
\ben \label{defPkaux} 
P_k(x)_i = \left( \int_{B_{x_0}\left(\frac{\delta}{2} \right)} \left \langle G_i(x,y),  \left[ \left( u + W_{k} + v_k \right)^{2^*}  - u^{2^*} \right](y) X(y) \right \rangle_{g(y)} dv_g(y) \right) \cdot \eta \Big( d_{g}(\xi_k,x) \Big).
\een
By definition of $G_i$, by the choice of $\eta$, by \eqref{propGreenvect}, \eqref{hypov} and by Giraud's lemma there holds, for any $x \in B_{\xi_k}(\delta/2)$:
\ben \label{estC1Pk}
\bal
 & |P_k|(x)  \le C \Big( |X(\xi_k)|_g + \dk \Vert \nabla X \Vert_{L^\infty(2 r_k)} \Big) \tk(x)^{2-n} + C \Vert X \Vert_{C^0(M)} \vek \\
&  |\Lg P_k|_g(x) \le C \Big( |X(\xi_k)|_g + \dk \Vert \nabla X \Vert_{L^\infty(2 r_k)} \Big) \tk(x)^{1-n} + C \Vert X \Vert_{C^0(M)} \vek. \\
 \eal \een
In \eqref{estC1Pk} and until the end of this section $C$ will denote some positive constant that does not depend on $k$ or on $\Vert X \Vert_{C^0(M)}$. Since $\left| \left( u + W_{k} + v_k \right)^{2^*}  - u^{2^*} \right| \le C \vek$ in $M \backslash B_{\xi_k}(\delta/4)$, and by definition of $G_i$, one obtains with \eqref{defchampsapp}, \eqref{defPkaux} and \eqref{estC1Pk} that: 
\[ \bal
\Dg \Big( T_k - P_k - T \Big) & = 0 \textrm{ in } B_{\xi_k}(\delta/4) \\ 
\Big|\Dg \Big( T_k - P_k - T \Big) \Big| & \le C \left(   |X(\xi_k)|_g + \dk \Vert \nabla X \Vert_{L^\infty(2 r_k)}  +  \Vert X \Vert_{C^0(M)} \vek \right) \textrm{ in } B_{\xi_k}(\delta/2) \backslash B_{\xi_k}(\delta/4) \\
\Big| \Dg \Big( T_k - P_k - T \Big) \Big| & \le  C \Vert X \Vert_{C^0(M)} \vek \textrm{ in } M \backslash B_{\xi_k}(\delta/2).
\eal \]
Standard elliptic regularity theory applies for $\Dg$ (see e.g. Premoselli \cite{Premoselli4}), and with the latter estimates shows that there holds, in $C^0(M)$:
\ben \label{expansion1Tk}
\Lg T_k = \Lg P_k + \Lg T + O \Bigg(  |X(\xi_k)|_g + \dk \Vert \nabla X \Vert_{L^\infty(2 r_k)}  +  \Vert X \Vert_{C^0(M)} \vek \Bigg) .
\een
Independently, similar arguments show that there holds:
\ben \label{thetakexpression}
 \Lg \Theta_k = \Lg Q_k + O \left(  |X(\xi_k)|_g + \dk \Vert \nabla X \Vert_{L^\infty(2 r_k)} \right), 
 \een
where $\Theta_k$ is defined in \eqref{defchampsapp} and where we have let, for $1 \le i \le n$:
\ben \label{expThetak}
 Q_k(x)_i =  \left( \int_{B_{x_0}\left(\frac{\delta}{2} \right)} \left \langle G_i(x,y),  W_{k}^{2^*}(y) X(y) \right \rangle_{g(y)} dv_g(y) \right)  \cdot \eta \Big( d_{g}(\xi_k,x) \Big).
 \een 
Since there holds, in $ B_{\xi_k}(\delta/4)$, that
\[ \left| \left( u + W_{k} + v_k \right)^{2^*}  - u^{2^*} - W_k^{2^*} \right| \le C  \Big( \vek + W_k^{2^*-1} + W_k \Big),\]
we obtain with Giraud's lemma, \eqref{defPkaux}, \eqref{thetakexpression} and \eqref{expThetak} that, for any sequence $(x_k)_k$ in $M$:
\[ \Lg P_k(x_k) = \Big( 1 + O(\dk) \Big) \Lg \Theta_k(x_k) + O \Bigg(  |X(\xi_k)|_g + \dk \Vert \nabla X \Vert_{L^\infty(2 r_k)}  +  \Vert X \Vert_{C^0(M)} \vek \Bigg) .\]
Then estimate \eqref{estLTkLT0} then follows from the latter expansion and \eqref{expansion1Tk}, and \eqref{estLTkLT} follows from \eqref{estLTkLT0}. 
\end{proof}

\noindent Using \eqref{expansionGi}, \eqref{LgGreenvectDg}, \eqref{defHij}, \eqref{thetakexpression} and \eqref{expThetak} one obtains the following asymptotic estimate for $\Lg \Theta_k$: for any sequence $(x_k)_k$ of points in $M$ satisfying $\dk <<d_g(x_k, \xi_k) \le \delta/4$ there holds: 
\ben \label{asymptoThetak}
\bal
 \Lx \Theta_k (x_k)_{ij}  = &K(n)f(\xi_k)^{- \frac{n}{2}}  \Bigg[ \delta_{ij} \zeta^p \check{x}_p - \zeta_i \check{x}_j - \zeta_j \check{x}_i - (n-2) \zeta^p \check{x}_p \check{x}_i \check{x}_j \Bigg]  \\
&  \times \Big( |X(\xi_k)|_g + O(\dk  \Vert \nabla X \Vert_{L^\infty(2 r_k)}) \Big) d_g(\xi_k,x_k)^{1 - n}  \big(1 + o(1) \big), \\ 
\eal
\een
where we have let, up to a subsequence:
\[ \check{x} = \underset{k \to \infty}{\lim} \frac{1}{d_g(\xi_k,x_k)} \exp_{\xi_k}^{-1}(x_k), \quad  K(n) = \frac{n^{\frac{n+2}{2}} (n-2)^{\frac{n}{2}} \omega_n}{2^{n+1} (n-1) \omega_{n-1}} \textrm{ and } \zeta = \underset{k \to \infty }{ \lim} \frac{ X(\xi_k)}{ |X(\xi_k)|_g} .\]

\bigskip

\noindent We conclude this section with a refinement of Proposition \ref{controleLTkdessus} which takes into account the behavior of $v_k$ ``far away'' from the concentration point $\xi_k$. 
\begin{prop}
Let $T_k$, $\Theta_k$ and $T$ be as in \eqref{defchampsapp}. Let $(R_k)_k$ be a sequence of positive numbers such that $1 \le R_k = o(\dk^{-\frac12})$ as $k \to + \infty$. There holds, for any $x \in M$:
\ben \label{estLTkLTordre2}
\bal
\Big| \Lg T_k  - \Lg T \Big|_g(x) \le C \Bigg( \Big[ |X(\xi_k)|_g + \dk \Vert \nabla X \Vert_{L^\infty(2 r_k)} \Big] \tk(x)^{1-n} +   \dk^{\frac{n}{2}} \tk(x)^{1-n} \\
+  \Vert X \Vert_{C^0(M)} \vek \dk R_k^2 +  \Vert X \Vert_{C^0(M)} \left \Vert v_k \right \Vert_{L^\infty (M \backslash B_{\xi_k}(R_k \sqrt{\dk}) )}   \Bigg)  ,
\eal
\een
where $C$ is some positive constant that does not depend on $k$, on $\Vert X \Vert_{C^0(M)}$ and on the choice of $(R_k)_k$.
\end{prop}

\begin{proof}
Let again $G_i$, $1 \le i \le n$, be the Green fields for $\Dg$ satisfying \eqref{GreenvectDg} in $B_{\xi_k}(\delta)$. Mimicking what we did in \eqref{defPkaux}, we define:
\ben \label{defQ1Q2} 
\bal
Q^1_k(x)_i &= \left( \int_{B_{x_0}\left(\frac{\delta}{2} \right)} \left \langle G_i(x,y),  \left[ \left( u + W_{k} + v_k \right)^{2^*}  - \left( u + v_k \right) ^{2^*} \right](y) X(y) \right \rangle_{g(y)} dv_g(y) \right) \times \eta \Big( d_{g}(\xi_k,x) \Big), \\
 Q^2_k(x)_i &= \left( \int_{B_{x_0}\left(\frac{\delta}{2} \right)} \left \langle G_i(x,y),  \left[ \left( u + v_k \right)^{2^*}  - u^{2^*} \right](y) X(y) \right \rangle_{g(y)} dv_g(y) \right) \times \eta \Big( d_{g}(\xi_k,x) \Big). \\
\eal
\een
The term $\left( u + W_{k} + v_k \right)^{2^*}  - \left( u + v_k \right) ^{2^*}$ is compactly supported in $B_{\xi_k}(2 r_k)$ and satisfies there:
\[ \left| \left( u + W_{k} + v_k \right)^{2^*}  - \left( u + v_k \right) ^{2^*} - W_k^{2^*} \right| \le C \left( W_k^{2^*-1} + W_k \right).\]
Therefore, \eqref{propGreenvect} and Giraud's lemma show that there holds, for any $x \in M$:
\ben \label{estQ1}
\big| Q^1_k \big|_g +  \tk(x) \big| \Lg Q^1_k \big|_g \le \Big[ |X(\xi_k)|_g + \dk \Vert \nabla X \Vert_{L^\infty(2 r_k)} \Big] \tk(x)^{2-n}. 
 \een
We thus obtain that there holds:
\ben  \label{lapQ1} 
\Dg \left( T_k - Q^1_k - T \right) = \left \{
\bal
&  \left[ \left( u + v_k \right)^{2^*}  - u^{2^*} \right]X & \textrm{ in } B_{\xi_k}(\frac{\delta}{4}), \\
&  \left[ \left( u + v_k \right)^{2^*}  - u^{2^*} \right]X + O \Big( |X(\xi_k)|_g + \dk \Vert \nabla X \Vert_{L^\infty(2 r_k)} \Big) & \textrm{ in } M\backslash B_{\xi_k}(\frac{\delta}{4}).
\eal
\right.
\een
We now estimate $Q^2_k$ in \eqref{defQ1Q2}. With \eqref{hypov} we write that there holds, for some positive $C$:
\[ 
\Big| \left( u + v_k \right)^{2^*}  - u^{2^*} \Big| \le C \times \left\{
\bal
&\vek^{2^*} W_k^{2^*} + 1 \textrm{ in } B_{\xi_k}(\sqrt{\dk}), \\
& \vek \textrm{ in } B_{\xi_k}(R_k \sqrt{\dk}) \backslash B_{\xi_k}(\sqrt{\dk}) ,\\
& \left \Vert v_k \right \Vert_{L^\infty (M \backslash B_{\xi_k}(R_k \sqrt{\dk}) )} \textrm{ in } M \backslash B_{\xi_k}(R_k \sqrt{\dk}). \\
\eal  \right. \]
Then \eqref{propGreenvect} and Giraud's lemma yield again, for any $ x \in M$:
\ben \label{estQ2}
\bal
 \big| Q^2_k \big|_g(x) + \tk(x) \big| \Lg Q^2_k \big|_g(x) \le C \Bigg( \ve_k^{2^*} \Big[ |X(\xi_k)|_g + \dk \Vert \nabla X \Vert_{L^\infty(2 r_k)} \Big] \tk(x)^{2-n} + \dk^{\frac{n}{2}} \tk(x)^{2-n} \\
+ \Vert X \Vert_{C^0(M)} \vek R_k^2 \dk +  \Vert X \Vert_{C^0(M)}\left \Vert v_k \right \Vert_{L^\infty (M \backslash B_{\xi_k}(R_k \sqrt{\dk}) )} \Bigg) ,
\eal
\een
so that there holds, with \eqref{lapQ1}:
\ben \label{lapQ2} 
\Dg \Big( T_k - Q^1_k - Q^2_k - T \Big) = \left \{
\bal 
&  0 \textrm{ in } B_{\xi_k}(\delta/4) \\
& O \Big(  |X(\xi_k)|_g + \dk \Vert \nabla X \Vert_{L^\infty(2 r_k)} +\dk^{\frac{n}{2}}  + \Vert X \Vert_{C^0(M)} \vek R_k^2 \dk \\
& +  \Vert X \Vert_{C^0(M)}\left \Vert v_k \right \Vert_{L^\infty (M \backslash B_{\xi_k}(R_k \sqrt{\dk}) )}  \Big)  
\textrm{ in } M  \backslash B_{\xi_k}(\delta/4) \\
\eal \right.
\een 
The Claim then follows from \eqref{lapQ2}, standard elliptic theory, \eqref{estQ1} and \eqref{estQ2}.

\end{proof}

\bibliographystyle{amsplain}
\bibliography{biblio}

\end{document}